\documentclass[reqno]{amsart} 
\usepackage{bm,amssymb,amsmath,verbatim,graphicx}

\newcommand{\bi}{\begin{itemize}}
\newcommand{\ei}{\end{itemize}}
\newcommand{\ben}{\begin{enumerate}}
\newcommand{\een}{\end{enumerate}}
\newcommand{\be}{\begin{equation}}
\newcommand{\ee}{\end{equation}}
\newcommand{\bea}{\begin{eqnarray}} 
\newcommand{\eea}{\end{eqnarray}}
\newcommand{\ba}{\begin{align}} 
\newcommand{\ea}{\end{align}}
\newcommand{\bse}{\begin{subequations}} 
\newcommand{\ese}{\end{subequations}}
\newcommand{\bc}{\begin{center}}
\newcommand{\ec}{\end{center}}
\newcommand{\bfi}{\begin{figure}}
\newcommand{\efi}{\end{figure}}
\newcommand{\ca}[2]{\caption{#1 \label{#2}}}
\newcommand{\ig}[2]{\includegraphics[#1]{#2}}
\newcommand{\brmk}{\begin{remark}}
\newcommand{\ermk}{\end{remark}}

\newcommand{\tbox}[1]{{\mbox{\rm \tiny #1}}}

\newcommand{\mbf}[1]{{\mathbf #1}}
\newcommand{\eps}{\varepsilon}
\newcommand{\half}{\mbox{\small $\frac{1}{2}$}}

\newcommand{\pO}{{\partial\Omega}}
\newcommand{\lo}{{L^2(\Omega)}}
\newcommand{\lpo}{{L^2(\pO)}}

\newcommand{\dn}[1]{{#1}_n}        
\newcommand{\dk}[1]{\dot{#1}}      
\newcommand{\ddk}[1]{\ddot{#1}}      
\newcommand{\matlab}{MATLAB}       
\newcommand{\mpspack}{{\tt MPSpack}}       
\newcommand{\xx}{x}
\newcommand{\yy}{y}
\newcommand{\zz}{z}    
\newcommand{\Sc}{{\mathcal S}}
\newcommand{\Dc}{{\mathcal D}}
\newcommand\dt[1]{\nabla_\tbox{tan}{#1}}   
\newcommand\VV{W}                    

\newcommand{\RR}{{\mathbb{R}}}
\newcommand{\CC}{{\mathbb{C}}}

\newcommand{\ang}[1]{{\left\langle{#1}\right\rangle}}

\newcommand{\ep}{{\epsilon}}

\renewcommand\div{\operatorname{div}}
\newcommand\xn{(x \cdot n)}
\newcommand\xni{(x \cdot n)^{-1}}
\def\wang#1#2{\langle #1  , #2  \rangle} 

\newcommand\dOmega{{\partial \Omega}}
\newcommand\Deltab{\Delta_{\partial \Omega}}
\newcommand\Deltabw{\Delta_{\partial \Omega, w}}
\newcommand\Id{\operatorname{Id}}

\newcommand\dnub{\partial_n u |_{\dOmega}}  
\newcommand\ub{u |_{\dOmega}}
\newcommand\ubar{\overline{u}}
\newcommand\ut{v}
\renewcommand\Im{\operatorname{Im}}
\renewcommand\Re{\operatorname{Re}}
\newcommand\ntd{\Theta}                           

\newcommand\kstar{k_*}                    
\newcommand\fstar{f_*}
\newcommand\bstar{\beta_*}
\newcommand\kustar{k_0}                

\newcommand{\hep}{\hat{\ep}_p}      

\newcommand{\vt}[2]{\biggl[\begin{matrix}#1\\#2\end{matrix}\biggr]} 
\newcommand{\mt}[4]{\biggl[\begin{matrix}#1&#2\\#3&#4\end{matrix}\biggr]} 

\theoremstyle{plain}
\newtheorem{theorem}{Theorem}[section]
\newtheorem{proposition}[theorem]{Proposition}
\newtheorem{lemma}[theorem]{Lemma}
\newtheorem{corollary}[theorem]{Corollary}
\theoremstyle{remark}
\newtheorem{remark}[theorem]{Remark}
\newtheorem{assumption}[theorem]{Assumption}
\theoremstyle{definition}
\newtheorem{defn}[theorem]{Definition}

\begin{document}

\title[Fast computation of high frequency
Dirichlet eigenmodes]{Fast computation of high frequency
Dirichlet eigenmodes via the spectral flow of the interior
Neumann-to-Dirichlet map}

\author{Alex Barnett}
\address{Department of Mathematics,
Dartmouth College, Hanover, NH, 03755, USA}
\email{ahb@math.dartmouth.edu}
\author{Andrew Hassell}
\address{Department of Mathematics, Australian National University, Canberra 02
00 ACT, \phantom{ar} AUSTRALIA}
\email{hassell@maths.anu.edu.au}


\begin{abstract}
We present a new algorithm for numerical computation of
large eigenvalues and associated
eigenfunctions of the Dirichlet Laplacian in a smooth, star-shaped
domain in $\mathbb{R}^d$, $d\ge 2$.
Conventional boundary-based methods require a root-search in
eigenfrequency $k$, hence
take $O(N^3)$ effort per eigenpair found, using dense linear algebra,
where $N=O(k^{d-1})$ is the number of unknowns required to discretize
the boundary.
Our method is $O(N)$ faster, achieved by linearizing with
respect to $k$ the spectrum of a weighted interior
Neumann-to-Dirichlet (NtD) operator for the Helmholtz equation.
Approximations
$\hat{k}_j$ to the square-roots $k_j$ of all $O(N)$ eigenvalues lying in
$[k - \ep, k]$, where $\ep=O(1)$, are found with $O(N^3)$ effort.
We prove an error estimate
$$
|\hat k_j - k_j| \leq C \Big( \frac{\ep^2}{k} + \ep^3 \Big),
$$
with $C$ independent of $k$.
We present a higher-order variant with eigenvalue error scaling
empirically as $O(\ep^5)$
and eigenfunction error as $O(\ep^3)$,
the former improving upon the `scaling method' of Vergini--Saraceno.
For planar domains ($d=2$),
with an assumption of absence of spectral concentration,
we also prove rigorous error bounds
that are close to those numerically observed.
For $d=2$ we compute robustly 
the spectrum of the NtD operator 
via potential theory,
Nystr\"{o}m discretization, and the Cayley transform.
%
At high frequencies (400 wavelengths across),
with eigenfrequency relative error $10^{-10}$,
we show that the method is $10^3$ times faster than standard ones
based upon a root-search.
%
%
%
%
\end{abstract}
\maketitle
\tableofcontents

\section{Introduction}

Let $\Omega$ be a smooth, bounded domain in $\RR^d$, strictly
star-shaped with respect to the origin,
that is $x\cdot n > 0$ for each $x\in\pO$ where $n$ is
the outward-pointing unit normal vector.
We are interested in computing numerically the eigenvalues $k_j^2$,
and eigenfunctions or eigenmodes
$\phi_j$ (normalized by $\| \phi_j\|_\lo = 1$),
of the Dirichlet Laplacian
$\Delta = \sum_{i=1}^d \partial^2/\partial_{x_i}^2$ on $\Omega$. 
That is,
\bea
(\Delta + k_j^2) \phi_j &=& 0 \qquad \mbox{ in } \Omega~,
\label{e:phij}
\\
\phi_j &=& 0 \qquad \mbox{ on } \pO~.
\label{e:bc}
\eea
We will refer to $k_j$, the square-roots of eigenvalues, as (Dirichlet)
eigenfrequencies,
and order them $0<k_1<k_2\le k_3 \le \dots$ counting multiplicities.
This classical problem has many applications
in engineering and physics \cite{CoHi53,babosrev},
principally in the modeling of
acoustic, electromagnetic and optical cavities,
vibrating membranes, trapped quantum particles
and nano-scale devices \cite{qdots}, and in data analysis \cite{saito}.
Note that some applications involve homogeneous boundary conditions
other than \eqref{e:bc}, or the Maxwell or elasticity equations,
yet the above serves as a paradigm for this larger class of problems.
In $d=2$ it is known as the `drum' problem, and is reviewed in
\cite{KS,tref06}.
A numerical approach is needed for all but the small subset of domains $\Omega$
where separation of variables is possible (explicitly,
those which are a product of intervals in a coordinate
system in which $\Delta$ is separable \cite{CoHi53}).

Many applications demand high eigenfrequency $k_j$ (i.e.\
high mode number $j$),
which creates a challenging numerical problem.
For instance, the design of high-power micro-laser resonators \cite{hakan05}
requires $j > 10^3$ (i.e.\ tens of wavelengths across the domain).
Knowledge of eigenfunctions informs high-frequency
wave scattering from resonant structures
such as jet engine inlets \cite{kriegs}.
Interest has also surged recently in quantum chaos 
\cite{zencyc,nonnenrev} and spectral geometry \cite{GiraudThas},
where numerical studies have played a key role,
such as in the discovery of `scars'
of periodic ray orbits in chaotic eigenfunctions \cite{hel84}, and the
study of eigenfunction
equidistribution rates \cite{baecker,que}.
This can involve computing thousands of modes at up to $j\sim 10^6$, i.e.\
hundreds of wavelengths across the domain \cite{scalinguse1,que}.
The above motivates the creation of efficient high frequency numerical methods
with controlled errors.

Existing
numerical methods for \eqref{e:phij}-\eqref{e:bc} generally fall into two
classes:
\ben
\item [A.]
{\em Direct discretization} of $\Omega$ (via finite differences or
finite elements \cite{babosrev}),
which has the advantage that eigenvalues $k_j^2$ are
approximated by the spectrum of a {\em linear}
(sparse, often generalized) matrix eigenvalue problem.
However, since several degrees of freedom per wavelength in each dimension
are needed, the number of unknowns $N$ grows at least like $k^d$.
In fact, to achieve bounded accuracy as $k\to\infty$
an {\em increasing} number of unknowns per wavelength are required;
this is the so-called `pollution effect' \cite{pollution}.
%
Iterative methods are needed for such huge eigenvalue problems.
We believe the furthest this has been pushed in $d=2$ is
$j \sim 3\times 10^3$ (around 30 wavelengths across the domain),
by Heuveline and others \cite{heuveline,dietzFEM,demen07}.
However, here specialized multigrid and removal of spurious eigenvalues are
needed, and relative errors in $k_j$ are as high as $10^{-3}$.
\item [B.]
{\em Boundary-based methods}, which make use of a basis of analytic solutions
to the Helmholtz equation \eqref{e:phij}, hence only require
discretization of $\pO$ via a much smaller $N=O(k^{d-1})$ unknowns.
The main disadvantage is that, since the $k$-dependence of the basis
is nonlinear, eigenfrequencies $k_j$ are now given by a
(dense) {\em nonlinear} eigenvalue problem.
This generally requires repeated iterative minimization of some error
measure along the $k$ axis,
which is cumbersome and prone to the omission of eigenfrequencies
\cite{backerbim,hakansca}.
The error measure is often a minimum singular value
(e.g.\ see App.~\ref{a:ref}),
hence $O(N^3)$ effort is required per eigenfrequency found.

This class includes the method of particular solutions (MPS) \cite{mps}
(also known as collocation, Trefftz, non-polynomial FEM,
or ultra-weak variational formulation \cite{uwvf,monkwang})
which uses plane-wave \cite{hel84}, regular Bessel \cite{lepore,steinbachunger},
or corner-adapted Fourier-Bessel solutions \cite{driscoll,mps,timodd};
the method of fundamental solutions \cite{Ka01} which uses point sources
placed outside of $\Omega$;
and boundary integral equation (BIE, also known as boundary element) methods
which make use of potential theory
on $\pO$ \cite{coltonkress}.
Such methods often have spectral (i.e.\ super-algebraic) error convergence,
although most BIE implementations remain low-order \cite{kirkup,backerbim,duran,veblemonza}. They can easily reach $j=10^4$, with relative errors
as small as $10^{-14}$ \cite{bnds},
and variants have reached $j>10^6$ \cite{scalinguse1,hakansca}.
%
\een
Can one combine the advantages of classes A and B,
i.e.\ is there a boundary-based method that does not require a root search
for each eigenfrequency?
This was answered, in the case of star-shaped domains,
by Vergini--Saraceno~\cite{v+s} who proposed
a `scaling method'---reviewed in section~\ref{s:v+s}---
which may be viewed as an acceleration technique for the MPS.
Here a single dense matrix eigenvalue problem, i.e.\ effort $O(N^3)$,
approximates all eigenfrequencies (and their eigenfunctions)
lying in an interval of the $k$ axis of length $\ep=O(1)$.
Since by Weyl's law \cite[Ch.~11]{garab}
one expects $O(k^{d-1})$ such eigenfrequencies,
this is also the speed-up factor of the method, 
assuming errors are acceptable.
The absolute eigenfrequency error is empirically $O(\ep^3)$ \cite{mythesis,que},
although this has scarcely been studied.
The scaling method has allowed large-scale studies of quantum chaos
to be performed in $d=2$ \cite{scalinguse1,que,mush} and $d=3$ \cite{prosen3d}
at speeds around $10^3$ times faster than any other known method.

This key idea of {\em linearizing the nonlinear eigenvalue problem}
in class B has been noticed by couple of other researchers.
Kirkup--Amini \cite{kirkup} used the linear formulation of a
polynomial eigenvalue problem to
approximate the nonlinear eigenvalue problem, for low $k$ only.
In terms of BIE, Tureci--Schwefel \cite{hakansca}
have used the empirical observation that
as a function of $k$, eigenvalues of the double
layer operator (see \eqref{e:D} below) rotate in the
complex plane at roughly constant speed.
Veble {\em et al} convert the BIE to a
generalized eigenvalue problem to similar effect \cite{veblemonza}.
Heuristically, these last two methods have the same $O(N)$ acceleration
as the scaling method.
However, the error analysis of the scaling method or such variants is
very primitive, and certainly no rigorous results exist.

Here we remedy this
by presenting, and analysing in depth, a new class B linearization method
for the eigenproblem \eqref{e:phij}-\eqref{e:bc} in smooth star-shaped
domains,
close in spirit to a BIE method.
It is based
upon the $k$-dependence of the spectrum of an interior%
\footnote{In contrast,
it is the {\em exterior} NtD or DtN map
that plays a common role in applying radiation conditions in wave
scattering. The interior NtD has been used in analysis of inverse problems,
\cite{nachman} and to bound eigenvalues \cite{Friedlander}.}
Neumann-to-Dirichlet
(NtD) operator for the Helmholtz equation at wavenumber $k$,
as presented in section~\ref{s:ntd}.
The key idea is that an eigenvalue of the NtD reaches zero whenever
$k$ reaches an eigenfrequency $k_j$,
and thus by computing all small eigenvalues of NtD one may predict
all nearby eigenfrequencies $k_j$.
The basic algorithm is presented and tested in section~\ref{s:basic}.

We devote a large part of this work to the analysis of the spectrum
and eigenfunctions of the NtD, in particular their flow with $k$,
in the $k\to\infty$ limit.
This enables us to analyze the basic method (in section~\ref{s:kerr}), then
propose (section~\ref{s:higher}) and analyze
(section~\ref{s:efn-error-analysis})
higher-order accurate variants.
The main tools we need are: analytic perturbation theory
(in section~\ref{s:ntd} and App.~\ref{a:A}), 
microlocal analysis (App.~\ref{a:AppB}), and a generalization of
a recent `spectral window quasi-orthogonality' result
of the authors \cite{bnds} (App.~\ref{a:estinv}).
%

Here we summarize our main theoretical
results:
\bi
\item
When correctly weighted (as in \cite{v+s})
by the function $\xni$ on $\pO$, the
spectrum of the NtD varies approximately linearly with $k$ with
slope known a priori (Theorem~\ref{thm:linear-beta}).
This will imply that the basic linearization method
has the eigenfrequency error estimate
$|\hat k_j - k_j| \leq C (\ep^2/k + \ep^3)$,
where $\hat k_j$ is the approximate eigenfrequency.
It is crucial that here $C$ is independent of $k$.
Since we establish a one-to-one correspondence between eigenfrequencies
slightly larger than $k$ and slightly negative NtD eigenvalues,
this proves that our method has {\em neither spurious
nor missing eigenfrequencies}.
\item
We propose a higher-order accurate formula
for prediction of eigenfrequencies \eqref{e:khatr},
using an identity for Helmholtz solutions \eqref{Alex's-identity}
due to the first author \cite{que}.
We show that the dominant error term is $O(\ep^5)$,
which improves upon the $O(\ep^3)$ of existing scaling methods
\cite{mush,veblemonza}.
\item
We propose a higher-order accurate formula for (boundary data of)
eigenfunctions \eqref{fhatimpsecondorder}.
This requires formulae for the 1st and 2nd derivative
with respect to $k$
of the NtD eigenfunctions $k$ is an eigenfrequency
(Prop.~\ref{prop:fderivs}).
We will show a dominant error $L^2$-norm of $O(\ep^3)$.

\item In $d=2$, and making a spectral non-concentration assumption
(see Assumption~\ref{ASC}),
we prove rigorously that these higher-order methods
achieve a dominant eigenfrequency error of $O(\ep^5)$
(Prop.~\ref{p:higherorderbetaerror}) and
eigenfunction error of $O(k \ep^3)$ (Prop.~\ref{p:fperror}).
The latter has the same dependence on $\ep$ as existing scaling methods.

\item In addition, we believe that Lemma~\ref{lem:utu} and the results of
Appendix~\ref{a:A} are useful contributions to the theory of elliptic
boundary-value problems, independent of any numerical considerations.
\ei

On the implementation side, we show in section~\ref{s:bie} that
the spectrum of the above NtD operator may be approximated
with an error uniformly close to machine precision,
hence that the above error bounds hold in practice.
This requires a new method
based upon potential theory, the Cayley transform, and
the quadratures of Kress \cite{kress91}.
For $\pO$ an analytic curve, we demonstrate exponential convergence.
%
This improves upon the low-order quadratures of all previous
scaling variants \cite{v+s,hakansca,veblemonza} and almost all BIE
methods for eigenvalues in the literature.

We compare the performance of our method in $d=2$ against a standard
BIE root-search (described in App.~\ref{a:ref}), which also
serves to give us reference sets of $k_j$ and $\phi_j$ against which 
to measure errors. 
We test two domains, one with no symmetry, and, in section~\ref{s:pentafoil},
one with a symmetry that causes an abundance of degeneracies.
The latter is evidence that Assumption~\ref{ASC} may be violated with no
impact on performance.
We find that the $O(N)$ speed-up translates in practice to a factor of $10^3$
faster solution at high frequencies.

In section~\ref{s:v+s} we give a new understanding  of the
original Vergini--Saraceno scaling method
in a mathematical framework, and draw some comparisons with our
proposed method.
Finally, we conclude in section~\ref{s:conc} and give some open questions.
We have made a documented software implementation of
the proposed algorithms freely available
(in the \mpspack\ toolbox for \matlab), and intersperse
section~\ref{s:basic}
and beyond
with code examples showing how to use these routines.





\bfi
a)\raisebox{-1.3in}{\ig{width=1.05\textwidth}{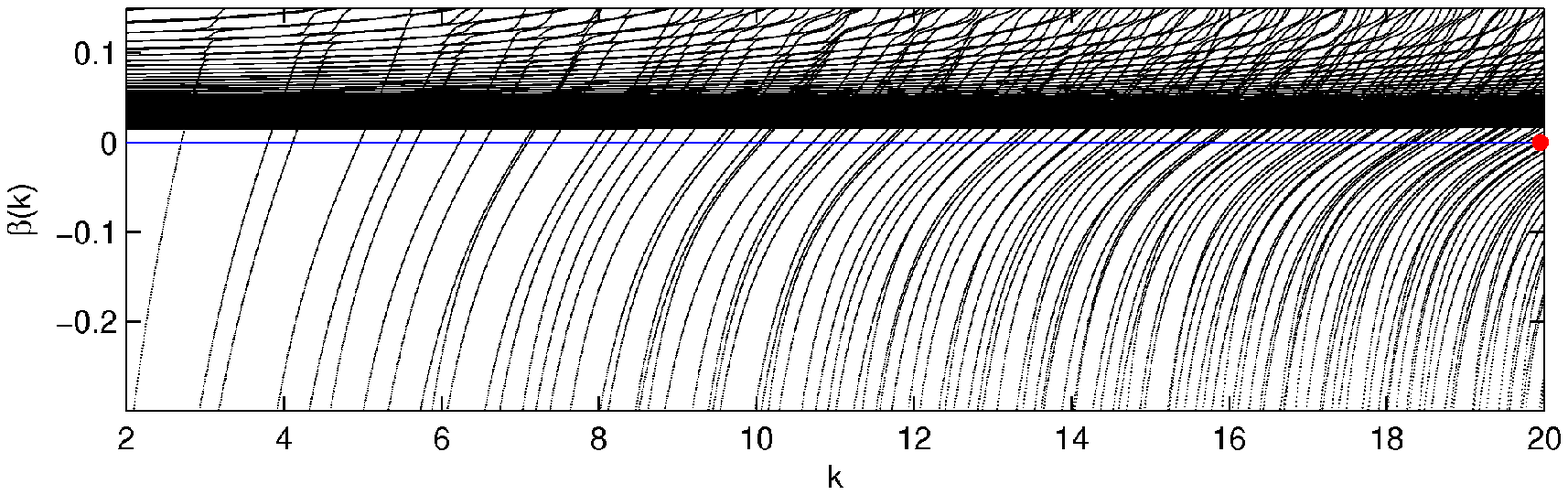}}

\mbox{b)\raisebox{-1.2in}{\ig{width=1.1in}{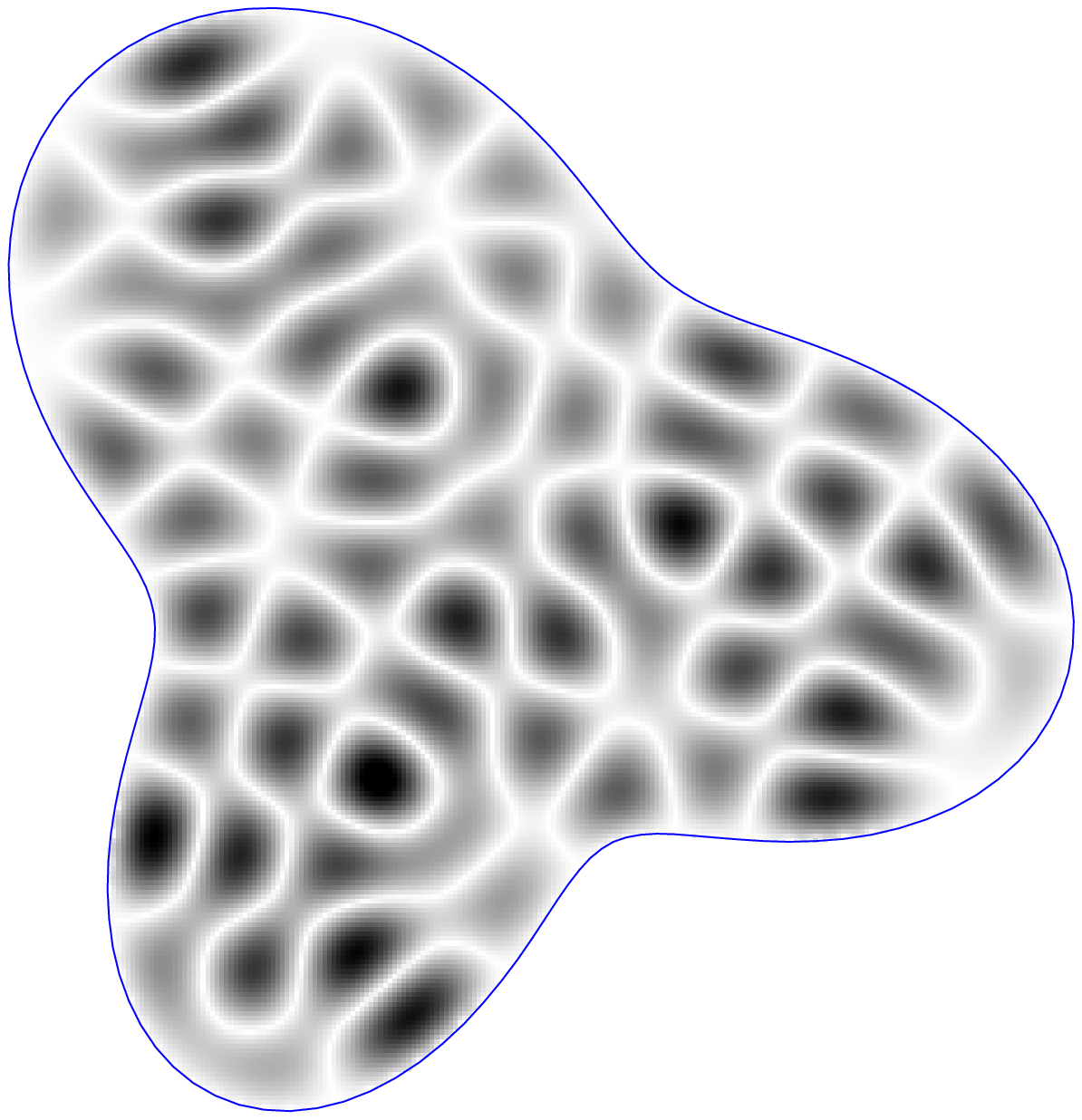}}
c)\hspace{0ex}\raisebox{-1.4in}{\ig{width=1.9in}{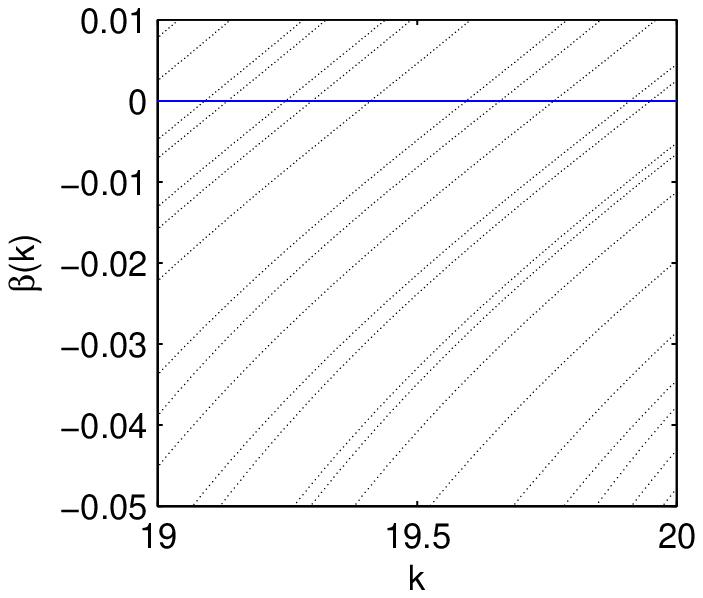}}
d)\hspace{0ex}\raisebox{-1.4in}{\ig{width=1.9in}{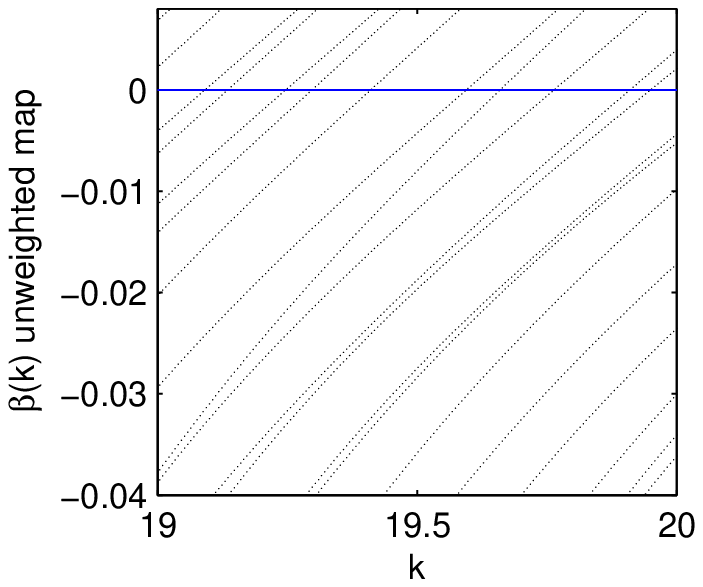}}
}

\ca{(a) Flow of the eigenvalues $\beta(k)$ of the weighted Neumann-to-Dirichlet
map $\ntd$ vs wavenumber $k$, for the domain $\Omega$ given
in polar coordinates by $r(\theta) = 1 + 0.3 \cos[3(\theta+ 0.2\sin\theta)]$.
(b) Eigenmode $\phi_{93}$ (density shows absolute value;
white is zero) with $k_{93} = 19.94995891589\cdots$ (also shown by red dot
in a). (c) Zoom in of flow. (d) Zoom in to the same region for eigenvalues of 
$\Lambda(k)^{-1}$ the {\em unweighted} Neumann-to-Dirichlet map;
there is variation in slopes at the zero-crossings.
}{f:flow}
\efi

\section{The Neumann-to-Dirichlet map and its spectral flow}
\label{s:ntd}

We will use $\dn{u}:=\partial_n u := n\cdot\nabla u$ to denote the outward normal derivative
of a function $u$ defined in $\overline{\Omega}$.
Let $\Lambda(k)$ be the interior Dirichlet-to-Neumann operator
for the Helmholtz equation with parameter $k^2$, that is,
the operator that sends a function
$g \in H^1(\dOmega)$ to the function $h \in L^2(\dOmega)$ given by
$h=\dn{u}$, where $u$ is the interior Helmholtz extension satisfying
\be
(\Delta + k^2) u = 0 \quad \text{ in } \Omega, \qquad u|_\pO = g.
\label{e:dbvp}
\ee
This is well defined for every $k \in \CC$ except when $k = k_j$ is a
Dirichlet eigenfrequency, in which case the function $u$ may not exist
(and is nonunique when it does exist).
It is well-known that for $k \neq k_j$, $\Lambda(k)$ is self-adjoint,
as the following elementary calculation involving Green's
second identity shows. Suppose that $g,h \in H^1(\pO)$ and that $u, v$ are
their interior Helmholtz extensions, respectively, then
$$\begin{gathered} 
\int_\pO \overline{(\Lambda g)} h -  \int_\pO \overline{g} \Lambda h = \int_\pO \dn{\overline{u}}v -  
\int_\pO \overline{u}\dn{v} \\
= \int_\Omega  [(\Delta + k^2 ) \overline{u}] v -
\overline{u} (\Delta + k^2)v  =
0~.
\end{gathered}$$
This and other properties of $\Lambda$ are presented by
Friedlander~\cite{Friedlander}.

Unless indicated, we work with a weighted inner product on the boundary $\pO$,
denoted by angle brackets $\wang{\cdot}{\cdot}$, and
induced norm, as follows,
\be
\wang{g}{h} := \int_{\partial \Omega} \overline{g(s)} h(s) \, \left(x(s)\cdot n(s)\right)^{-1} ds~,
\qquad
\|g\|^2 := \wang{g}{g}~.
\label{e:ip}
\ee
Note that if $\Omega$ is
strictly star-shaped about the origin, the weight is bounded and positive.
It is easy to
check that the operator $\xn \circ \Lambda(k)$ is self-adjoint
with respect to this weighted inner product.
Let $\ntd(k)$ denote the inverse of this operator, that is,
$\ntd(k) := \Lambda(k)^{-1} \circ \xn^{-1}$, then $\ntd$ is also self-adjoint
with respect to the weighted inner product.
By definition, if $u$ is any interior Helmholtz solution, then
\be
\ntd(k)\, \xn \dn{u} = u|_\pO ~.
\label{e:ntdu}
\ee
In this paper, we will analyze the flow of eigenvalues and eigenspaces
of $\ntd(k)$ as $k$ varies along the real axis, that is,
nontrivial solutions $f\in\lpo$ to
\be
\ntd(k) f(k) = \beta(k) f(k) ~.
\label{e:ef}
\ee
Taking $u=\phi_j$, a Dirichlet eigenmode, in \eqref{e:ntdu},
we see that $\ntd(k)$ has a zero
eigenvalue at each Dirichlet eigenfrequency $k=k_j$,
with eigenfunction 
$f=\xn \partial_n \phi_j$;
this is why the Neumann-to-Dirichlet map is of interest computationally.
%
Considering the case of $u$ a Neumann Laplace eigenmode of the domain
shows that $\ntd(k)$, and hence its spectrum, has a pole at each
Neumann eigenfrequency $k$.
Fig.~\ref{f:flow} illustrates the zeros in $\beta(k)$ occurring at
each of the lowest 93 Dirichlet eigenvalues of a domain
(the poles are also hinted at for larger negative $\beta$).
Also visible is the accumulation%
\footnote{The small gap visible above $0$ is due to the
numerical approximation of the operator.}
of eigenvalues at $0^+$ that occurs for all real $k$,
a result of $\Lambda$ being a pseudodifferential operator of
order $+1$ \cite{Friedlander}, hence $\ntd$ a compact operator (of order $-1$).

We wish to flow along an interval of the real $k$-axis that will
likely contain several Neumann eigenfrequencies, and
need to guarantee that
all of the eigenprojections and eigenvalues of $\ntd$ vary
smoothly
except possibly for a finite number
associated with a pole if $k$ is a Neumann eigenfrequency.
To do that, we consider the Cayley transform
of $\ntd$,
\be
R(k) = \big(\ntd(k) - i \big) \big(\ntd(k) + i \big)^{-1}~.
\label{e:cayley1}
\ee
In Appendix~\ref{a:A}, Corollary~\ref{cor:R(k)-analytic}, we
show that $R(k)$ is analytic in some
neighbourhood of the positive real axis. As $R(k)$ is unitary for real
$k$, its spectrum lies on the unit circle, and is discrete except at
$-1$ because $\ntd$ is compact (its spectrum accumulates only at
$0$).
Thus we deduce from Kato \cite[Ch. VII, sec.~3]{Kato}
%
that the eigenprojections and eigenvalues of
$R(k)$ vary analytically away from eigenvalue $-1$. Translated back to
$\ntd$
this means that the eigenprojections and eigenvalues of $\ntd(k)$
vary analytically in $k$ on any finite $k$-interval away from
eigenvalue $0$, apart from a finite number which have a pole at
one of the Neumann eigenfrequencies in this interval.

\begin{defn}
Let $\beta=\beta(k)$ be a finite eigenvalue of $\ntd(k)$
with boundary-normalized eigenfunction $f = f(k)$, $\|f\|=1$. 
The {\em extended eigenfunction} is then the unique solution $u$ to the
interior boundary-value problem
\bea
(\Delta + k^2)u &=& 0 \quad\mbox{ \rm in } \Omega~,
\label{e:upde}
\\
\xn \dn{u} &= &f \quad\mbox{ \rm on } \pO~,
\label{e:unbc}
\\
u &=& \beta f \quad\mbox{ \rm on } \pO~.
\label{e:ubc}
\eea
\end{defn}
Note that we have both Neumann and Dirichlet conditions on $u$;
the latter is needed for uniqueness when $k$ is a Neumann eigenfrequency.
Their consistency at all $k$ is ensured by \eqref{e:ntdu}.
We may view $u$ as a solution to a Stekloff eigenvalue problem
with Robin condition
\be
\beta\xn \dn{u} = u~.
\label{e:robin}
\ee
Note that the extended eigenfunction $u$ is not normalized in $\lo$.

The rate of change with $k$ of each isolated eigenvalue is then
given by the following variant of a result of
Friedlander~\cite[Prop.~2.5]{Friedlander}.
For convenience we give the proof.
\begin{lemma}\label{lem:Fried} 
Let $\beta(k)$ be a
analytic eigenvalue branch of $\ntd(k)$
with normalized eigenfunction $f=f(k)$, $\|f\|=1$, and let $u$
be its extended eigenfunction.
Then, using the notation $\dk{\beta}:=d\beta(k)/dk$, it holds that
\be
\dk{\beta}
= 2k \int_\Omega |u|^2~.
\label{e:dkbeta}
\ee
\label{l:dkbeta}\end{lemma}  
\begin{proof} 
From \eqref{e:ef} follows the usual Hellman-Feynman formula,
\be
\dk{\beta} = \frac{d}{dk}  \wang{\ntd f}{f}
= \wang{\dk{\ntd} f + \ntd \dk{f}}{f} + \wang{\ntd f}{\dk{f}}
= \wang{\dk{\ntd} f}{f}
\label{e:HF}
\ee
where the last step comes from the normalization of $f$, which implies
$\wang{\dk{f}}{f} = 0$.
Let $k=k_0$ be the frequency in the
statement of the Lemma, and
restrict for now to the case that this is
not a Neumann eigenfrequency, in which case there is a unique
solution $u$ to the pair \eqref{e:upde} and \eqref{e:unbc} given
boundary data $f$.
Holding this boundary data fixed at  $f = f(k_0)$,
let $\ut(k)$ be the solution to the boundary value problem
\begin{equation}
(\Delta + k^2) \ut(k) = 0, \quad \xn \dn{\ut}(k) = f(k_0).
\label{e:vdefn}\end{equation}
(Note $v$ is not the same as the extended eigenfunction $u$ except at $k = k_0$.)
Let $\dk v$ be the $k$-derivative of this solution at $k=k_0$.
Then, by the definition \eqref{e:ntdu},
\be
\dk{\ntd}f = \dk v |_{\dOmega}~ \ \text{ at } k = k_0.
\label{e:ntdpf}
\ee
Also, by
differentiating the defining conditions \eqref{e:vdefn}
we get a boundary value problem for $\dk v$,
\begin{equation}
(\Delta + k^2) \dk v = -2kv \; \mbox{ in } \Omega, \qquad
\dn{\dk v} = 0 \; \mbox{ on } \pO.
\label{e:upbvp}
\end{equation}
Combining this with \eqref{e:ntdpf} and \eqref{e:HF} in Green's 2nd identity gives
\begin{equation}\begin{gathered}
\dk \beta(k_0) = \wang{\dk{\ntd}(k_0) f}{f} = 
\int_{\pO} \xni \dk{\overline{v}} f = 
\int_{\dOmega} \dk{\overline{v}} \dn{v} =
\int_{\dOmega} \left( \dk{\overline{v}} \dn{v} - \dn{\dk{\overline{v}}} v \right)  \\
= \int_{\Omega}
\dk{\overline{v}} (\Delta + k_0^2)v - [(\Delta + k_0^2) \dk{\overline{v}}] v
= 2k_0 \int_{\Omega} |v|^2 = 2k_0 \int_{\Omega} |u|^2.
\end{gathered}\label{betadotpositive}\end{equation}
This completes the proof when $k_0$ is not a Neumann eigenfrequency.
When $k_0$ is a Neumann eigenfrequency,
$f(k)$ and $\beta(k)$
are still analytic in a neighbourhood of $k_0$ (as discussed above),
so one may take a sequence
with $k_0$ as the limit and prove the same formula.
\end{proof}  
\vspace{-2ex}
This fact that this lemma can be applied in the limit $\beta \uparrow 0$
is justified at the end of App.~\ref{a:A}.
Notice that we always have $\dk{\beta} > 0$,
illustrated by the positive slopes in Fig.~\ref{f:flow}.

We now can explain the reason for choosing
the particular weight in the inner product
\eqref{e:ip}.
Let $\beta(k)$ be an analytic eigenvalue branch of $\ntd(k)$
which has $\beta(k_j) = 0$ for some $j$, that is,
the branch corresponding to the $j$th eigenfrequency.%
\footnote{This existence of this branch is guaranteed by Proposition~\ref{prop:eig-branches}. }
Then at $k=k_j$, the extended eigenfunction $u$ is a Dirichlet eigenfunction.
For Dirichlet eigenfunctions, we
have Rellich's identity \cite{rellich} (a special case of \eqref{Alex's-identity}),
\be
2k_j^2 \int_\Omega |u|^2 = \int_{\dOmega} \xn |\dn{u}|^2 = \wang{f}{f} = 1~.
\label{e:rellich}
\ee
Inserting this into Lemma~\ref{l:dkbeta} gives a 
formula for the slopes at zero eigenvalue,
\be
\beta = 0 \;\;\implies \;\; \dk{\beta}(k_j) = \frac1{k_j}~.
\label{e:invk}
\ee

\brmk
\eqref{e:invk} shows that,
for the special boundary weight function $\xni$,
the eigenvalues of $\ntd(k)$ cross zero
at a uniform, predictable
positive speed that is
{\em independent of the details of the distribution of the
eigenmode} $\phi_j$.
This predictable behavior is not known to occur for any other
weight function: for example,
the contrast between this special weight and the unweighted case (where
speeds vary unpredictably with $j$) is shown in Fig.~\ref{f:flow} (c) and (d).
\ermk

\section{Basic numerical algorithm}
\label{s:basic}

We first present a simple fast algorithm
to approximate the eigenfrequencies and eigenfunctions of the domain $\Omega$
using spectral data of $\ntd$;
in section~\ref{s:higher} we will improve it to have higher-order accuracy.

\bfi   
\ig{width=0.48\textwidth}{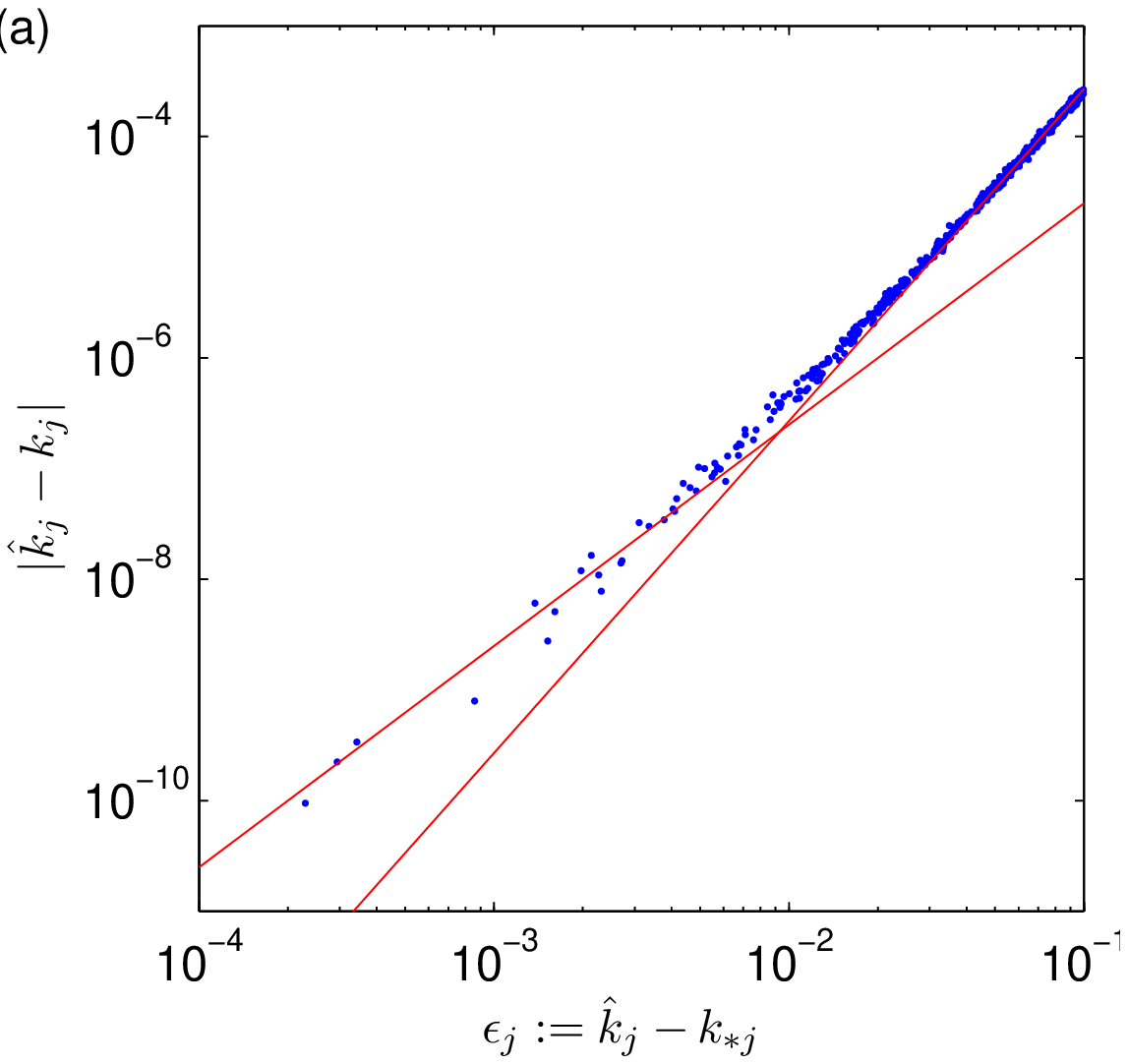}
\ig{width=0.49\textwidth}{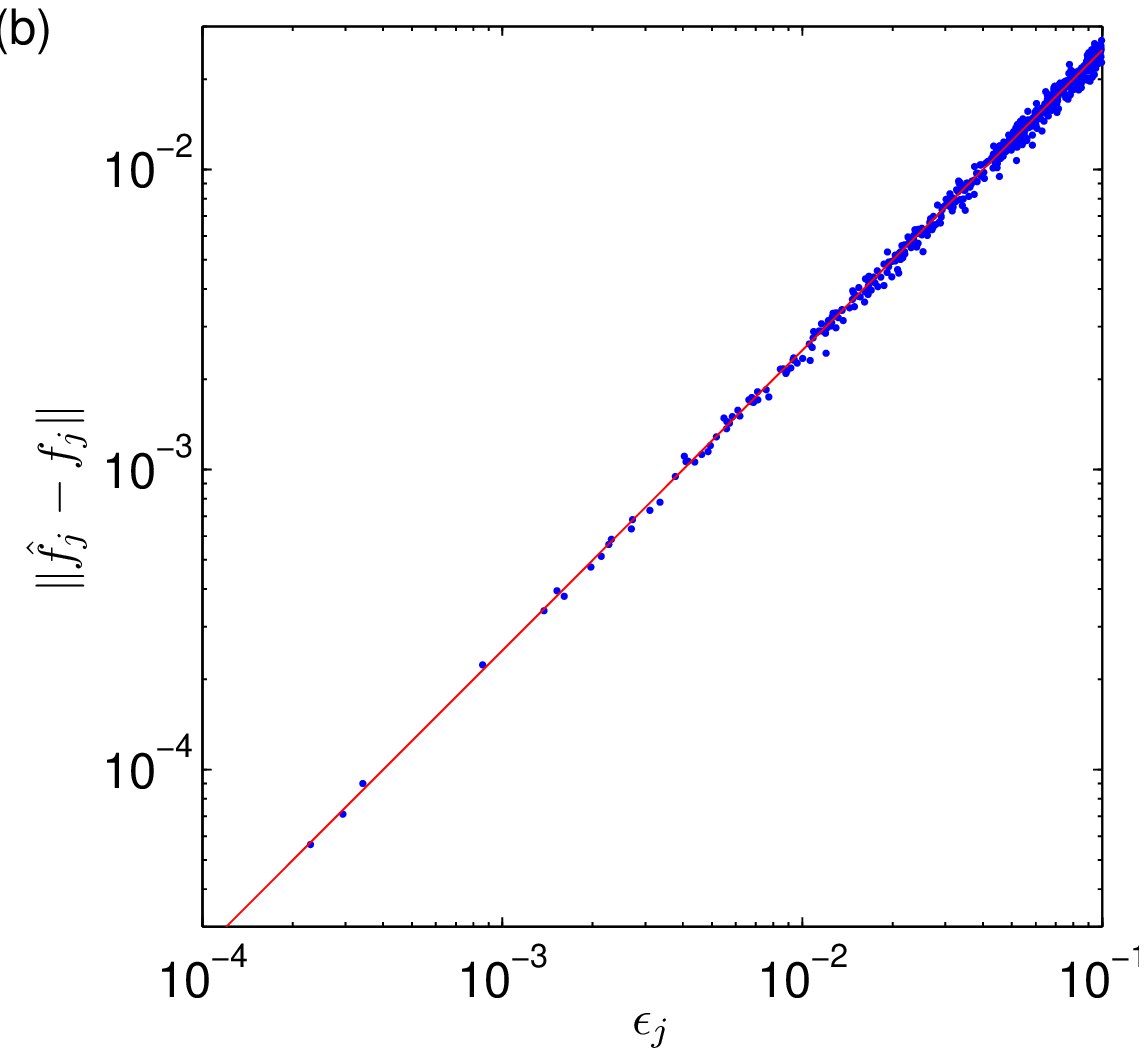}
\ca{Errors with basic method, for the domain
of Fig.~\ref{f:flow}(b). (a)
Error of predicted eigenfrequency
$\hat{k}_j$ using linear formula \eqref{e:khatl},
vs prediction  distance $\ep_j$,
for all frequencies $k_j\in[90,100]$.
Lines show $0.27\ep^3$ and $0.25\ep^2/k$.
(b) Errors of predicted boundary functions $\hat{f}$
in the weighted  
$L^2$ norm \eqref{e:ip}. Line shows $0.25\ep$.
}{f:errl}
\efi

\subsection{Reconstructing eigenfrequencies}
\label{s:khat}

Since each Dirichlet eigenfrequency $k_j$ is associated with an
analytic eigenvalue branch of the spectrum of $\ntd(k)$,
we may use this spectral flow of $\ntd(k)$ to locate approximately the $k_j$.
Fig.~\ref{f:flow} (c) illustrates that the
gradients $\dk{\beta}(k)$ are approximately constant on each branch
for $k$ near $k_j$;
in section~\ref{s:kerr} we will prove that the range of $k_j-k$ for
which this usefully holds is a constant independent of $k_j$.
Thus, choosing a frequency $\kstar$ and computing
the spectrum of $\ntd(\kstar)$,
then for each 
of its small negative eigenvalues $\bstar=\beta(\kstar)$,
one may extrapolate linearly to the corresponding Dirichlet eigenvalue by
the
\be
\mbox{``linear estimator'':} \qquad \hat{k} \;= \;\frac{\kstar}{1 + \bstar}~.
\label{e:khatl}
\ee
This follows simply from \eqref{e:invk} and by making the linear approximation
$\bstar \approx \dk{\beta}(k_j)(\kstar - k_j)$.
We keep only those $\hat{k}$ values lying in the interval or `window'
$[\kstar,\kstar + \ep]$, where $\ep$ is an $O(1)$ constant.
Since, by Weyl's law \cite[Ch. 11]{garab} asymptotically $O(k^{d-1})$
eigenfrequencies lie
in such an interval, this is also the order by which the method
is faster than the standard iterative search for each eigenfrequency.
By repeating the above with adjacent intervals one may find
approximations to all eigenfrequencies lying in any desired subset of
the frequency axis.

In section~\ref{s:bie} we present the spectrally-accurate method
we use (in $d=2$) to compute numerically 
the spectrum of $\ntd(\kstar)$.
This algorithm 
has been built into the \mpspack\ toolbox 
toolbox for \matlab\ \cite{matlab},
so that 
the set of approximate eigenfrequencies $\hat{k}_j$ lying in
$[90,100]$ may be computed, for example, for the nonsymmetric, smooth
(in fact analytic) domain $\Omega\subset\RR^2$
shown in Fig.~\ref{f:flow} (b), as follows:

\begin{verbatim}
s = segment.smoothnonsym(720, 0.3, 0.2, 3); % create a closed curve
d = domain(s, 1);                           % create an interior domain
s.setbc(-1, 'D');                           % Dirichlet BCs on inside
p = evp(d);                                 % create eigenvalue problem
o.khat = 'l'; o.eps = 0.1; p.solvespectrum([90 100], 'ntd', o);
\end{verbatim}

Here $N=720$ sets the number of boundary quadrature points to
about 6 per wavelength on the boundary,
typically sufficient for approximating $\ntd$ at close to double-precision
accuracy.
The options structure {\tt o} chooses the linear method \eqref{e:khatl}
and sets $\ep=0.1$.
The object {\tt p} now contains {\tt p.kj}, being the list of
492 approximate eigenfrequencies $\hat{k}_j$ found
(these are in fact numbers $j=[2064,2555]$ for the domain).
All were found to be simple, as expected generically since $\Omega$
has no symmetry.
The majority of them have absolute errors less than $10^{-4}$.
The CPU time for the above example was 13 min, ie 1.6 s
per computed eigenfrequency.%
\footnote{Runtimes are reported for 
a 2005-era workstation (two single-core Opteron 2GHz 250 CPUs)
with 8 GB of RAM, running linux, MATLAB 2008a,
and \mpspack\ version 1.2.}

The size of the absolute eigenfrequency errors
are shown in Fig.~\ref{f:errl} (a), versus
\be
\ep_j := k_j - \kstar~,
\label{e:epj}
\ee
the frequency `distance' over which the linearization occurred.
Errors are $O(\ep_j^2/k)$ at small distances
but $O(\ep_j^3)$ at large distances: these two terms are
shown by straight lines in Fig.~\ref{f:errl} (a).
We are able to prove a bound involving these two terms
in Corollary~\ref{p:khatl},
which states that the implied constants are independent of $k$.
The transition point (intersection of the straight lines)
occurs at $\ep = O(1/k)$. Since generically only a fraction $O(1/k)$ of
the eigenfrequencies in the window lie below this $\ep$ distance,
the method is asymptotically 3rd-order accurate in the interval width $\ep$.


These errors reported above were found by comparison
against an accurate set of
eigenfrequencies $k_j$ 
found independently by a standard method
from the literature described in App.~\ref{a:ref}.
This reference method requires 53 s per eigenfrequency found,
thus our method is a factor 33 times faster.
Assuming constant absolute eigenfrequency error is acceptable, then
this speed-up factor grows (in $d=2$) in proportion to $O(N) = O(k)$:
the reference method takes $O(N^3)$ effort per eigenfrequency found
whereas our proposed method takes only $O(N^2)$ effort.


\subsection{Reconstructing eigenfunctions from boundary data}
\label{s:recon}

We assume for now that for
each eigenvalue $\bstar$ of $\ntd(\kstar)$ we 
can generate an
accurate approximation to its corresponding boundary eigenfunction
$\fstar$ (e.g.\ as in section~\ref{s:bie}).
Approximations $\hat\phi$
to Dirichlet eigenfunctions $\phi_j$ can then be evaluated
using potential theory, as follows.

At wavenumber $k$, the free space
Green's function for the Helmholtz equation, $G_0(k; x, y)$, is defined as the
unique 
radiating solution to
$-(\Delta + k^2) G_0 = \delta$ in $\RR^d$, where $\delta$ is
the Dirac delta distribution. Specifically,  we have,
\be
G_0(k; x, y) = \frac{i}{4} \biggl(\frac{k}{2\pi |x-y|}\biggr)^{d/2-1}
\!\!H_{d/2-1}^{(1)}(k|x-y|),
\qquad x, y \in\RR^d,
\ee
where $H_\nu^{(1)}$ is the
outgoing Hankel function of order $\nu$ \cite[Ch.~10]{dlmf}. 
The standard single- and double-layer potentials \cite{CK83} are then defined
for $\xx\in\Omega$ by
\bea
(\Sc(k)\sigma)(\xx) &=& \int_\pO G_0(k;x,y) \sigma(\yy) ds_\yy ~,
\label{e:s}
\\
(\Dc(k)\tau)(\xx) &=& \int_\pO
\frac{\partial G_0(k; \xx,\yy)}{\partial n_\yy} \tau(\yy) ds_\yy ~,
\label{e:d}
\eea
where the
derivative is with respect to the $y$ variable in the outward
surface normal direction at $\yy$.
Then any solution $u$ to $(\Delta + k^2)u=0$ in $\Omega$
with smooth boundary may be written via Green's representation theorem
\cite{CK83},
\be
u = \Sc(k) u_n - \Dc(k) u|_\pO \quad \mbox{ in } \Omega ~.
\label{e:grf}
\ee

Suppose an exact eigenfrequency $k_j$ were known,
and also $f_j = f(k_j)$ the corresponding exact
eigenfunction of $\ntd(k_j)$.
We could then use \eqref{e:grf} to compute the extended eigenfunction $u_j$,
since its Dirichlet data vanishes, and its Neumann data
is given by \eqref{e:unbc}.
According to \eqref{e:rellich} we also need a
prefactor $\phi_j = \sqrt{2} k_j\, u_j$ to recover unit $\lo$ norm.
Thus a Dirichlet eigenfunction $\phi_j$ is represented exactly
throughout $\Omega$ by
\be
\phi_j = \sqrt{2} k_j\, \Sc(k_j) [\xni f_j] ~.
\label{e:phijrec}
\ee
However, we do not have access to $f_j$; we only have $\fstar$,
the corresponding eigenfunction of $\ntd(\kstar)$ for $\kstar$ near
$k_j$.
Similarly, $k_j$ is only known approximately (e.g. as in the previous section).
Given approximations $\hat{k}\approx k_j$ and $\hat{f}\approx f_j$,
we propose to reconstruct an approximate eigenfunction via
\be
\hat{\phi} = \sqrt{2} \hat{k}\, \Sc(\hat{k}) [\xni \hat{f}] ~.
\label{e:phirec}
\ee

\bfi   
\ig{width=0.5\textwidth}{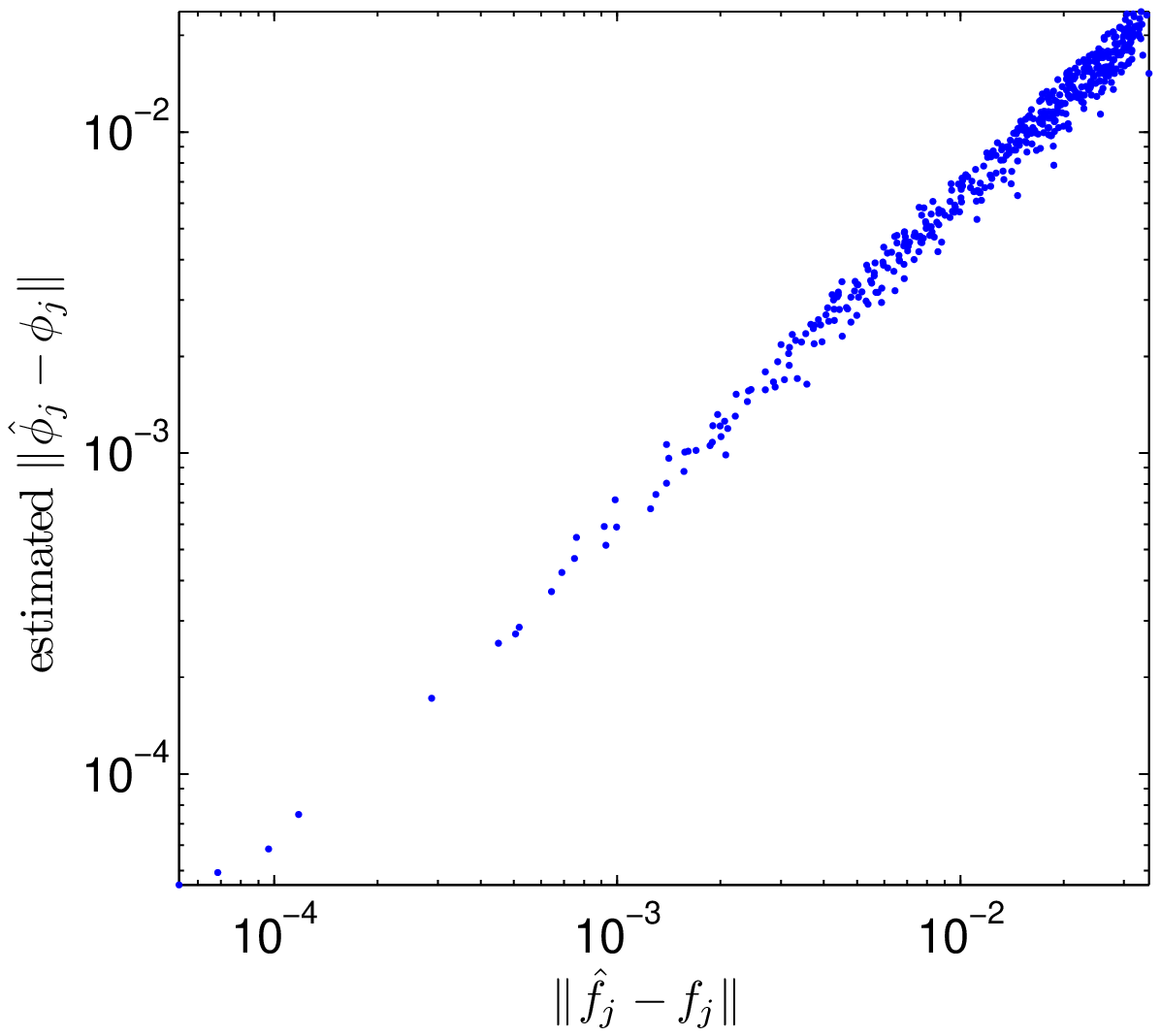}
\ca{Scatter plot of estimated eigenfunction errors in the $\lo$ norm
versus corresponding boundary function $f$ errors in the weighted $\lpo$ norm,
for the domain of Fig.~\ref{f:flow} and $k_j \in [90,100]$ with
the basic method.
$\lo$ norms are estimated on a Cartesian grid of 327 interior points,
giving a statistical error of order $\pm10\%$.
}{f:intnrm}
\efi

For now we will present a method that is only first-order in $\ep$:
we use the
\be
\mbox{``trivial $f$ estimator'':} \qquad \hat{f} =\fstar ~.
\label{fhatf}   
\ee
(We present higher-order methods in
section~\ref{s:higher}.)
Figure~\ref{f:errl}(b) shows the resulting $\|\hat{f}-f_j\|$
errors in the weighted $\lpo$ norm,
computed relative to a highly-accurate set of boundary functions $f_j$
found by the method of App.~\ref{a:ref}.
This behavior is clearly first order.
\brmk \label{r:ferr}
To prove a rigorous estimate on $\|\hat{f}-f_j\|$
one would need to control $\dk{f}$ over the interval $[\kstar,k_j]$;
we have by \eqref{12derivs} and Lemma~\ref{lem:utu} that $\|\dk{f}\|=O(1)$ at $k=k_j$,
but cannot exclude the possibility that ``avoided crossings'' in the
spectral flow cause $\dk{f}$ to be much larger at other $k$ values.
Based on empirical observations, the latter possibility seems very rare.
\ermk

How do the errors in $\hat f$ propagate to errors in eigenfunctions $\hat\phi$?
To test this, we insert
$\hat{f} = \fstar$, and $\hat{k}$ from \eqref{e:khatl},
into the reconstruction formula \eqref{e:phirec},
and estimate numerically the $\lo$ errors against
an accurate set of reference eigenfunctions $\phi_j$.
In the resulting Fig.~\ref{f:intnrm}
the data clusters close to a straight line of unit slope
(scatter from this line being part due to our
estimation of $\lo$ errors using a relatively small number of
interior points).
Hence the domain error norm of $\phi$ is empirically
controlled by the boundary error norm of $f$.
This is to be expected because, although
\eqref{e:phijrec} and \eqref{e:phirec} use different
wavenumbers, the error in $\hat k$ is of higher order than
that of $\hat f$, and error induced by the $k$-dependence of $\Sc(k)$
is expected to be negligible.

\brmk \label{r:L2err}
Supported by the above evidence, we
henceforth discuss eigenfunction errors only in terms of
boundary functions $f$,
postponing analysis of $\|\hat{\phi}-\phi_j\|_{\lo}$ to future work.
A rigorous proof that boundary error controls domain error
would demand bounds on the $k$-dependence of the
operator $\Sc(k):\lpo\to\lo$.
Accurate numerical study of $\lo$ errors is also difficult,
since i) the eigenmodes are highly oscillatory, demanding
$O(k^2)$ evaluation points (in the above example
around $10^5$ would be needed), and
ii) accurate evaluation of a layer potential such as \eqref{e:phirec}
near $\pO$ is difficult and a topic of current research
\cite{helsing_close}.
\ermk

In terms of computational effort,
extracting all boundary eigenfunctions $\fstar$ at each $\kstar$
is best done
by complete diagonalization of a matrix (given below by \eqref{e:Retamat})
at each $\kstar$;
this increases the CPU time per mode from the 1.6 s of the previous section
(when only matrix eigenvalues were needed)
to around 2.3 s per mode. However, the reference method of App.~\ref{a:ref}
also requires longer to
extract modes (an additional 14 s per mode). The net effect is that the
proposed NtD method is still 30 times faster than the reference
method.

\brmk
In \cite{bnds} we proved error bounds on $\hat{k}$ and on the $\lo$
error of $\hat{\phi}$ in terms of $\|\hat{\phi}\|_\lpo$.
The latter could be evaluated using \eqref{e:phirec} and a singular quadrature
scheme as in Section~\ref{s:bie}.
This would remove any need to compare against reference eigenpairs.
However, we avoided this approach since the errors in $f$
would dwarf the higher-order 
errors in $k$ that we wish to study.
\ermk

\bfi   
\ig{width=\textwidth}{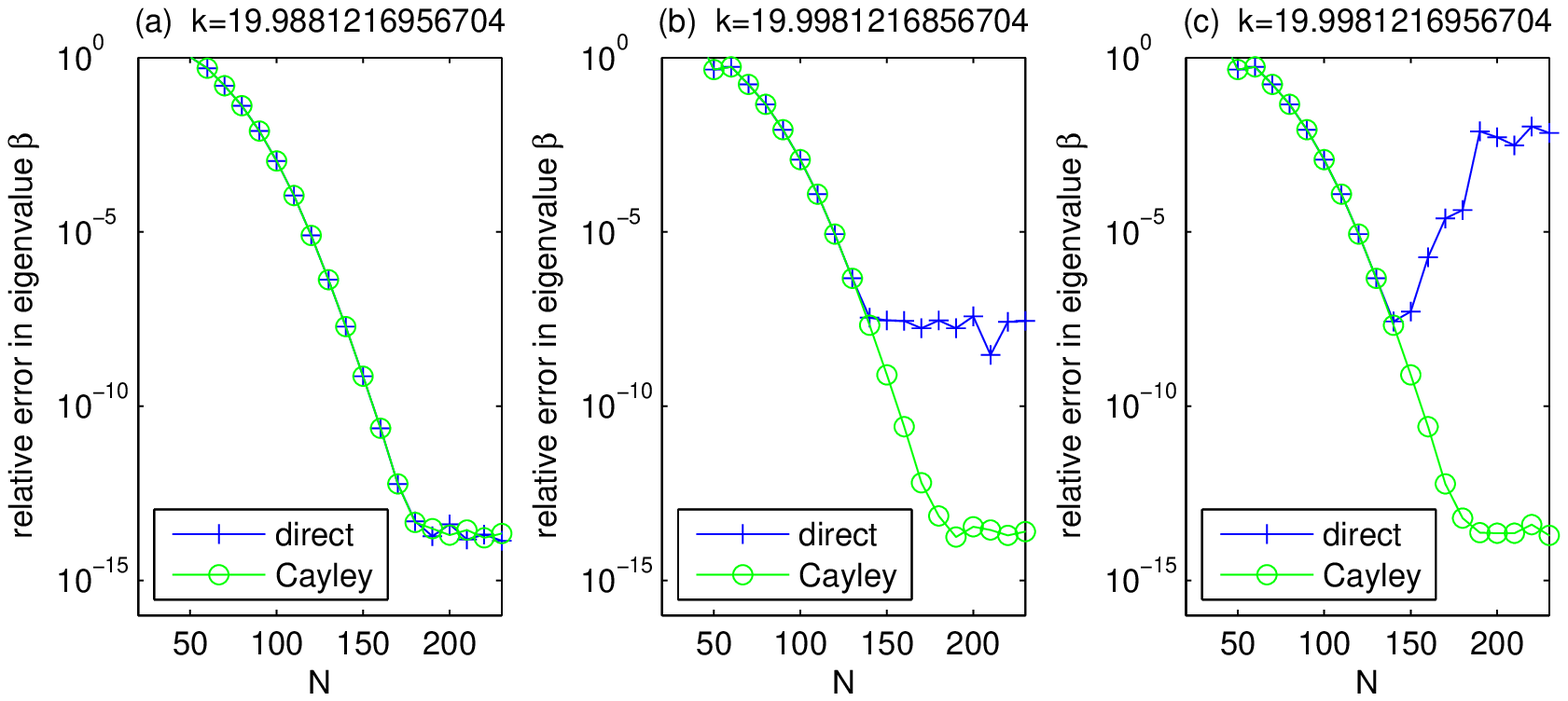}
\ca{Convergence of numerical scheme for eigenvalues $\beta$ of $\ntd(k)$
at three different $k$ values with distances from a Neumann eigenfrequency
as follows: (a) $10^{-2}$, (b) $10^{-10}$, (c) zero.
The domain is as in Fig.~\ref{f:flow}.
The $\beta$ tested was the largest negative eigenvalue (i.e. closest to zero),
roughly $-5\times 10^{-3}$ in each case.
The proposed Cayley scheme \eqref{e:rbie} and \eqref{e:eigmap} is
compared against the direct discretization of \eqref{e:ntdbie}.
}{f:convspec}
\efi

\subsection{Numerical computation of spectrum of $\ntd$}
\label{s:bie}

To implement the above algorithm,
at any given frequency $k$ we need to compute numerical approximations
to an $O(1)$ fraction of the eigenpairs of the weighted NtD
operator $\ntd(k)$.
Here we present, and test, a
robust integral equation method based upon the Cayley transform.
We need some standard results from potential theory \cite{CK83}.
Let $S(k)$ and $D(k)$ be the single- and double-layer
boundary integral operators formed by restricting \eqref{e:s} and \eqref{e:d}
respectively to the boundary,
\bea
(S(k)\sigma)(\xx) &=& \int_\pO G_0(k;x,y) \sigma(\yy) ds_\yy ~,
\qquad \xx\in\pO
\label{e:S}
\\
(D(k)\tau)(\xx) &=& \int_\pO
\frac{\partial G_0(k; \xx,\yy)}{\partial n_\yy} \tau(\yy) ds_\yy ~.
\qquad \xx\in\pO
\label{e:D}
\eea
Then, taking $\xx\in\Omega$ to the boundary in the representation formula
\eqref{e:grf}, and applying the jump relation for the double layer potential,
\begin{equation}
(\Dc \tau )\big|_{\dOmega} = (D  - \half)\tau~,
\label{e:jump}
\end{equation}
gives, for any Helmholtz solution $(\Delta + k^2)u=0$ in $\Omega$,
the boundary data relation
\be
(\half + D) u|_\pO = S u_n ~.
\label{e:grfbdry}
\ee
We also record for later use the jump relation for the single
layer potential,
\begin{equation}
\dn{(\Sc \sigma )} = (D^t  + \half)\sigma~. 
\label{e:jump-tr}
\end{equation}

We now generalize the Cayley transform \eqref{e:cayley1} slightly, defining
\be
R_\eta(k) := \big( \eta\ntd(k) - i \big) \big( \eta\ntd(k) + i \big)^{-1}~,
\label{e:cayley}
\ee
where $\eta\in \RR \setminus\{0\}$
is a scale parameter with units of inverse length
(i.e.\ of $k$).
Therefore, if $R_\eta(k) f = g$, we have 
$$
\big( \eta \ntd + i \big) g = \big( \eta\ntd - i \big) f,
$$
which we rearrange to
$$
\ntd [i \eta(g - f)] = g + f ~.
$$
That is, there exists a function $u$ on $\Omega$ such that 
\begin{equation}
\big( \Delta + k^2) u = 0, \quad u_n = i \eta \xni (g - f), 
\quad u |_\pO = g + f.
\label{e:cayleyu}
\end{equation}
Inserting this boundary data into \eqref{e:grfbdry} implies
$$
i\eta S [\xni (g-f)] = (\half + D)(g+f)
$$
which can be rearranged, recalling that $R_\eta f = g$ for all $f$,
to show, 
\be
R_\eta = ({K}_-)^{-1} {K}_+, \quad {K}_\pm = \pm(\half + D) + i\eta S \circ \xni ~.
\label{e:rbie}
\ee
The scheme is now to choose the scale parameter
(we prefer $\eta=k$, similar to \cite{kress91}),
and to approximate the spectrum and eigenfunctions
of $R_\eta$, using
known efficient Nystr\"{o}m
discretizations for the 
operators $S$ and $D$, as described below.
We then convert back to eigenpairs of $\ntd$ as follows:
the eigenvalues $\beta$ of $\ntd$ come from the eigenvalues
$\lambda$ of $R_\eta$ simply by inverting the formula \eqref{e:cayley},
that is,
\be
\beta = \frac{i}{\eta} \frac{1+\lambda}{1-\lambda} ~,
\label{e:eigmap}
\ee
and the eigenfunctions of $\ntd$ are the same as those of $R_\eta$.

\brmk  
The advantage of discretizing \eqref{e:rbie} then transforming eigenvalues
via \eqref{e:eigmap}, over
discretizing the usual {\em direct} representation of the weighted NtD map
\be
\ntd = (\half+D)^{-1} S \circ \xni
\label{e:ntdbie}
\ee
which follows from \eqref{e:grfbdry},
is that $R_\eta(k)$ is unitary and thus its eigenvalues
remain of size $O(1)$.
By contrast, the eigenvalues of $\ntd(k)$ have a large dynamic
range, and a finite number of eigenvalues diverge to infinity
whenever $k$ is a Neumann eigenfrequency of the domain,
causing inevitable large round-off error in the desired (small) eigenvalues.
We demonstrate this contrast numerically in Fig.~\ref{f:convspec}:
in the `direct' method this round-off error limits accuracy in $\beta$
to 8 digits for $k$ near a Neumann eigenfrequency
(and fails to converge at a Neumann eigenfrequency),
whereas the `Cayley' method achieves 
14-digit accuracy uniformly in $k$.
(Note that we expect some mild loss of accuracy as $k$ increases,
due to the condition number of the $({K}_-)^{-1}$ factor, but
in the range explored in this paper, $1<k<10^3$, this was negligible.)
Thus to discretize \eqref{e:ntdbie} is not robust, whereas our
proposed scheme is robust.
\ermk   

We summarize briefly our preferred Nystr\"{o}m discretization
for Helmholtz layer potential operators on analytic curves in
$d=2$, following Kress \cite{kress91}.
Let $\zz:[0,2\pi)\to\RR^2$ be a $2\pi$-periodic parametrization
of $\pO$, and let $k(x,y)$ be the kernel of either $S$ or $D$.
Changing variable to $s,t \in [0,2\pi)$ we get kernel
$K(s,t) := k(\zz(s),\zz(t))|\zz'(t)|$ where $\zz'= d\zz/dt$.
Note that $S$ has a logarithmic singularity on its diagonal, whereas
$D$ has a continuous kernel but is non-analytic on the diagonal;
in both cases the following splitting allows spectral accuracy to be
achieved.
We choose quadrature nodes $t_j = 2\pi j/N$, $j=1,\ldots,N$,
and 
split the kernel into the form
\be
K(s,t) = \log\left(4 \sin^2 \frac{s-t}{2}\right) K_1(s,t) + K_2(s,t)
\label{e:Kst}
\ee
with $K_1$ and $K_2$ both $2\pi$-periodic and analytic.
The matrix representation of $K_2$ comes from the periodic trapezoid rule
(weights being constant at $2\pi/N$), whereas the representation of $K_1$
involves a product quadrature appropriate
for the periodized log singularity. Together these give a matrix $\mbf{K}$
with elements
\be
\mbf{K}_{ij} = \frac{2\pi}{N} \left[R^{(N)}_{|i-j|}(0)\,K_1(t_i,t_j) + 
K_2(t_i,t_j)\right], \quad i,j = 1,\ldots, N ~,
\label{e:Kmat}
\ee
where the Martensen--Kussmaul quadrature weights 
(deriving from the Fourier series for
the log factor, see \cite[Lemma 8.21]{LIE}) are defined by
\be
R^{(N)}_j(s) \;=\; -\sum_{m=1}^{N/2-1}\frac{2}{m}
\cos m(s-t_j)\;-\;\frac{2}{N}\cos \frac{N}{2}(s-t_j)~.
\ee

Abusing notation slightly by letting $K$ be an operator with kernel $K(s,t)$,
it is standard to approximate operator eigenvalue problems of the type
\be
K \phi = \lambda \phi
\label{e:Keig}
\ee
by the $N$-dimensional matrix eigenvalue problem
\be
\mbf{K} \phi^{(N)} = \lambda^{(N)} \phi^{(N)}~.
\label{e:Kmateig}
\ee
If $K$ were compact and normal, it is known that
the spectrum and eigenspaces of \eqref{e:Kmateig} converge
to those of \eqref{e:Keig} as $N\to\infty$ \cite{atkinspec},
at a rate given by the error of the quadrature scheme applied
to vectors in the eigenspace (for the spectrum
see \cite{atkinrate}---here normality ensures that the index of each eigenvalue is one---and for the eigenspaces see \cite[Thm.~1]{osbornrate}).
This analysis relies on the framework of
collectively compact operators \cite{anselone} \cite[Ch.~10]{LIE}.
The above product quadrature scheme is within this framework
and is spectrally accurate
for analytic functions, i.e.\ errors are bounded
by $c e^{-\gamma N}$ for some $\gamma>0$ \cite{kress91,LIE}.

However, our goal is to approximate the spectrum and eigenspaces
of the operator $R_\eta$; 
this is not covered by the above-mentioned analytic results,
for two reasons.
Firstly $R_\eta$ is not compact (although $R_\eta+I$ is),
and secondly
the application of $R_\eta$ in \eqref{e:rbie} requires an operator
product and inverse.
We will not attempt a rigorous error analysis here, rather
merely describe our scheme and show its efficacy.
We approximate the spectrum of $R_\eta$ by that of the matrix
\be
\mbf{R}_\eta = (\mbf{K}_-)^{-1} \mbf{K}_+
\label{e:Retamat}
\ee
built from the matrices $\mbf{K}_\pm$ which approximate the operator
factors $K_\pm$ appearing in \eqref{e:rbie},
according to the above Nystr\"om scheme \eqref{e:Kmat}.
Dense linear algebra is used both for the matrix inverse $\mbf{K}_-^{-1}$ in \eqref{e:rbie},
and the full diagonalization of $\mbf{R}_\eta$
(MATLAB's {\tt inv} and {\tt eig} respectively).
The computational effort is $O(N^3)$.
The eigenvectors of $\mbf{R}_\eta$ then give approximations to
the eigenvectors of $R_\eta$, hence of $\ntd$, at the quadrature nodes.
In \mpspack\ the above algorithm is available via
\begin{verbatim}
[beta,V] = p.NtDspectrum(k);
\end{verbatim}
which returns approximate eigenvalues of $\ntd(k)$ in {\tt beta},
and corresponding
eigenfunction values at the quadrature points in the columns of {\tt V}.
Returning to Fig.~\ref{f:convspec} we observe exponential
convergence of this `Cayley' scheme, with saturation
at relative error $10^{-14}$ uniformly in $k$.




\section{Error analysis of the linear eigenfrequency estimator}
\label{s:kerr}

The main result of this section is an estimate on the
accuracy of the eigenfrequencies as reconstructed by the basic formula
\eqref{e:khatl}.  In section~\ref{s:khigher} we will describe an
improved method for which we can prove better error estimates, but
those better estimates are conditional on absence of spectral
concentration (Assumption~\ref{ASC}); here, the result is
unconditional.

The key result is the following bound on the deviation from
linearity of the weighted NtD eigenvalue flow.
\begin{theorem}\label{thm:linear-beta}
There are constants $\epsilon , K > 0$ dependent only on $\Omega$
such that the following holds.
Let $k_p > K$ be a Dirichlet eigenfrequency, and $\beta(k)$ be the corresponding
eigenvalue branch of $\ntd(k)$, i.e.\ such that $\beta(k_p) = 0$ (the existence of which is guaranteed by Proposition~\ref{prop:eig-branches}).  Then
\be
\beta(k) = \frac{k-k_p}{k_p} + O\Big( \frac{(k-k_p)^2}{k_p^2}  + \frac{(k-k_p)^3}{k_p} \Big) 
\label{betaintermsofk}
\ee
for all $k \in [k_p-\ep, k_p]$. The implied constant
in the $O(\cdot)$ depends only on $\Omega$.
\label{t:khatl}
\end{theorem}

It is then easy to derive
the following error estimate for the basic method.
Note that we have already numerical evidence (section
\ref{s:khat}) that the powers of $\ep$ are sharp.
\begin{corollary}\label{cor:linear-beta}
There are constants $\epsilon, C > 0$ depending only on $\Omega$
such that, for any sufficiently large $\kstar$, and any $k_p$ lying in the range $[\kstar, \kstar+\ep]$, and $\beta(k)$ related to $k_p$ as in Theorem~\ref{thm:linear-beta},  we have
\be
| \hat k - k_p | \; \leq \;
C \Big(\frac{\epsilon^2}{\kstar} + \epsilon^3 \Big)
\label{kerror} 
\ee
where $\hat{k} = \kstar/(1+\beta(\kstar))$ is
the linear estimator of $k_p$ according to \eqref{e:khatl}. 
\label{p:khatl}
\end{corollary}

\begin{remark} Note that the theorem holds for a fixed  window width $\ep$,
independent of $k$.
By Weyl's Law there are $O(k^{d-1})$ eigenfrequencies lying in
such a window; all are found within the stated error.
As $k$ grows, the $\epsilon^2/k$ term becomes negligible for almost
all reconstructed eigenfrequencies, thus
the eigenfrequency error of the basic method is effectively $O(\ep^3)$
with constant independent of $k$.
\end{remark}

{\em Proof of Theorem~\ref{t:khatl}.}
We use the following identity from \cite[Lemma 3.1]{que}, which
allows one to express the right hand side of \eqref{e:dkbeta} in terms
of boundary data: for any
Helmholtz solution $u$ at frequency $k$, we have
\be
2k^2 \int_\Omega |u|^2 = \int_{\dOmega} \xn
\big( |u_n|^2 + k^2 |u|^2 - |\dt{u}|^2 \big)
+ \dn{u} \overline{\VV u} + (\VV u) \overline{\dn{u}}.
\label{Alex's-identity}
\ee
Here, $\dt{}$ is the tangential gradient on $\pO$, and
$\VV$ is the tangential part of the vector field $x \cdot \nabla$
which generates dilations.
Explicitly, $W = x_\tbox{tan}\cdot\dt{}$ where $x_\tbox{tan} = x - \xn n$.
For example, in
$d=2$ we have $\dt = \partial_t$ and $\VV=(x\cdot t) \partial_t$,
where $t$ is the unit tangent vector.
Putting \eqref{Alex's-identity} together
with \eqref{e:dkbeta}, taking $u$ to be the extended
eigenfunction, we obtain
\begin{equation}
\dk{\beta}(k) = \frac1{k} \bigg( \int_{\dOmega} \xn \big( |\dn{u}|^2 + k^2 |u|^2 - |\dt{u}|^2 \big) + 2 \Re (\dn{u} \overline{Wu}) \bigg). 
\label{e:dkbterms}\end{equation}
The principal term on the right hand side of \eqref{e:dkbterms} is
(using \ref{e:rellich}),
$$ \frac1{k} \int_{\dOmega} \xn |\dn{u}|^2  =  \frac1{k} \| f \|^2 =  \frac1{k} .
$$
The other terms are small when $\beta$ is small, and we try to estimate them in terms of $\beta$. Two of the terms are not hard to estimate: we have using the boundary condition \eqref{e:robin},
$$
\int_{\dOmega} \xn k^2 |u|^2 = O(\beta^2 k^2),
$$
while (using the divergence theorem on $\pO$ in the third step below),
\bea
\int_{\dOmega}2\Re ((\VV u) \overline{\dn{u}}) &=&
\int_{\dOmega} \xni \beta^{-1} \cdot 2 \Re ((\VV u) \ubar )\nonumber \\
&=& \frac1{\beta} \int_{\dOmega} \xni \VV(|u|^2) \nonumber \\
&=& - \frac1{\beta} \int_{\dOmega} |u|^2 (\VV + \div\VV)\xni \nonumber \\
&=& -\beta \int_{\dOmega} \big(\xn \VV (\xni) + \div\VV\big) \xn |\dn{u}|^2
\nonumber \\
&=& O(\beta). \nonumber
\eea
The scalar function $\div\VV$ may also be written $\dt{}\cdot x_\tbox{tan}$.
To deal with the $|\dt{u}|^2$ term in \eqref{e:dkbterms},
we prove the following in Appendix~\ref{a:AppB}:

\begin{lemma}\label{lem:utu} 
There are constants $c$, $K>0$ depending only on $\Omega$, such that
whenever $k \geq K$ and $u$ solves $(\Delta + k^2) u = 0$ in $\Omega$,
with 
\begin{equation}
u = \beta \xn \dn{u} \quad \text{ on }\dOmega
\label{robin-bc}\end{equation}
for some Robin constant $\beta \in [-c,0]$, then
\begin{equation}
\| \dt{u} \|_{L^2(\dOmega)} \leq 2k \| u \|_{L^2(\dOmega)}.
\label{u_t-in-terms-of-u}\end{equation}
\end{lemma}   

\begin{remark}\label{constant-remark} 
The intuition behind Lemma~\ref{lem:utu} is that $f$, as a boundary trace of a Helmholtz solution at frequency $k$, should be band-limited to frequencies $\leq k$. Indeed, the coefficient $2$ in \eqref{u_t-in-terms-of-u} could be replaced by any factor $\alpha$ strictly larger than $1$, for $k \geq K(\alpha)$.   Also, using the same proof is not hard to show the corresponding result for higher derivatives:
\begin{equation}
\| \nabla^{(m)}_{\tan} u \|_{L^2(\dOmega)} \leq 2k^m \| u \|_{L^2(\dOmega)}, \quad k \geq K_m
~.
\label{higherderivs}\end{equation}
\end{remark} 

Using Lemma~\ref{lem:utu}, we can estimate the $|\dt{u}|^2$ term in
\eqref{Alex's-identity} the same way as the $k^2 u^2$ term. So, combining
the estimates of terms in \eqref{e:dkbterms}, we get
\begin{equation}
\dk{\beta}(k) = \frac1{k} + O\Big(\frac{|\beta|}{k}\Big) + O(\beta^2 k)
\quad \mbox{whenever }-c\leq\beta\leq 0~, 
\label{betadot}\end{equation}
with implied constants depending only on $\Omega$. 

We now conclude the proof of Theorem~\ref{t:khatl} by establishing \eqref{betaintermsofk}. This follows directly from \eqref{betadot-est} below by integrating in $k$. Therefore, it remains to prove the following result:

\begin{lemma} 
There exists constants $\ep,C_1>0$ depending only on $\Omega$
such that, for any sufficiently large $k_p$ (with $\beta(k)$ as in
Theorem~\ref{thm:linear-beta}), it holds for all $k\in[k_p-\ep,k_p]$ that
\begin{equation}
\frac{2(k - k_p)}{k_p} \leq \beta(k) \leq \frac{k - k_p}{2k_p}~, \quad \text{and}
\label{squeeze}\end{equation}
\begin{equation}
\Big|  \dk{\beta}(k) - \frac1{k_p} \Big|
\leq
C_1 \Big( \frac{k_p - k}{k_p^2} + \frac{(k_p - k)^2}{k_p} \Big)
~. 
\label{betadot-est}\end{equation}
\end{lemma}

\begin{proof} We first prove the left hand side of \eqref{squeeze}. 
We first use the continuity of $\beta(k)$ to observe that in some  small left neighbourhood $(k'', k_p]$ of $k_p$, $\beta(k)$ itself is arbitrarily close to $0$ --- in particular, such that $\beta(k) \geq -c$. Therefore, \eqref{betadot} applies, and, by requiring $\beta$ sufficiently small we can make the right hand side of \eqref{betadot} less than $3/(2k)$, and hence less than $2/k_p$ (since $k/k_p$ can be made as close as we like to $1$). By integrating this, we find that in this neighbourhood we have 
the left hand inequality in \eqref{squeeze}. 
However, the size of the neighbourhood may still depend on $k_p$.

Now we prove that for some $\epsilon > 0$,
the left hand inequality in \eqref{squeeze} holds on each
interval $[k_p - \epsilon, k_p]$ for {\em all} sufficiently large $k_p$.
We do this by contradiction. Suppose that there is a $k \in [k_p - \epsilon, k_p]$ such that $\beta(k) < 2(k - k_p)/k_p$. Let $k'$ be the largest such element of the interval $[k_p - \epsilon, k_p]$; notice that $k'$ is strictly less than $k_p$ using the paragraph above. Then we have 
\begin{equation}
\beta(k) \geq 2\frac{k - k_p}{k_p} \text{ for } k \in [k', k_p], \quad
\beta(k') = 2\frac{k' - k_p}{k_p}. 
\label{eeest}\end{equation}
For small $\epsilon$ relative to $k$ this certainly implies that $\beta \geq -c$ on the interval $[k', k_p]$, so \eqref{betadot} applies. Using \eqref{betadot} and the estimate \eqref{eeest} we find that 
\begin{equation}\begin{aligned}
\dot \beta(k) &\leq \frac1{k} + C \Big( \frac{2(k_p - k)}{k k_p} + \frac{4k (k_p - k)^2 }{k_p^2} \Big)  , \quad k \in [k', k_p] \\
&\leq \frac1{k} \bigg( 1 + C \Big( \frac{2\epsilon}{k_p} + \frac{4\epsilon^2 k^2}{k_p^2} \Big) \bigg) \\
&\leq \frac{3}{2k} \\
& \leq \frac{1.6}{k_p},
\end{aligned}\label{betadot-eeest}\end{equation}
where we need $\epsilon$ sufficiently small in the second last line, and $k_p$ sufficiently large in the last.  
Integrating this we find that 
$$
\beta(k_p) - \beta(k') \leq \frac{1.6 (k_p - k')}{k_p},
$$
which contradicts the second part of \eqref{eeest}. We conclude that no such $k'$ exists, so the left hand inequality of \eqref{squeeze} holds on the whole interval $[k_p - \epsilon, k_p]$. 

The right hand inequality is proved similarly. In fact, using the left hand inequality, we see for small $\epsilon$ that $\beta \geq -c$ on the whole interval $[k_p - \epsilon, k_p]$, so we can use \eqref{betadot} on the whole interval, and conclude in a similar way to \eqref{betadot-eeest} that 
$$
\dot \beta(k) \geq \frac{0.8}{k_p}, \quad k \in [k_p - \epsilon, k_p]
$$
to derive the right hand inequality in \eqref{squeeze}. 

To prove \eqref{betadot-est}, we first note that \eqref{squeeze}
inserted into \eqref{betadot} implies that 
\begin{equation}
\Big|  \dk{\beta}(k) - \frac1{k} \Big|
\leq
C_1 \Big( \frac{k_p - k}{k k_p} + \frac{k(k_p - k)^2}{k_p^2} \Big), \quad k \in [k_p - \epsilon, k_p].  
\label{betadot-est-2}\end{equation}
This is almost the same as \eqref{betadot-est}, but there are factors of $k$ in place of $k_p$. For the left hand side, replacing $1/k$ by $1/k_p$ makes an error of $(k_p - k)/k k_p$, and this can be absorbed on the right hand side (by increasing $C_1$ by $1$).
Then, by taking $k_p$ large relative to $\epsilon$, we can replace the occurrences of $k$ on the right hand side by $k_p$, at the cost of increasing $C_1$ slightly.   We conclude that \eqref{betadot-est} holds. 
\end{proof}


\begin{remark} \label{r:robust}
Theorem~\ref{thm:linear-beta} and Corollary~\ref{cor:linear-beta} imply that every $k_p$ slightly bigger than $k_*$ corresponds to a slightly negative eigenvalue $\beta$ of $\Theta(k_*)$, with an almost linear relationship between $k_p$ and $\beta$. The converse is also true: every slightly negative eigenvalue $\beta$ of $\Theta(k_*)$ corresponds to a $k_p$ slightly bigger than $k_*$. To see this we note that \eqref{betadot}, and Proposition~\ref{prop:eig-branches}, imply that the eigenvalue branch starting at $\beta(k_*) \geq -\epsilon/k$ will, for small $\epsilon$, reach zero near $k \approx k_*/( 1 + \beta(k_*))$. Thus there is a one-to-one correspondence between eigenfrequencies $k_p$ slightly bigger than $k_*$, and the slightly negative eigenvalues of $\Theta(k_*)$. 
\end{remark}

\bfi   
\ig{width=0.48\textwidth}{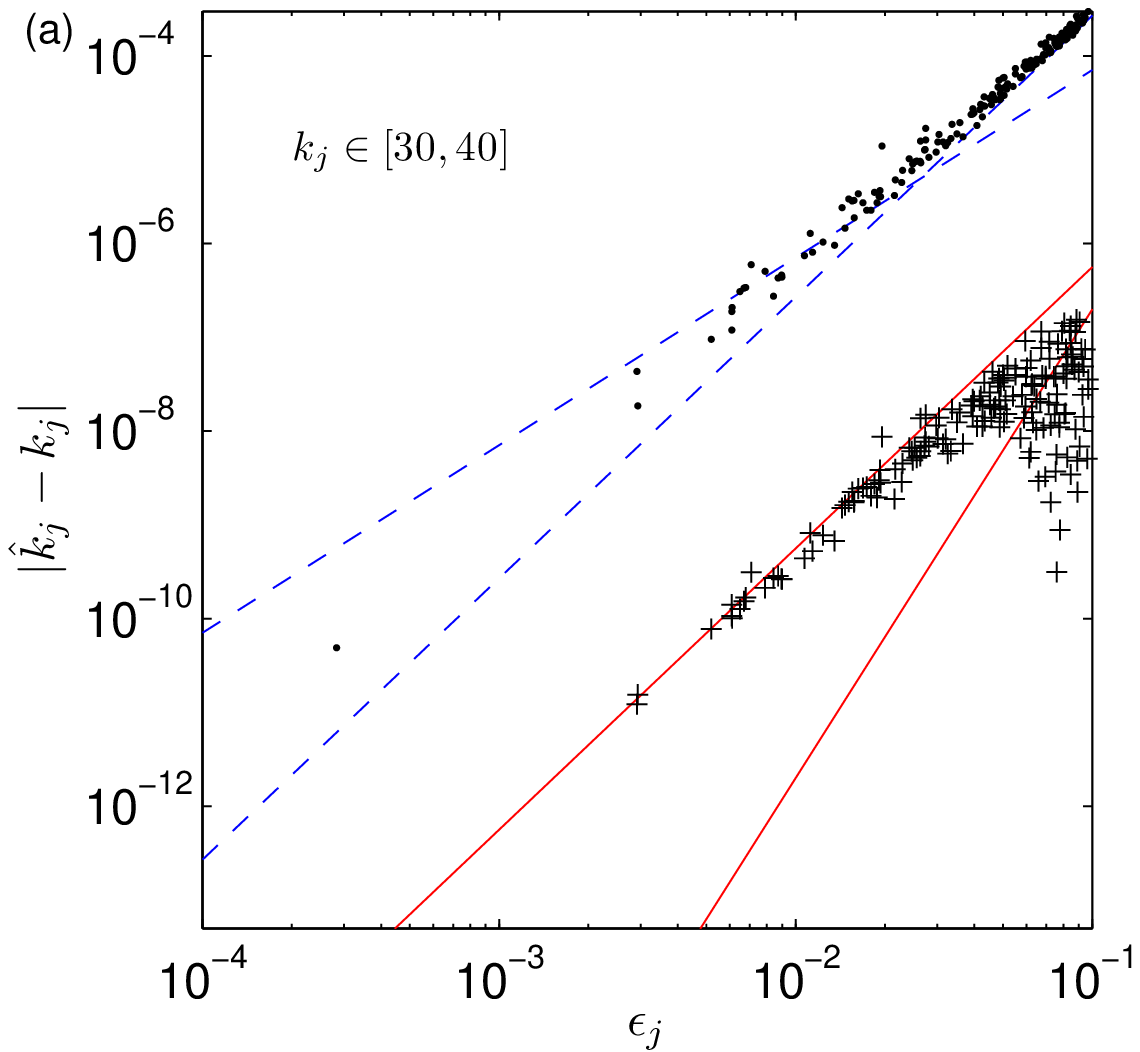}\quad
\ig{width=0.48\textwidth}{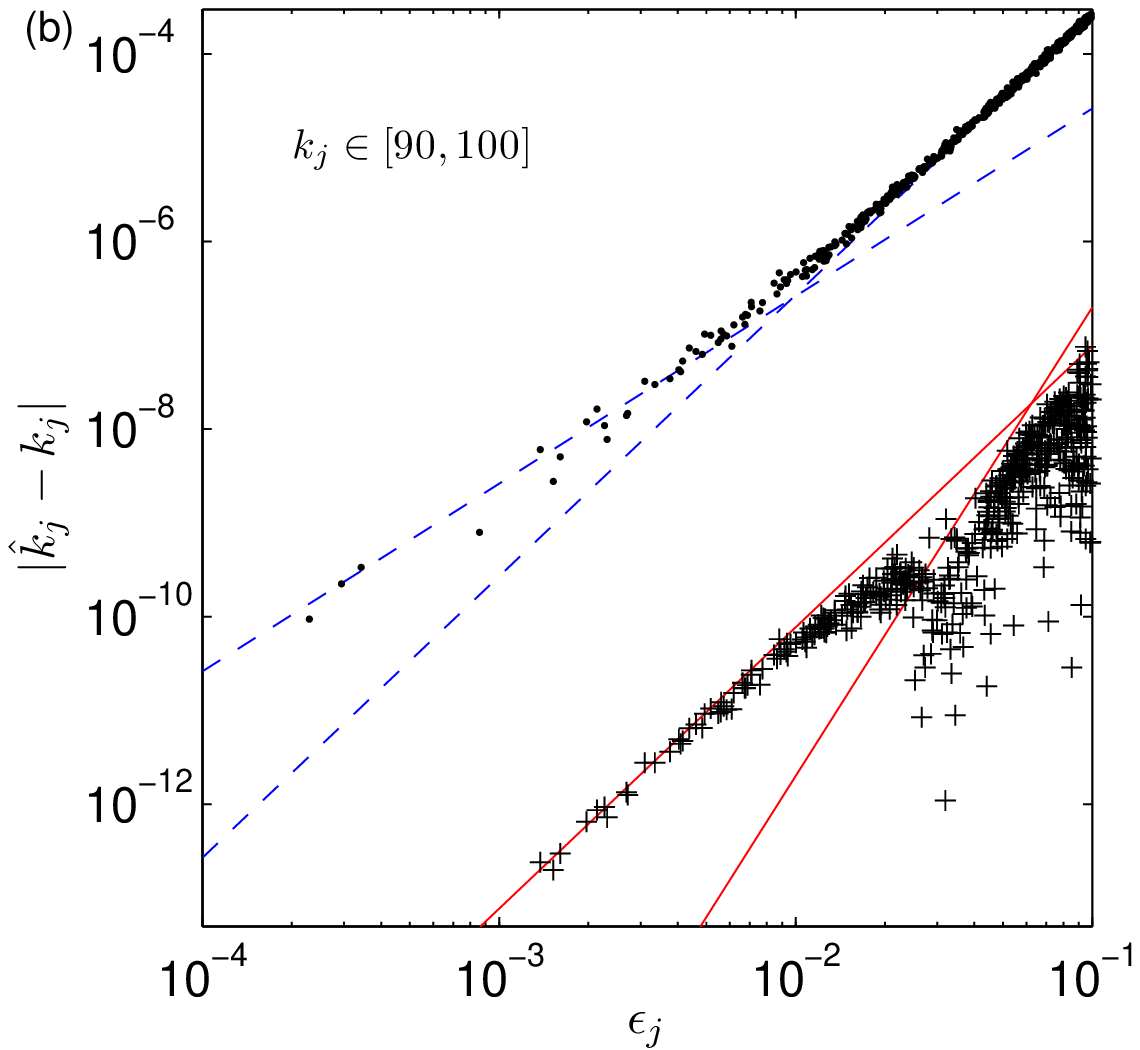}

\ig{width=0.48\textwidth}{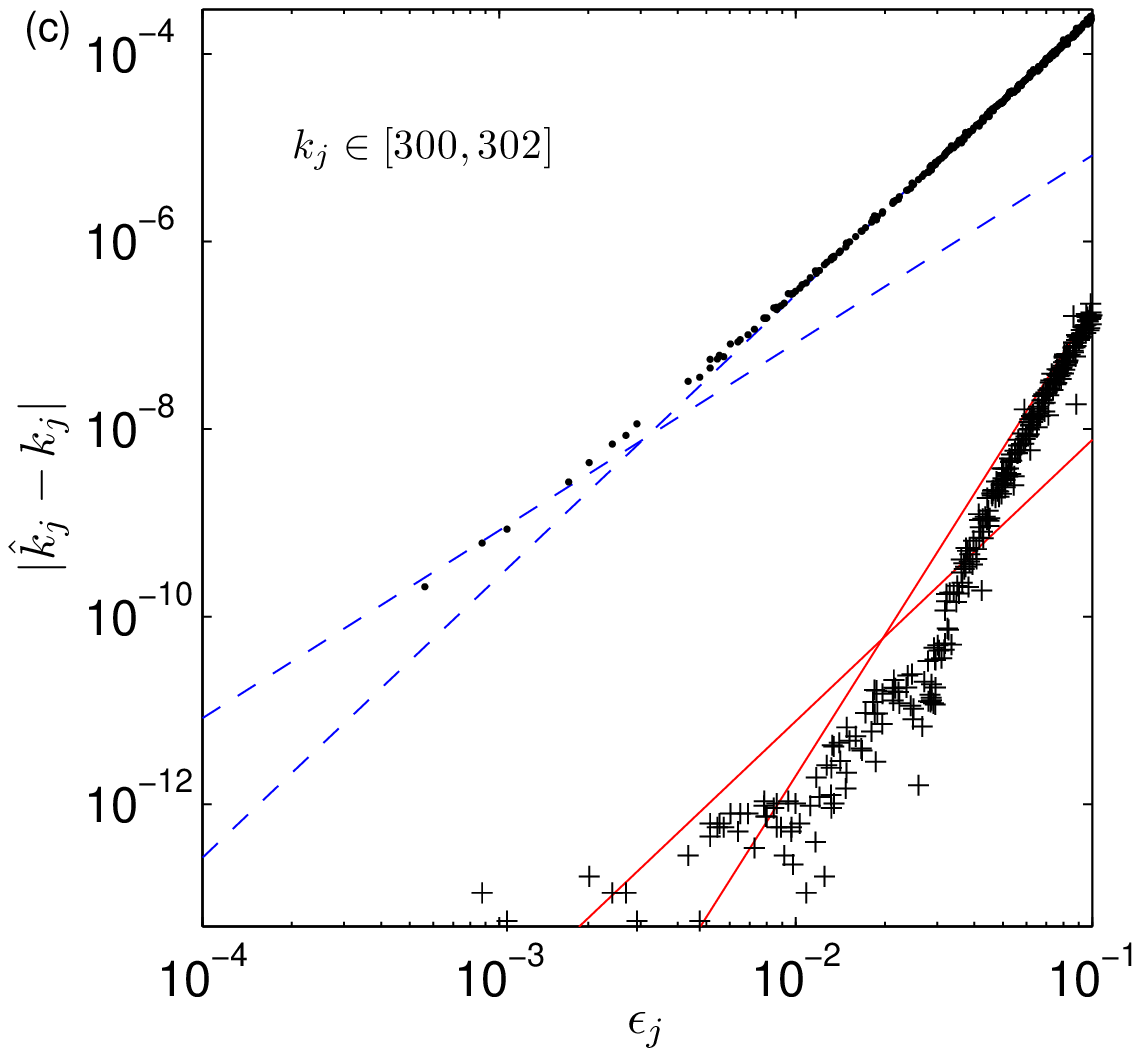}\quad
\ig{width=0.48\textwidth}{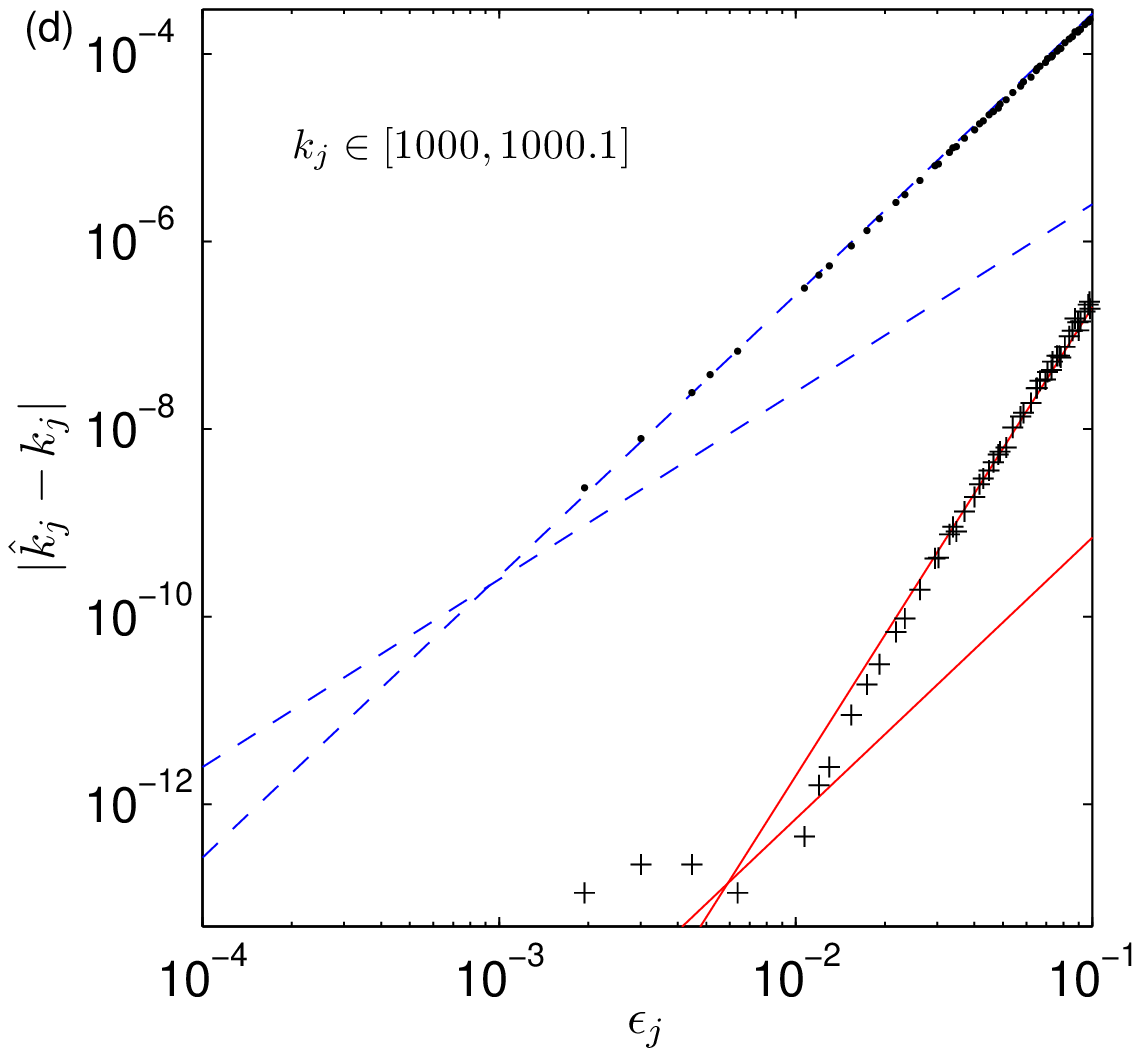}
\ca{Errors of the two eigenfrequency prediction schemes for four different
$k$ ranges, vs $\ep_j:=k_j-\kstar$. Linear scheme (dots),
using \eqref{e:khatl}, compared against power laws $0.27 \eps^3$
and $0.25 \eps^2/k$ (dotted lines).
Higher-order Riccati scheme (crosses), using \eqref{e:khatr},
compared against $0.02 \eps^5$ and $7 \eps^3/k^2$ (solid lines).
For the power laws, the mean $k$ in the interval is used.
The domain is as in Fig.~\ref{f:flow}(b).
}{f:kerr}
\efi

\section{Higher-order accurate reconstruction methods}
\label{s:higher}

\subsection{Higher-order eigenfrequency approximation}
\label{s:khigher}

In section~\ref{s:khat} we presented a formula \eqref{e:khatl} for
eigenfrequencies $k_p$.
For its error analysis in section~\ref{s:kerr} we treated all terms on the right hand side of \eqref{e:dkbterms}, other than the first, as error terms, and estimated them.
However, numerically we have at our disposal not just the eigenvalues of $\Theta(\kstar)$, but the eigenfunctions. Observe that the RHS of \eqref{e:dkbterms} can be expressed \emph{exactly} in terms of $\beta$ and its associated eigenfunction $f$, using the relation $u |_{\dOmega} = \beta f$ from \eqref{e:ubc}.
Precisely, we have
\begin{equation}
k^2 \int_{\pO} \xn |u|^2 = k^2 \beta^2 \| \xn f \|_{\dOmega}^2, 
\label{u1}\end{equation}
\begin{equation}
\int_{\pO} \xn |\nabla_{\tan} u|^2 = \beta^2 \| \xn \nabla_{\tan} f \|_{\dOmega}^2, 
\label{utwo}\end{equation}
\begin{equation}
\int_{\pO} 2 \xn \Re (\dn{u} \overline{Wu}) = -\beta \int_{\dOmega} (W + \operatorname{div} W)(\xni) |f|^2 = -\beta \langle f, mf \rangle
~,
\label{u3}\end{equation}
where we introduce the scalar boundary function
\be
m := \xn W(\xni) + \operatorname{div} W. 
\label{e:m}
\ee
Using the above, we can rewrite \eqref{e:dkbterms} as
\begin{equation}
\dot \beta = \frac1{k} + k \beta^2 \| \xn f \|_{\dOmega}^2 - \frac{\beta^2}{k} \| \xn \nabla_{\tan} f \|_{\dOmega}^2 - \frac{\beta}{k} \langle f, mf \rangle.
\label{betadot2}\end{equation}

Of course, for values of $k$ other than $\kstar$ we no longer know the exact boundary eigenfunction $f(k)$. However, we can get a potentially more accurate estimate of the function $\beta(k)$ by ``freezing'' the values of the norms in \eqref{u1}-\eqref{u3} by fixing $f=f(\kstar)$.
That is, we define constants
\begin{equation}
A := k_z^2 \| \xn f \|_{\dOmega}^2 -  \| \xn \nabla_{\tan} f \|_{\dOmega}^2,
\qquad B := -\langle f, mf \rangle~,
\label{ABC}\end{equation}
where $k_z$ is a frozen value of $k$ yet to be specified, and consider the ODE 
\begin{equation}
\dot \beta = \frac1{k} \Big(1  + A \beta^2 +  B \beta \Big)~.
\label{ABODE}\end{equation}
After changing independent variable to $\log k$, this is a constant-coefficient Riccati equation that can be solved exactly. Assuming that $A > (B/2)^2$ which is expected (cf.  Remark~\ref{constant-remark}; note that $B = O(1)$ as $k \to \infty$), the general solution is 
$$
\beta(k) = \frac{B}{2A} + \frac{\mu}{A} \tan (\mu \log k + \alpha), \qquad
\mbox{where }
\mu = \sqrt{A - (B/2)^2},
$$
where $\alpha$ is an arbitrary constant chosen so that the initial condition $\beta(\kstar) = \bstar$ is satisfied. Solving for $k_p$ gives the
\begin{equation}
\mbox{``Riccati estimator'':}
\quad \hat k = \kstar \exp \frac1{\mu} \bigg(   \tan^{-1} \big( \frac{B}{2\mu} - \frac{A \bstar}{\mu}   \big) - \tan^{-1} \big( \frac{B}{2\mu} \big) \bigg).
\label{e:khatr}\end{equation}

Figure \ref{f:kerr} shows the observed errors for this
Riccati estimator in $d=2$ (in our code example this is achieved via
option {\tt o.khat = 'r'});
they are $10^3$ to $10^5$ times better than those shown for
the linear estimator \eqref{e:khatl}.
We in fact compared the constant choice $k_z=\kstar$ against
$k_z=\half(1+(1+\bstar)^{-1})\kstar$, the mean of
$\kstar$ and the linear estimator $\hat k$,
and found that the latter choice has slightly smaller errors, hence prefer it.
We study four frequency ranges, so that the $k$ behavior
of the constants in the $\ep$ power laws becomes visible.
This provides strong evidence that the error of the Riccati scheme
is $O(\ep^3/k^2 + \ep^5)$, which is dominated by the second term when
the window $\ep$ is chosen to be large enough to collect many eigenfrequencies
(i.e.\ $>k^{-1}$). Note that the constant in $O(\ep^5)$ is
independent of $k$, and appears quite small, thus absolute $\hat k_j$ errors
are around $10^{-7}$ for $\ep=0.1$.
In Section~\ref{s:efn-error-analysis}, we shall give a theoretical analysis
of this method (with the choice $k_z=\kstar$),
under a spectral nonconcentration assumption (see Assumption~\ref{ASC})
for $\Theta(\kstar)$ at $\bstar$.

\begin{table}[t]  
\small
\begin{tabular}{l|r|r|r|rrr|rr|rr|}
\multicolumn{4}{l|}{} &
\multicolumn{3}{|l|}{time / mode (sec)} &
\multicolumn{2}{|c|}{abs error of $\hat k_j$} &
\multicolumn{2}{|c|}{$L^2$-error of $\hat f_j$} 
\\
\hline
$k$ interval & $j$ & $N$ & $n_m$ & ref &
NtD & ratio & max & median & max & median \\
\hline
&&&&&&&&& \\[-2ex]   
$[30,40]$ & 4e2 & 300 & 176 & 8.1 & 0.72 & 11 &
  1.5e$-7$ & 1.3e$-8$ & 1.6e$-3$ & 1.5e$-4$
\\
$[90,100]$ & 2.6e3 & 720 & 492 & 67 & 2.3 & 30 &
  8e$-8$ & 1.2e$-9$ & 3e$-3$ & 1.2e$-4$
\\
$[300,302]$ & 2.3e4 & 2200 & 314 & 1200 & 15 & 80 &
  2e$-7$ & 3e$-9$ & 7e$-3$ & 3e$-4$
\\
$[1000,1000.1]$ & 2.6e5 & 7200 & 53 & 3e4$^\ast$ & 134 & 250$^\ast$ &
  2e$-7$  & 6e$-9$ & 1.1e$-2$ & 6e$-4$
\\[1ex]
\hline
\end{tabular}
\vspace{2ex}
\caption{Runtimes (per mode found)
and errors for proposed NtD method
for eigenmodes of the nonsymmetric domain of Fig.~\ref{f:flow}(b).
The number of modes found in each frequency interval is $n_m$,
and the approximate mode number is $j$.
In all cases $\ep=0.1$; note that error can of course be reduced by
reducing $\ep$.
The Riccati \eqref{e:khatr}
and quadratic \eqref{fhatimpsecondorder} estimators were used.
In the reference method (App.~\ref{a:ref})
the absolute $k_j$ errors were better than $10^{-12}$.
Asterisk ($\ast$) indicates estimated values; in fact a faster method
(that we shall not describe here) was used for the highest reference set.
}\label{t:rfn}
\end{table} 

Table~\ref{t:rfn} summarizes the numerical experiments:
note that
the speed-up ratio relative to the reference method is roughly linear in $k$,
reaching a couple of hundred for
our largest calculation (around 400 wavelengths across).
Thus the speed-up is close to the number of wavelengths across the domain,
for the errors reported.

\begin{remark}
We have tested the above Riccati estimator against a more accurate approximation
which considers a linear approximation in $k$ for the
quantities $ \| {\xn} f \|_{\dOmega}^2$, \eqref{utwo} and \eqref{u3},
and solves \eqref{ABODE} with $k$ varying (this requires a numerical ODE
solver).
We found no significant improvement, hence recommend \eqref{e:khatr}.
\end{remark}

\bfi   
\ig{width=0.45\textwidth}{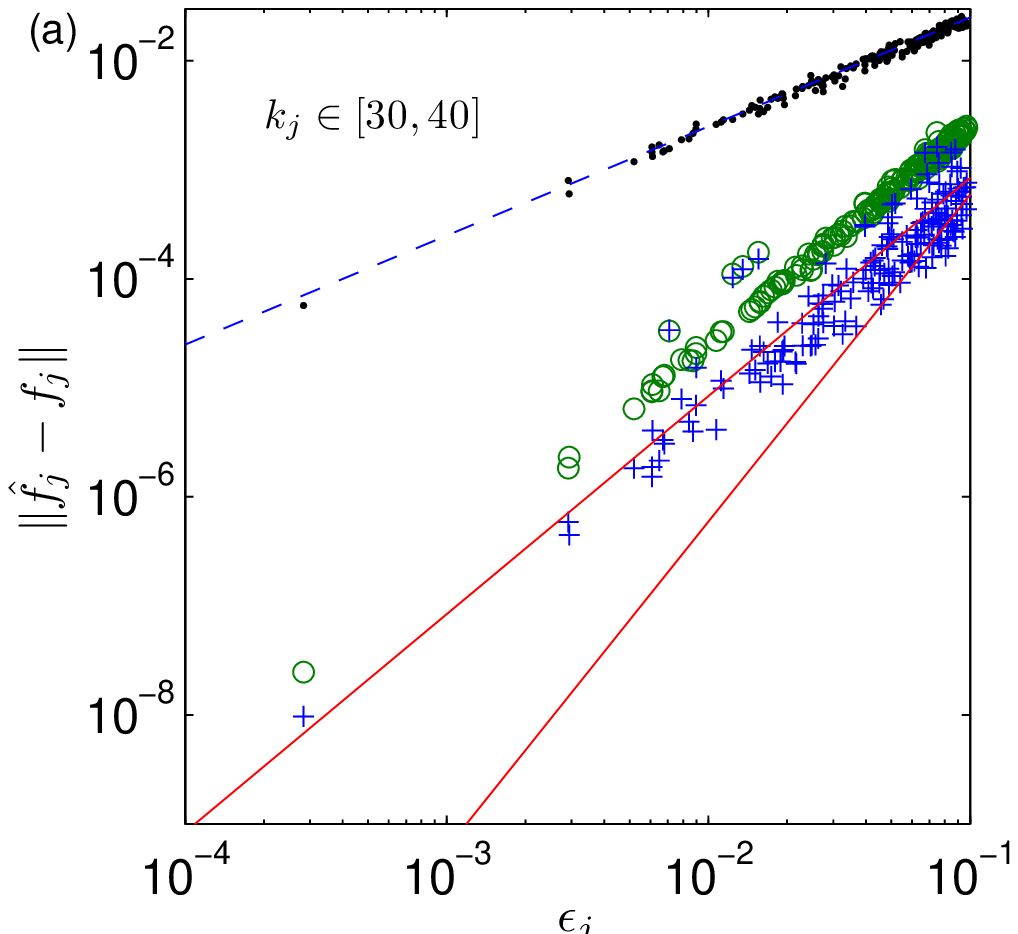}\qquad
\ig{width=0.45\textwidth}{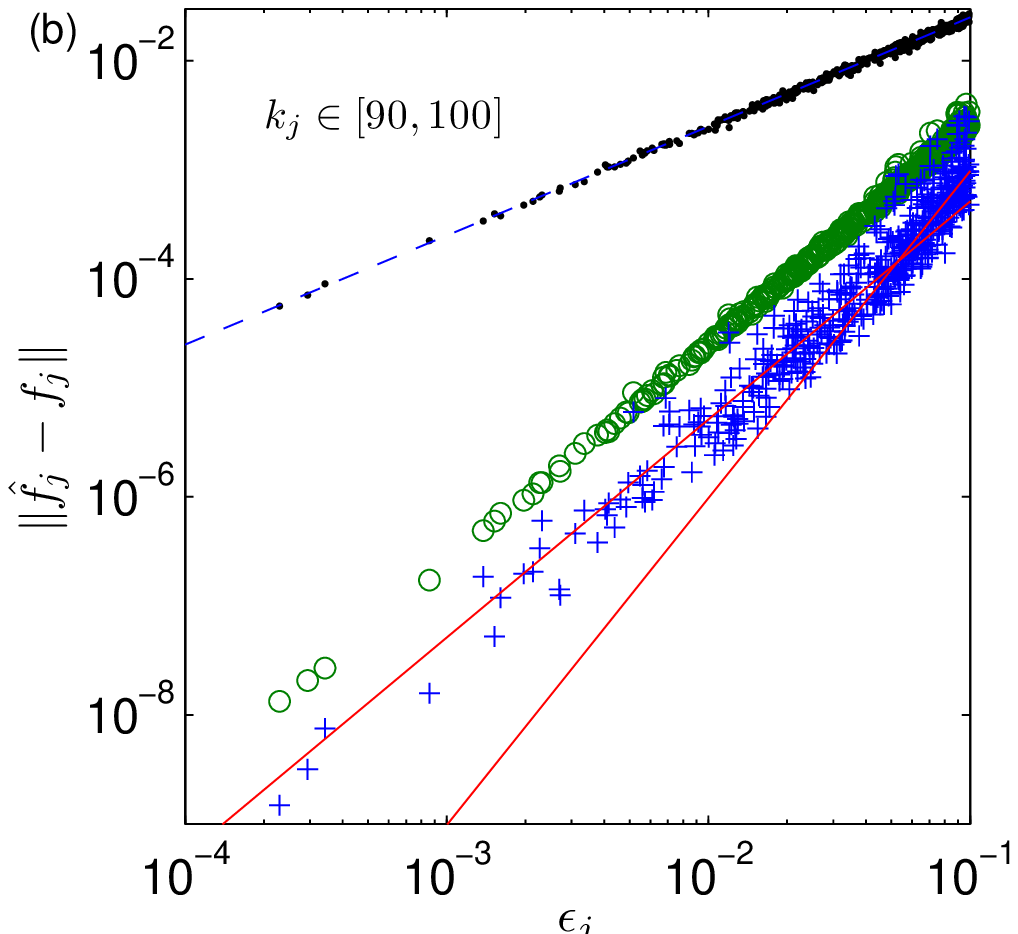}

\ig{width=0.45\textwidth}{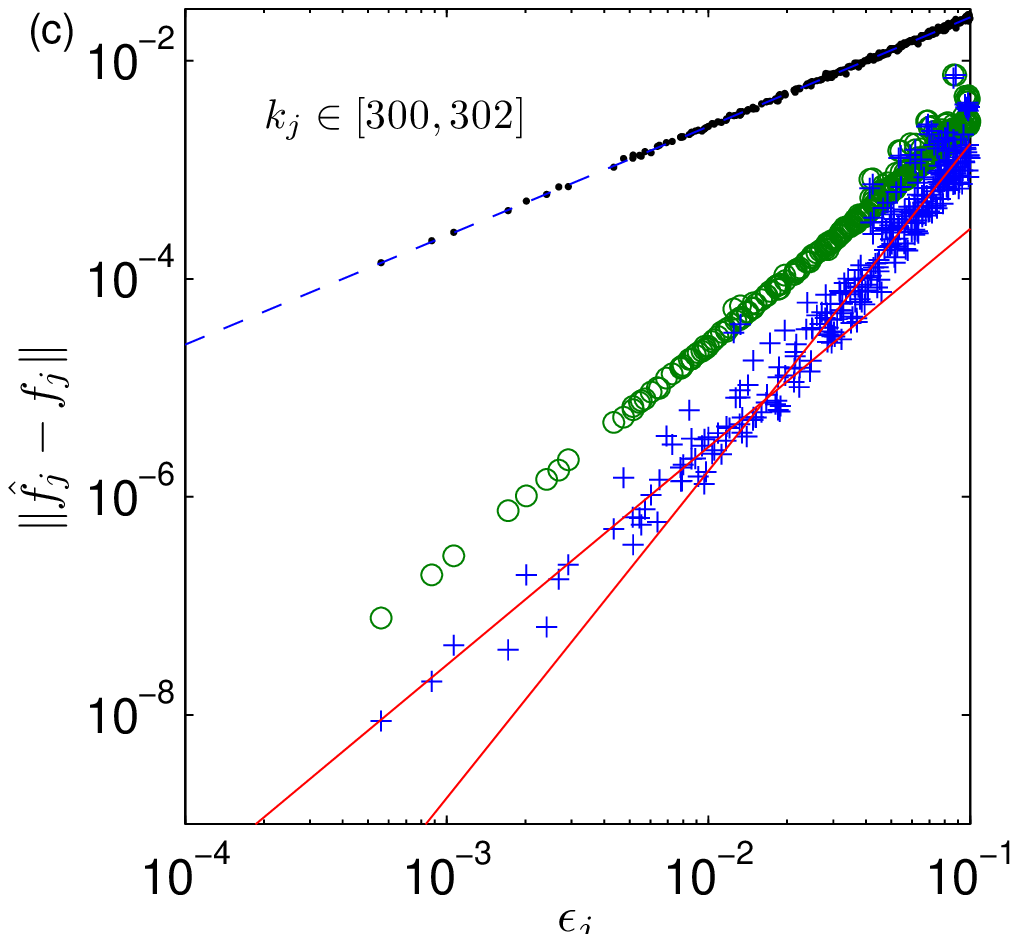}\qquad
\ig{width=0.45\textwidth}{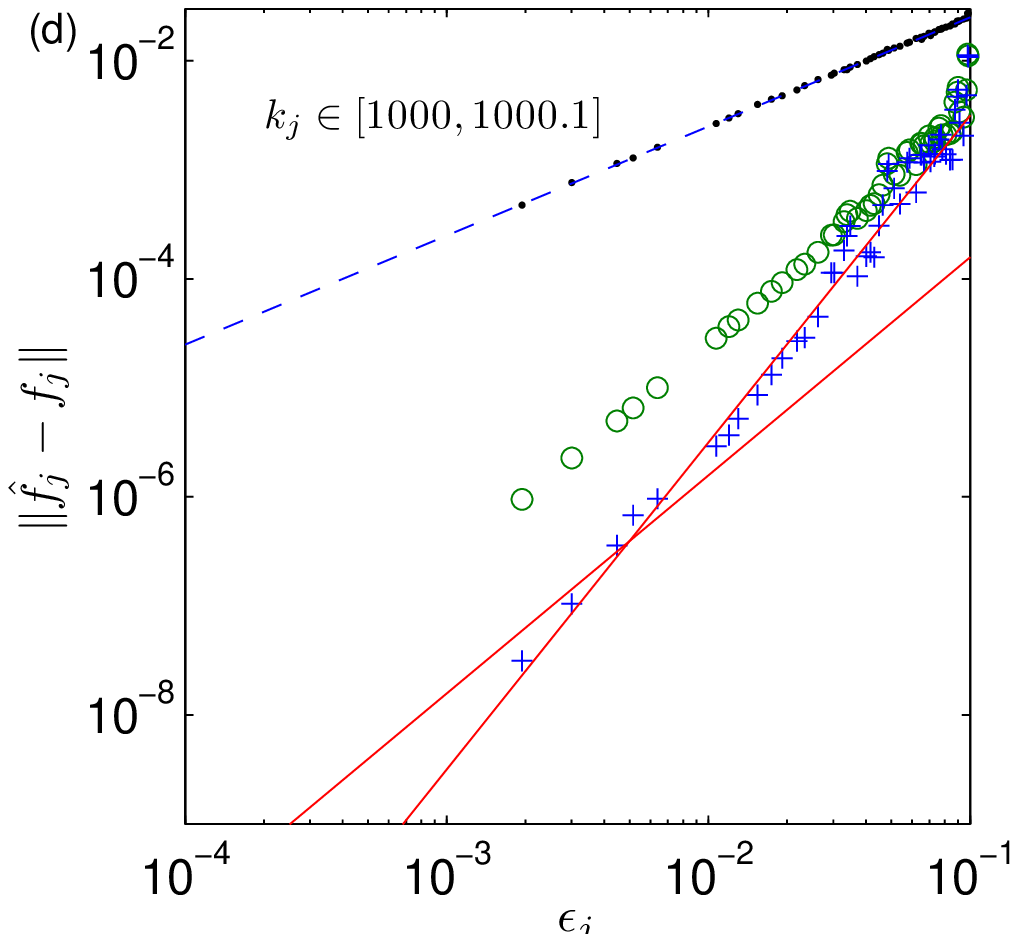}
\ca{Boundary error norms
of three eigenfunction prediction schemes for four different
$k$ ranges, vs $\ep_j:=k_j-\kstar$. Trivial scheme (dots), using \eqref{fhatf},
is compared against $0.25\ep$ (dotted line).
Linearized scheme (circles), using \eqref{fhatsecondorder} is shown.
Quadratic scheme (crosses),
using \eqref{fhatimpsecondorder}, is compared against power laws
$0.5\ep^2k^{-1/2}$ and $0.1\ep^3 k^{1/2}$.
The domain is as in Fig.~\ref{f:flow}.
}{f:ferr}
\efi

\subsection{Higher-order reconstruction of eigenfunctions}
\label{s:fhigher}


In order to find higher order estimators for the Dirichlet eigenfunction (or more precisely, its normal derivative at the boundary), we first compute the $k$-derivative of an eigenfunction branch $f(k)$ of $\Theta(k)$. For simplicity we assume the eigenspace is simple\footnote{Note that our rigorous results also make this assumption as it is a consequence of Assumption~\ref{ASC}.}. 


Taking the derivative of \eqref{e:ef} gives the formula
\begin{equation}
(\ntd - \beta) \dk f = (\dk \beta - \dk \ntd) f .
\label{Thetaf}\end{equation}
Now fix $\kustar \in [\kstar, \kstar + \epsilon]$ and let $\ut(k)$ be as in \eqref{e:vdefn}, hence satisfying \eqref{e:ntdpf} and \eqref{e:upbvp}. 
Also, at $k = k_0$ we have 
\begin{equation}
\ut(k_0) |_{\dOmega} = \beta(k_0) f(k_0).
\label{vbf}\end{equation}

We make the observation that, due to the commutation formula
$$
\big[ \Delta, x \cdot \nabla \big] = 2 \Delta,
$$
the function $x \cdot \nabla \ut/k$ satisfies
$$
(\Delta + k^2) \Big( \frac1{k} x \cdot \nabla \ut \Big) = -2k \ut.
$$
Therefore, combining this with \eqref{e:upbvp}, the function
$q:=\dk\ut - \frac1{k}x \cdot \nabla \ut$ is Helmholtz for every $k$,
and so we have, at $k = \kustar$, a relation between the value
and normal derivative of $q$ on the boundary,
$$
(\ntd - \beta)\xn q_n \;=\; q|_\pO - \beta \xn q_n~.
$$
Using the second part of \eqref{e:upbvp} this simplifies to the following at $k = \kustar$,
\be
\frac1{k}\big( \Theta - \beta\big) \xn \partial_n (x \cdot \nabla \ut)
\;=\;
- \dk{\ut}|_\pO +  \frac1{k}(x \cdot \nabla \ut)|_\pO  - \frac{\beta}{k} \xn \partial_n (x \cdot \nabla \ut)
~.
\label{e:qsimp}
\ee
We need to re-express the spatial 2nd-derivatives in terms of the boundary
$\pO$.
The Laplace-Beltrami operator $\Delta_\pO$
is related to the Laplacian in $\RR^d$ by
\be
\Delta = \partial_{nn} + (d-1)H \partial_n + \Delta_\pO
~,
\label{e:lb}\ee
where the scalar function $H$ is the mean curvature of $\pO$.
Thus, for any Helmholtz function $w$, writing the scalar function
$h := - (d-1) \xn H$, we have,
\be
(\xn \partial_n )^2 w \big|_{\dOmega}
=
- \xn^2 (\Delta_{\dOmega} + k^2) w|_\pO + (1  + h) \xn w_n
~.
\label{e:ident1}
\ee
Writing $x \cdot \nabla$ at the boundary as $\xn \partial_n + W$,
we also compute, for any smooth function $w$,
that,
\be
\big[ \xn \partial_n, x \cdot \nabla \big] w \big|_{\dOmega}
=
g \xn w_n + W' w
~,
\label{e:ident2}
\ee
where $g:=-\xni W \xn$ is a scalar function, and
$W'$
is a tangential derivative operator on $\pO$
whose vector field is given by the covariant derivative of $n$ with respect to the dilation vector field $x \cdot \nabla$.
Here we extend the normal vector $n = (n_1, \dots, n_d) = \sum n_i e_i $
into a neighbourhood of $\pO$ so that it is of unit length and constant along lines perpendicular to the boundary. Explicitly,
\begin{equation}
W' := \xn \sum_{i,j = 1}^d x_i \frac{\partial n_j}{\partial x_i} \partial_{x_j}
~.
\label{e:Wp}
\end{equation}
One may check that $g+h+1 = m$ from \eqref{e:m},
via the identity $h = \div W - 1$.
Thus combining \eqref{e:ident1} and \eqref{e:ident2},
and inserting \eqref{vbf} and $\xn \dn{\ut} = f$,
we get
$$
\xn \partial_n (x \cdot \nabla \ut) \; = \;
(W + m) f  -  \beta \xn^2 (\Delta_{\dOmega} + k^2) f + \beta W' f
~.
$$
Acting on this with $\frac1{k}(\ntd - \beta)$ then
equating with \eqref{e:qsimp}, replacing $\dk{v}$ via \eqref{e:ntdpf},
and again expanding $x \cdot \nabla \ut$
gives, at $k = \kustar$,
$$
\begin{aligned}
\frac1{k}\big( \Theta &- \beta\big) \Big( (W + m) f  + \beta W' f - \beta \xn^2 (\Delta_{\dOmega} + k^2)  f  \Big)  \\
&= - \dk \ntd f + \frac1{k} \Big( f + \beta W f \Big) - \frac{\beta}{k} \Big( (W + m)f  -  \beta \xn^2 (\Delta_{\dOmega} + k^2) f \Big) \\
&= (\dk \beta - \dk \ntd) f  + \big(\frac1{k} - \dk \beta) f  - \frac{\beta}{k} \Big( mf  + \beta W' f -  \beta \xn^2 (\Delta_{\dOmega} + k^2) f \Big)  .
\end{aligned}
$$
Notice that the $\beta W f/k$ terms canceled in the last step.
Combined with \eqref{Thetaf}, and observing that the range of $\ntd - \beta$ is orthogonal to $f$, we get 
$$
\begin{aligned}
\frac1{k}\big( \Theta &- \beta\big) \Big( (W + m) f  + \beta W' f - \beta\xn^2 (\Delta_{\dOmega} + k^2)  f \Big)  \\
&= (\ntd - \beta) \dk f  - \frac{\beta}{k} \Big( mf + \beta W' f - \beta \xn^2 (\Delta_{\dOmega} + k^2)  f \Big)^\perp  
\end{aligned}
$$
where  ${}^\perp$ indicates projection onto the space orthogonal to $f$. Now applying $(\ntd - \beta)^{-1}$ (again we consider the generalized inverse, equal to zero on the span of $f$ and inverting $\ntd - \beta$ on the orthogonal complement), we find
\begin{multline}
\dot f = \frac1{k} \Big( (W + m) f   + \beta W' f - \beta \xn^2 (\Delta_{\dOmega} + k^2) f \Big) \\ + \frac{\beta}{k} (\ntd - \beta)^{-1} \Big(
mf + \beta W' f - \beta \xn^2 (\Delta_{\dOmega} + k^2  )  f
\Big) + cf,
\label{fdot}\end{multline}
where the constant $c$ is determined by the normalization, i.e.\ $\wang{\dk f}{f}=0$.

From this we can determine the first and second derivatives of $f_p(k)$,
the eigenfunction on the branch corresponding to Dirichlet eigenfrequency $k_p$,
when $k = k_p$, that is, when $\beta = 0$:

\begin{proposition}\label{prop:fderivs} Let $(f(k), \beta(k))$ be an eigenpair for $\Theta(k)$, and let $D$ be the differential operator
$$
D := W + m - \frac1{2} \ang{mf, f}~.
$$
Then if $\beta(k) = 0$, the first and second derivatives for the eigenfunction $f(k)$ are 
\begin{equation}\begin{aligned}
\dk{f} &=  \frac1{k} Df ~; \\
\ddk{f} &=  \frac1{k^2} \Big( (D^2 - D) f  - \xn^2 (\Delta_{\dOmega} + k^2) f  + W' f  + \Theta(k)^{-1}(mf) \Big) \\
{} & \phantom{aaaaa} - \frac1{k^2}\ang{mf, Df} f  + c_1 f
~, 
\end{aligned}\label{12derivs}\end{equation}
where $c_1$ is some normalization constant.
\end{proposition}

\begin{proof} The first identity follows from \eqref{fdot} by setting $\beta = 0$, and computing that at $\beta = 0$, noting
that $W+m$ is adjoint to $-W$ with respect to \eqref{e:ip},
$$
c = - \frac1{k} \ang{(W+m)f, f} = - \frac1{2k} \ang{mf, f}. 
$$
The second formula follows by taking the $k$-derivative of the right hand side of \eqref{fdot} and then setting $\beta$ to zero, using \eqref{e:invk}.
\end{proof} 

This proposition suggests that the following two estimators for $f(k_p)$ should be more accurate than the trivial estimator $\hat f = f_p(\kstar)$ considered in Section~\ref{s:recon}. First, using just the first derivative formula in \eqref{12derivs}, we consider, with $f=f_p(\kstar)$, the
\begin{equation}
\mbox{``linear $f$ estimator'':}
\qquad
\hat f_p := f + \frac{\hep}{\kstar}
Df
~, \quad \mbox{ where } \hep := \hat k_p - \kstar~,
\label{fhatsecondorder}\end{equation}
$\hep$ being the best available estimate for $k_p-\kstar$, e.g.\ via
\eqref{e:khatr}.
Numerically in $d=2$ we handle the term $Wf=(x\cdot t)\partial_t f$ using a
$N\times N$ spectral
differentiation matrix \cite[Ch.~3]{tref} applied to the discretized $f$;
the FFT could also be used.
Referring to the data shown by circles in Fig.~\ref{f:ferr},
for the domain of Fig.~\ref{f:flow}(b), we see
that empirically, this estimator is second-order accurate in $\epsilon$,
with a constant that is independent of $k$.
This improves upon the trivial estimator
by one to three extra digits of accuracy.

In principle, one should be able to use the second derivative of $f$ given by \eqref{12derivs} to obtain a third order accurate estimator. Unfortunately, this formula involves the operator $\Theta(k_p)^{-1}$ which is not known explicitly;
it could be approximated numerically at a cost of $O(N^3)$, but this would 
need to be done afresh at each eigenfrequency and thus destroy
the $O(N^2)$ complexity per mode.
However, if we study the size of the terms in the second derivative formula \eqref{12derivs}, we see that some of them can be expected to be lower order (in $k$) than others. For example, the terms $k^{-2} Df$ and $W' f$ are  lower order than $k^{-2} D^2f$. Also, as discussed in Remark~\ref{sqrtk}, subject to a spectral nonconcentration assumption, $k^{-2}\Theta(k_p)^{-1}(mf)$ is typically
a factor $k^{1/2}$ smaller than the leading terms.
The normalization constant $c_1$ is also irrelevant to the order of
accuracy we seek (we will instead normalize numerically).
Thus, keeping the leading terms in \eqref{12derivs},
\be
\ddk{f}
\;\approx\; \frac{1}{k^2}\big(
D^2 - D - \xn^2 (\Delta_{\dOmega} + \kstar^2)
\big) f
~.
\ee
At this order we also need to consider linear variation in $\dk{f}$, so
we approximate $\dk{f}(\hat k_p)$ by substituting \eqref{fhatsecondorder}
into the $\dk{f}$ formula in \eqref{12derivs}, that is,
$$
\hat{\dk{f}}_p \;= \;\frac1{\kstar}D \Big( f + \frac{\hep}{\kstar} D f \Big)
~.
$$
A second-order $f$ expansion
about $\hat k_p$ then gives
$\hat f_p = f + \hep \hat{\dk{f}}_p - (\hep^2/2) \ddk{f}$,
which we can simplify,%
\footnote{Note that one may view our procedure as inverting
a Taylor series to second order.}
noting the sign change in the $D^2$ term, to the improved
\begin{multline}
\mbox{``quadratic $f$ estimator'':}
\quad
\hat f_p
:=
f + \frac{\hep}{\kstar} D f + \frac{\hep^2}{2\kstar^2}
\big(D^2 + D + \xn^2 (\Delta_{\dOmega} + \kstar^2)
\big) f
~. 
\label{fhatimpsecondorder}\end{multline}
As before, we use the best available
$\hep$ estimate.
In $d=2$ we approximate $\Delta_\pO f = \partial_{tt} f$ via
a spectral differentiation matrix.
Finally we normalize $\hat f_p$ numerically.

Figure~\ref{f:ferr} (data shown by crosses) shows the improved accuracy of this estimator: it gives typically one extra digit over \eqref{fhatsecondorder},
and up to four extra digits over the trivial estimator.
This error data is also summarized in the last two columns of Table~\ref{t:rfn}.
The figures strongly suggest an empirical error of
$O(\ep^2 k^{-1/2} + \ep^3k^{1/2})$ for this estimator.
As expected from the above discussion,
the first term is a factor $k^{1/2}$ smaller
than the error of \eqref{fhatsecondorder}.
As with the linear eigenfrequency estimator,
the cubic term dominates for larger frequency distances $\ep\gg k^{-1}$,
which are needed anyway in $d=2$
to capture more than $O(1)$ mode per frequency window.
Thus, in the fast regime,
this method has asymptotic eigenfunction error $O(\ep^3k^{1/2})$.
If it is desired to keep this error bounded as $k\to\infty$,
one must choose $\ep<k^{-1/6}$ rather than the $\ep=O(1)$ allowed for
bounded eigenfrequency error. This reduces the speed-up factor of the
NtD method slightly from $O(N)$ to $O(N^{5/6})$ in $d=2$.

In section~\ref{s:efn-error-analysis} we will give rigorous estimates on the
Riccati estimator \eqref{e:khatr} and linear $f$ estimator \eqref{fhatsecondorder}, assuming that the spectrum of $\Theta(\kstar)$ does not concentrate near $\beta_p$.

\bfi   
\mbox{a)\raisebox{-1.8in}{\ig{width=0.48\textwidth}{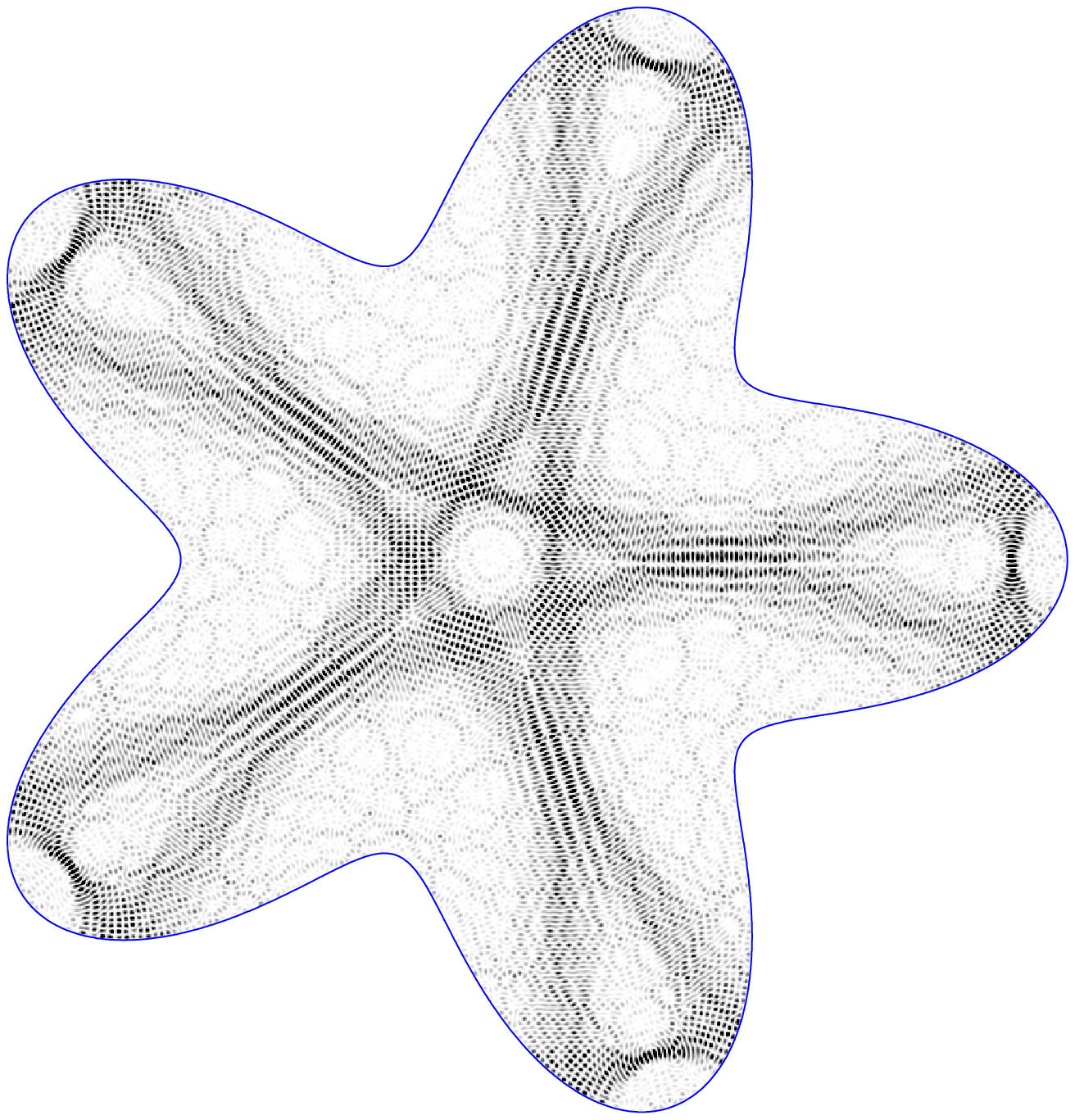}}
b)\raisebox{-1.8in}{\ig{width=0.48\textwidth}{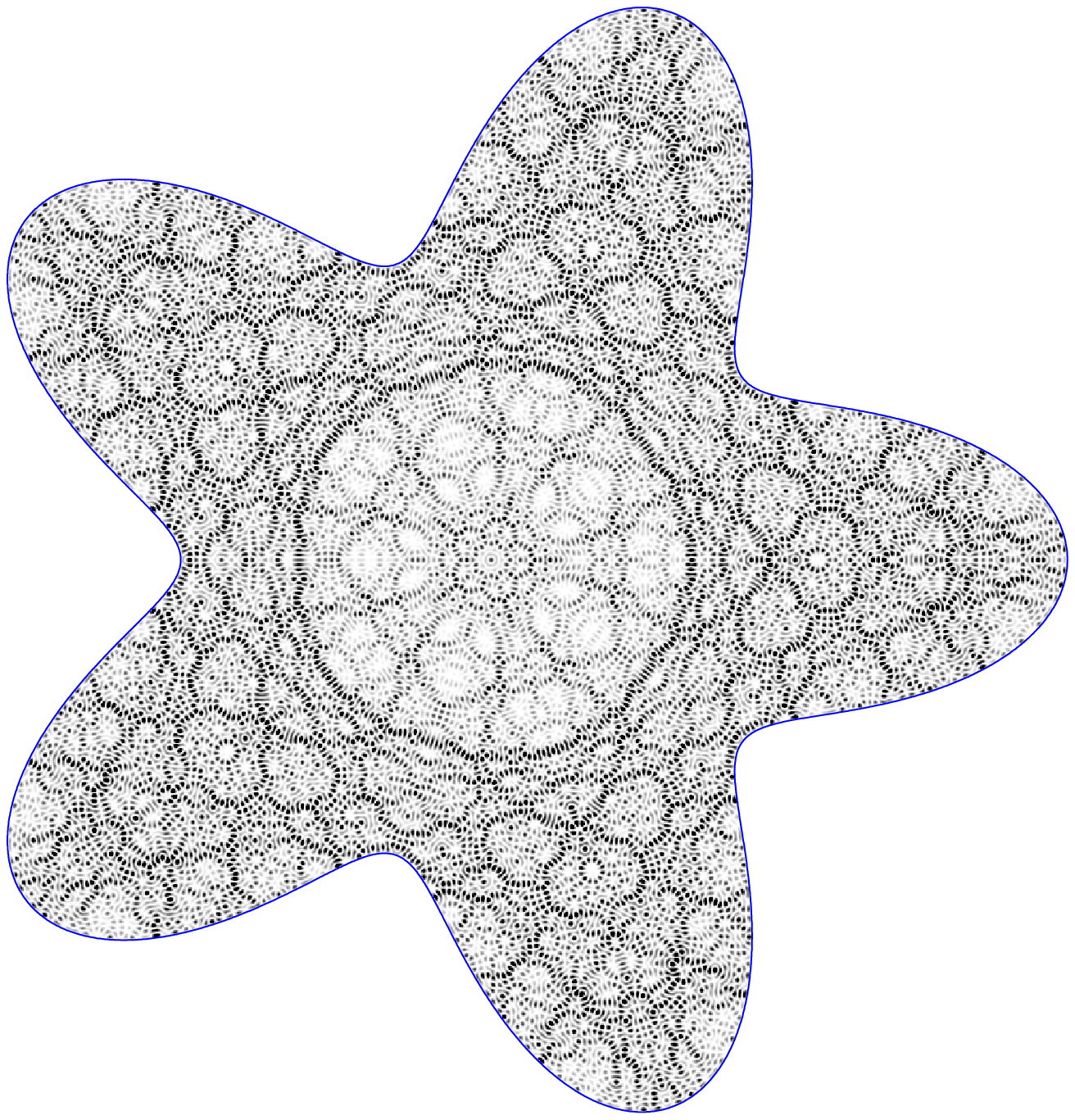}}
}
\ca{Eigenfunctions of the pentafoil domain.
(a) mode from the two-dimensional eigenspace at 
$k_j = 300.005956478458 \cdots$ (function is not $D_5$-symmetric);
note scarring on a periodic orbit.
(b) simple (hence $D_5$-symmetric) mode at $k_j = 300.03832269 \cdots$.
(All digits believed correct.)
Density shows $|\hat\phi_j|^2$, with white being zero.
Parameters are as in first row of Table~\ref{t:penta}.
}{f:pentamed}
\efi

\bfi   
\ig{width=1.0\textwidth}{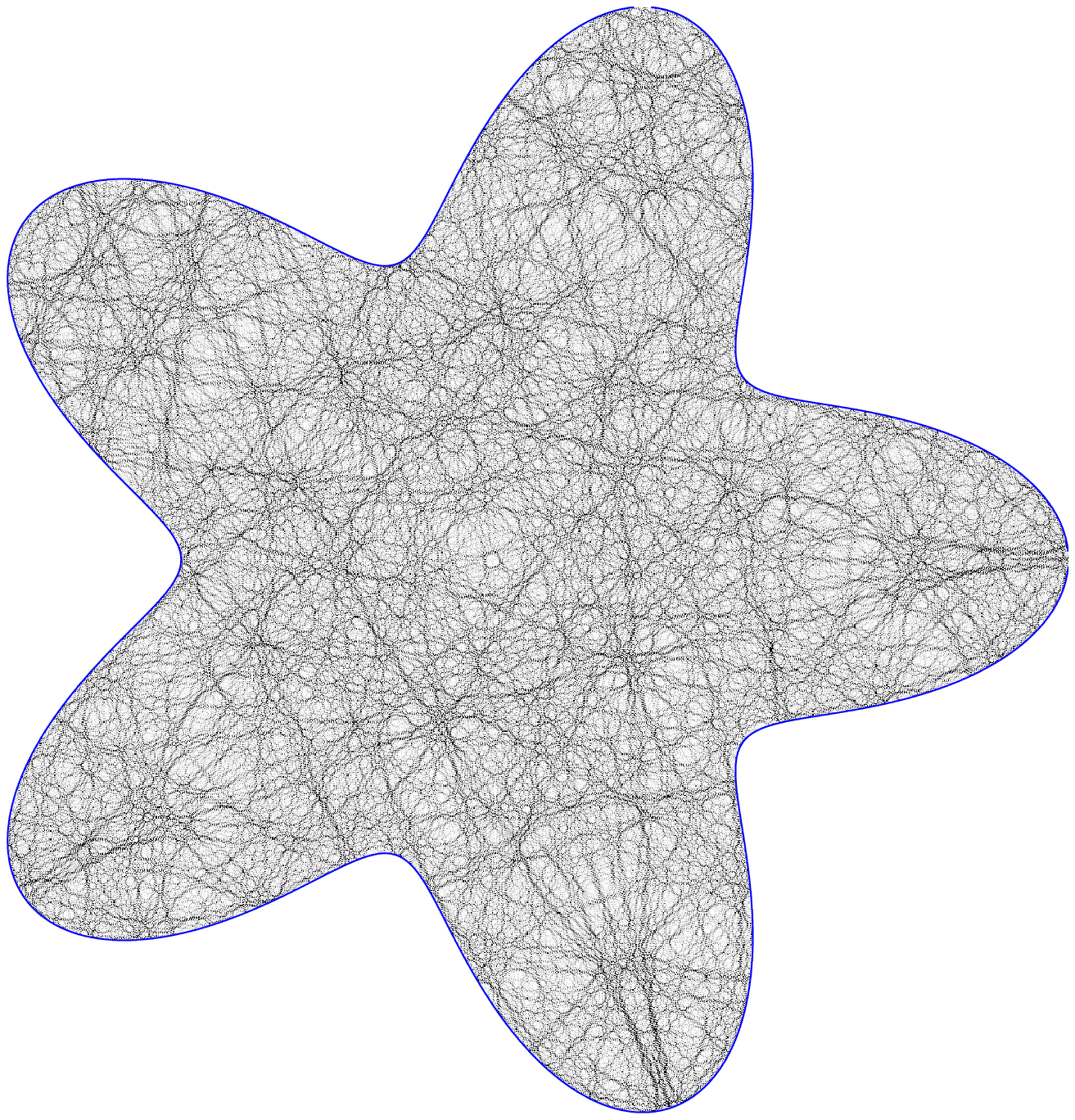}
\ca{An eigenfunction of the pentafoil domain
in the two-dimensional eigenspace with
$k_j = 1000.00302930323 \cdots$ (all digits believed correct).
There are around 400 wavelengths across the domain.
Parameters are as in second row of Table~\ref{t:penta}.
}{f:penta}
\efi

\begin{table}[t]  
\small
\begin{tabular}{l|r|r|r|rrr|rr|rr|}
\multicolumn{4}{l|}{} &
\multicolumn{3}{|l|}{time / mode (sec)} &
\multicolumn{2}{|c|}{abs error of $\hat k_j$} &
\multicolumn{2}{|c|}{$L^2$-error of $\hat f_j$} 
\\
\hline
$k$ interval & $j$ & $N$ & $n_m$ & ref &
NtD & ratio & max & median & max & median \\
\hline
&&&&&&&&& \\[-2ex]   
$[300,300.1]$ & 2.3e4 & 2700 & 20 & 6500 & 20 & 320 &
  3e$-7$ & 2e$-7$ & 1.0e$-3$ & 1.3e$-4$
\\
$[1000,1000.1]$ & 2.6e5 & 9000 & 51 & -- & 250 & 1000$^\ast$ &
  -- & -- & -- & --
\\[1ex]
\hline
\end{tabular}
\vspace{2ex}
\caption{Runtimes and errors for the pentafoil domain of Figs.~\ref{f:pentamed}
and \ref{f:penta}.
Details are as in Table~\ref{t:rfn}, except
that in the reference method
the absolute $k_j$ errors were only around $10^{-7}$, so errors below
this are not discernable.
For degenerate pairs, subspace angle replaces $L^2$-error (see
Remark~\ref{r:subspace}).
Dashes show experiments not performed, and asterisk ($\ast$) indicates
estimated ratio.
}\label{t:penta}
\end{table} 

\subsection{Performance in a domain with abundant degeneracies}
\label{s:pentafoil}

Having constructed higher-order estimators and tested them in a nonsymmetric
domain,
we now apply them to a ``pentafoil'' domain parametrized by
$r(\theta) = 1 + 0.3 \cos(5 \theta)$.
For group-theoretic reasons (it has the dihedral symmetry group $D_5$),
its Dirichlet eigenfrequencies are generically either simple or
a doubly-degenerate pair.
As before, we found that an $N$ of around $6.3$ points per wavelength on $\pO$
gave full double-precision accuracy in computing the spectrum of $\ntd(\kstar)$.
Table~\ref{t:penta} summarizes our experiments
comparing the proposed NtD method (using
the Riccati \eqref{e:khatr}
and quadratic \eqref{fhatimpsecondorder} estimators)
against the reference solver of App.~\ref{a:ref}.
Observe that error levels of the NtD method
are similar to those for the previous
domain, thus degeneracies seem to have no deleterious effect on error.

\begin{remark} \label{r:subspace} 
For simple eigenvalues, as before,
the $L^2$ error $\|\hat f_j - f_j\|$ was measured,
with $\|\hat f_j\| = \|f_j\| = 1$.
For $p$-fold degenerate eigenvalues we used its generalization, the {\em principal angle between subspaces}.
Here, one subspace is the eigenspace computed by the
NtD method, while the other is that computed by the reference method.
In the small angle limit and $p=1$ this is equivalent to the $L^2$ error.
\end{remark}                     

Note that the reference method was slower,
by roughly a factor of three compared to the nonsymmetric domain of
Fig.~\ref{f:flow}(b), due to the difficulty of resolving eigenfrequency pairs.
(Here a large tolerance {\tt tol = 1e-6} was chosen to limit this
slow-down.)
In contrast, the NtD method pays no penalty for close or exact
degeneracies---this is one of its main advantages---thus its speed-up
factors are around three times better than for the former domain at
similar frequencies.

In Figs.~\ref{f:pentamed} and \ref{f:penta} we show some eigenfunctions
coming from the calculations of the
first and second rows of Table~\ref{t:penta} respectively.
In the latter case, forming and diagonalizing the matrix of size $N=9000$
took 3.6 hrs\footnote{This resulted in some swapping of RAM to hard drive, indicating that this about the largest $N$ that can be handled on this 8 GB machine.}
and returned 51 modes.
Based on the previous domain, we expect mediam error
similar to those in the first row of the table.
However, since for the mode shown, $k_j$ is so close to $\kstar$
that we expect $\hat k$ error to be limited by machine precision ($10^{-13}$
absolute error), and eigenfunction $L^2$ error to be $10^{-7}$.
We did not attempt to run a reference calculation here (it would have
taken 3 weeks),
but using the $O(N^3)$ scaling from our previous tests, we
estimate that our method is faster than the reference method by a
factor of $10^3$.

To create Fig.~\ref{f:penta},
evaluation of the representation \eqref{e:phirec} on a grid of
$8.2 \times 10^5$ points took only 27 sec per eigenfunction
using the Helmholtz fast multipole method (FMM) implementation
of Gimbutas--Greengard \cite{HFMM2D}.
The entire eigenmode calculation and plot
is done by the following \mpspack\ code:
\begin{verbatim}
s = segment.smoothstar(9000, 0.3, 5);
d = domain(s, 1); s.setbc(-1, 'D'); p = evp(d);
o.eps = 0.1; o.modes = 1; o.khat = 'r'; o.fhat = 's';
p.solvespectrum([1000 1000.1], 'ntd', o);
o = []; o.inds = 1; o.dx = 0.002; o.fmm = 1; o.col = 'bw'; showmodes(p, o);
\end{verbatim}
The third line selects the Riccati esimator for $\hat k$ and
the quadratic estimator for $\hat f$.

\brmk
If a domain has a known symmetry (such as the $D_5$ symmetry of this
pentafoil example),
it is possible to reduce $N$ by
{\em desymmetrizing} and finding eigenfunctions in each symmetry class
separately \cite{backerbim,mythesis}.
This is often done in high-frequency studies \cite{que}
because it increases efficiency by a significant factor.
For simplicity, we did not implement that here.
\ermk


\section{Error analysis of higher-order methods}
\label{s:efn-error-analysis}

In this section we specialize to the case of two dimensions, $d=2$. This allows us to use the exploit the relatively large mean spacing of Dirichlet eigenfrequencies in two dimensions (relative to higher dimensions). 

\subsection{A spectral nonconcentration assumption}
All error estimates in this section will be conditional on the following assumption:

\begin{assumption}[Absence of Spectral Concentration at scale $\eta$]\label{ASC} Let $\eta$ be a positive real number, and let $\beta$ be a negative eigenvalue of $\Theta(k)$ satisfying
$$
-\frac{\epsilon}{k} \leq \beta \leq 0.
$$
We say that there is absence of spectral concentration at $\beta$ at the scale $\eta$ if $\beta$ is the only eigenvalue of $\Theta(\kstar)$ (counted with multiplicity) in the interval 
$$
\big[\beta - \frac{\eta}{k^2}, \  \min(\beta + \frac{\eta}{k^2}, 0) \big].
$$
This implies, in particular, that $\beta$ is a simple eigenvalue. 
\end{assumption}

Notice that the eigenfrequencies of $\Delta$ are spaced $\sim 1/k$ apart on average when $d=2$, so in view of \eqref{betaintermsofk}, the eigenvalues of $\Theta(\kstar)$ are spaced $\sim 1/k^2$ apart on average. Therefore, for sufficiently small $\eta$, we can expect that typically Assumption~\ref{ASC} is satisfied for most eigenvalues $\beta$ of $\Theta(\kstar)$ in the range $[-\epsilon/k, 0]$, uniformly in $k$. \emph{We will always assume that $\eta \leq 1$ in our estimates below.}

One simple consequence of \eqref{betadot} is that, if $\epsilon$ is not too large relative to $\eta$, Assumption~\ref{ASC} implies that the eigenvalue branch $\beta_p(k)$ is well-separated from neighbouring branches on the whole interval $[\kstar, k_p]$, where $\beta(k_p) = 0$:

\begin{lemma}\label{lem:evalgaps} Assume that $d=2$ and that the eigenvalue $\beta_p$ of $\Theta(\kstar)$ lies in the interval $[-\epsilon/k, 0]$ and satisfies Assumption~\ref{ASC} at scale $\eta$. If $\eta$ satisfies 
\begin{equation}
\eta \geq 8C \epsilon^2,  \quad \eta \geq 8C \epsilon^3 k,
\label{etaconditions}\end{equation}
then $\beta_p(k)$ satisfies Assumption~\ref{ASC} at the scale $\eta/2$ for all $k \in [\kstar, k_p]$. Here the constant $C$ is the implied constant in \eqref{betadot}. 
\end{lemma}

\begin{proof} Let $\beta_p(k)$ and $\beta_q(k)$ be two negative eigenvalue branches of $\Theta(k)$, with $\beta_q(k) > \beta_p(k)$. Then (disregarding the trivial case in which $\beta_q(\kstar) = 0$) we have 
$$
| \beta_p(\kstar) - \beta_q(\kstar)| \geq  \frac{\eta}{k^2}.
$$
We will show that 
$$
| \beta_p(k) - \beta_q(k)| \geq \frac{\eta}{2 k^2}
$$
for all $k > \kstar$ for which both eigenbranches are defined (i.e. such that both $\beta_p(k) \leq 0$ and $\beta_q(k) \leq 0$). Using \eqref{betadot} we have 
$$
\Big| \frac{d}{dk} \big( \beta_p(k) - \beta_q(k) \big) \Big| \leq 2C \big( \frac{\epsilon}{k^2} + \frac{\epsilon^2}{k} \big),
$$
since $|\beta_p(k)|, |\beta_q(k)| \leq \epsilon/k$ for all $k \geq \kstar$. Integrating over the interval $[\kstar, k]$ which is no bigger than $\epsilon$ for $k \leq k_p$, we find that 
$$\begin{gathered}
|\beta_p(k) - \beta_q(k)| \geq |\beta_p(\kstar) - \beta_q(\kstar)| - 2C \big( \frac{\epsilon^2}{k^2} + \frac{\epsilon^3}{k} \big) \\
\geq \frac{\eta}{k^2} - 2C \big( \frac{\epsilon^2}{k^2} + \frac{\epsilon^3}{k} \big) \\
\geq \frac{\eta}{2k^2},
\end{gathered}$$
using the conditions \eqref{etaconditions} in the last step. 
\end{proof}

\subsection{Error estimate for second-order eigenfunction reconstruction}
We now give an error estimate for the estimator \eqref{fhatsecondorder}, 
 given Assumption~\ref{ASC} at scale $\eta$. We begin with 
 
\begin{lemma}\label{lem:fdotsize}
Assume $d=2$. For $\kustar \in [\kstar, \kstar + \epsilon]$ and $-\epsilon/k \leq \beta(\kstar) \leq 0$, the $k$-derivative of $f$ satisfies 
\begin{equation}
\dot f = \frac1{k}(W + m)f + O(\frac{\epsilon + \epsilon^2 k}{\eta}) 
\label{dotfest}\end{equation}
where $W$ is as in \eqref{Alex's-identity} and $m$ as in \eqref{e:m}.
\end{lemma}

\begin{proof}
Consider the terms on the right hand side of \eqref{fdot}. 
Indeed, using Lemma~\ref{lem:utu} (and Remark~\ref{constant-remark} for the higher order derivatives), and since $|\beta| \leq \epsilon/k$, we see that 
$$
\frac{\beta}{k} \| W' f \| \leq \frac{\epsilon}{k}, \quad 
\frac{\beta}{k} \| \xn^2 (\Delta_{\dOmega} + k^2) f \| \leq \epsilon.
$$
Next, using Proposition~\ref{prop:gen-inverse} and Assumption~\ref{ASC} at scale $\eta$, we can estimate the remaining terms on the right hand side of \eqref{fdot} as follows: 
$$\begin{aligned} 
 &\Big\| \frac{\beta}{k}  (\ntd - \beta)^{-1} \Big(   \beta \xn^2 (\Delta_{\dOmega} + k^2  )  f  - \beta W' f -mf \Big) \Big\| \\ 
&\leq C \frac{|\beta|}{k} \frac{k^2}{\eta} \big\|  \beta \xn^2 (\Delta_{\dOmega} + k^2  )  f  - \beta W' f -mf \big\|_{L^2(\dOmega)} \\
& \phantom{33} + C \frac{|\beta|}{k} \big\|  \beta \xn^2 (\Delta_{\dOmega} + k^2  )  f  - \beta W' f -mf \big\|_{H^1(\dOmega)}  \\
&= \frac{C |\beta| k}{\eta} \big( |\beta| k^2 + |\beta| k + 1 \big)
+ \frac{C|\beta|}{k} \big( |\beta| k^3 + |\beta| k^2 + k \big) \\
&\leq \frac{C}{\eta} \big( \epsilon + \epsilon^2 k \big)  .
\end{aligned}$$ 
Here we used Lemma~\ref{lem:utu} and Remark~\ref{constant-remark} to estimate the $L^2$ and  $H^1$ norms in the second and  third lines. 
Finally, the term $c f$ is the result of projecting orthogonally onto the subspace orthogonal to $f$, so this term does not increase the norm. We conclude \eqref{dotfest}. 
\end{proof}

This leads to 
\begin{proposition} Suppose that $d=2$ and that $\beta(\kstar) = \beta_p(\kstar)$ satisfies  Assumption~\ref{ASC} at scale $\eta$. Let $f_p = \xn \partial_n \phi_p / \| \xn \partial_n \phi_p \|$. Then the estimator \eqref{fhatsecondorder} for $f(k_p)$ satisfies 
\begin{equation}
\|\hat f_p - f_p\|_{L^2(\dOmega)} = O(\frac{\epsilon^2 + \epsilon^3 k}{\eta}).
\label{fperror}\end{equation}
\label{p:fperror}
\end{proposition}

\begin{proof}
To do this, we consider two flows. One is the eigenfunction flow
\eqref{fdot} above. The second, for the function $g = g(k)$,  is the linear flow starting at $g(\kstar) = f(\kstar)$, and flowing according to 
$$
\dk g = \frac1{k} (W f(\kstar) + m f(\kstar)).
$$
Note that the RHS here is independent of $k$ (apart from the $1/k$ prefactor). Now we estimate
the difference between $g(\hat k) = g(\kstar/(1 + \beta))$ and the 
weighted normal derivative of the corresponding Dirichlet eigenfunction. This is a sum of two terms: one arising from the difference between $g(\hat k)$ and $f(\hat k)$, and one arising from the difference between $f(\hat k)$ and $f(k_p)$, where $k_p$ is the true eigenvalue. Using \eqref{kerror} and Lemma~\ref{lem:fdotsize}, together with Lemma~\ref{lem:utu} to see that $\| W f \|/k = O(1)$, the second error term is
$$
O\big( \frac{\epsilon^2}{k} + \epsilon^3 \big) \times O \big( 1 + \frac{\epsilon + \epsilon^2 k}{\eta} \big) = O \big( \frac{\epsilon^2}{k} + \epsilon^3 + \frac{\epsilon^3}{k\eta} + \frac{\epsilon^4}{\eta} + \frac{\epsilon^5 k}{\eta} \big),
$$
which is certainly bounded by \eqref{fperror} for large $k$ and $\eta \leq 1$. 

The first difference can be estimated as follows. We have, with $f_* = f(\kstar)$, 
$$\begin{gathered}
\frac{d}{dk} {(f - g)} = \frac1{k} \Big( (W + m)(f - f_*) \Big) + O(\frac{\epsilon + \epsilon^2 k}{\eta}) \\
= \frac1{k} \Big( (W + m)(f - g) \Big) + \frac1{k} \Big( (W + m)(g - f_*) \Big) + O(\frac{\epsilon + \epsilon^2 k}{\eta}).
\end{gathered}$$
Since $g - f_* = O(\epsilon)$ (and using Lemma~\ref{lem:utu} again)
the second term can be absorbed in the error term, and we get
$$
\frac{d}{dk} {(f - g)} =  \frac1{k} \Big( (W + m)(f - g) \Big) + O(\frac{\epsilon + \epsilon^2 k}{\eta}).
$$
Again applying Lemma~\ref{lem:utu}, we have $k^{-1} \| (W + m)(f - g) \| \leq C \| f - g \|$. Therefore, 
$$
\frac{d}{dk} \Big( e^{-C(k - k_p)} \| f - g \|_{L^2} \Big) = O(\frac{\epsilon + \epsilon^2 k}{\eta}) e^{-C(k - k_p)}.
$$
Since $(f-g)(\kstar) = 0$, this inequality integrates to 
\begin{equation}
 \| (f - g)(k) \|_{L^2} = O(\frac{\epsilon^2 + \epsilon^3 k}{\eta}) \text{ for } |k - k_*| \leq \epsilon.
 \end{equation}
 In particular,
 \begin{equation}
 \| (f - g)(\hat k_p) \|_{L^2} = O(\frac{\epsilon^2 + \epsilon^3 k}{\eta}).
 \end{equation}
\end{proof}

\subsection{Error estimate for the Riccati eigenfrequency estimator}

Here we derive an error
estimate for the higher-order eigenfrequency estimator of
section~\ref{s:khigher}, given Assumption~\ref{ASC} at scale $\eta$.


\begin{proposition}
Let the frozen frequency be $k_z=\kstar$.
Then the estimator \eqref{e:khatr} for the eigenfrequency $k_p$ satisfies 
\begin{equation}
|k_p - \hat k_p| \leq C \big( \frac{\epsilon^3}{k^2} + \frac{\epsilon^4}{k \eta} + \frac{\epsilon^5}{\eta} + \frac{ k\epsilon^6}{\eta} \big).
\label{higherorderbetaerror}\end{equation}
\label{p:higherorderbetaerror}
\end{proposition}

Note that if we work in a regime with $\ep=O(1)$, then in the high
frequency limit the term $O(k\ep^6/\eta)$ dominates in this estimate.
However, as we showed in section~\ref{s:khigher}, empirically
the dominant error is only $O(\ep^5)$. When
$k_z=\half(1+(1+\bstar)^{-1})\kstar$ is used instead of $k_z=\kstar$,
empirically the $O(\ep^4/k)$ term is also absent, reducing errors
slightly at intermediate $\ep$ values.

\begin{proof}
Consider the error in estimating the right hand side of \eqref{betadot2} by \eqref{ABODE}. We compute
\begin{equation}
 \frac{d}{dk} \| \xn f \|^2 = 2 \wang{ \dot f}{ \xn^2 f} = \frac{2}{k} \wang{(W + m)f}{ \xn^2 f} + O(\frac{\epsilon + \epsilon^2 k}{\eta}). 
\label{ddkfest}\end{equation}
By integrating by parts, we see that 
$$ \frac1{k} \Big| \ang{(W + m) f, \xn^2 f}  \Big| \leq \frac{C}{k}
$$
(uniformly in $\epsilon$ and $k$). Therefore, for any $\kustar \in [\kstar, \kstar + \epsilon]$, the difference between the value of 
$k \beta^2 \| \xn f \|^2$ (i.e. the second term of \eqref{betadot2}) at $\kustar$ compared to the value at $\kstar$ is bounded by 
$$
Ck \beta^2 \cdot \epsilon \big( \frac{1}{k} +  \frac{\epsilon + \epsilon^2 k}{\eta} \big) = O\big( \frac{\epsilon^3}{k^2} + \frac{\epsilon^4}{k\eta} + \frac{\epsilon^5}{\eta} \big).
$$
 A similar calculation shows that the difference between the third term of \eqref{betadot2} at $\kustar$ compared to the value at $\kstar$ is again $O(\epsilon^3/k  + \epsilon^4/(k\eta) + \epsilon^5/\eta)$. Treating the fourth and  last term of \eqref{betadot2} similarly, we obtain an error estimate of $O(\epsilon^2/k^2 + \epsilon^3/(k^2 \eta) + \epsilon^4/(k\eta))$ 
between the value of this term at $\kustar$ compared to the value at $\kstar$. 

Therefore, the error term in $\beta$ for any $\kustar \in [\kstar, \kstar + \epsilon]$ is bounded by integrating this error on the interval $[\kstar, \kstar + \epsilon]$ and is therefore bounded by 
$O(\epsilon^3/k^3 + \epsilon^4/(k^2 \eta) + \epsilon^5/(k\eta) + \epsilon^6/\eta)$. Finally, since the derivative  $d\beta/dk$ is comparable to $1/k$ (say, between $1/2k$ and $2/k$) by \eqref{betadot}, we see that the error in the estimate for $k_p$ is bounded by \eqref{higherorderbetaerror}. 
\end{proof}

\section{Connection to the scaling method of Vergini--Saraceno}
\label{s:v+s}

Our above proposed NtD method is closely related to, indeed inspired by, the
scaling method of Vergini--Saraceno \cite{v+s}.
Here we explain briefly the latter, using the language of
numerical mathematics (the original paper is very short and written in
a physics style), thus improving upon previous understandings
\cite{mythesis,que}.
We at least heuristically explain its observed accuracy, and
highlight the many differences with the present NtD method.

\subsection{Sketch of the scaling method}
The method exploits the fact that a scaled, or dilated,
Helmholtz solution is still Helmholtz.
Let $\Phi(k)$, $k\neq k_j$, be the operator mapping Dirichlet data to
the {\em dilational} derivative of its interior Helmholtz extension, that is,
given $g\in H^1(\pO)$, and $u$ satisfying its Dirichlet problem \eqref{e:dbvp},
its action is
$$
\Phi(k) g = x\cdot\nabla u|_\pO~.
$$
Now consider a Dirichlet eigenfunction $\phi_j$,
and let $f = \xn \partial_n \phi_j |_{\dOmega}$.
Take a frequency $\kstar = k_j - \epsilon$ where $\epsilon > 0$ is small,
and define $\phi_j(\kstar)$ to be the {\em dilation}
of the function $\phi_j$ to this new frequency $\kstar$, that is
\be
\phi_j(\kstar)(\xx) := \phi_j\bigl(\frac{\kstar}{k_j} \xx \bigr)
~,\qquad \xx \in \RR^d
\label{e:dil}
\ee
Then to first order in $\epsilon$, we have
\be
\phi_j(\kstar)|_{\dOmega} = -\frac{\epsilon}{k} (1 + O(\epsilon)) f
\label{e:phidil}
\ee
and 
\be
\xx \cdot \nabla \phi_j(\kstar)|_{\dOmega} = (1 + O(\epsilon)) f ~.
\label{e:xdphidil}
\ee
The last two equations tell us that the dilated eigenmode $\phi_j(\kstar)|_\pO$,
is an approximate eigenfunction of $\Phi(\kstar)$ with
approximate eigenvalue $-k/\epsilon$.

The scaling method uses a linearized self-adjoint version of the above.
Let $\Phi^\ast$ be the adjoint of $\Phi$ with respect to \eqref{e:ip}.
The eigenvalue problem used is
(analogous to \eqref{e:ef}),
\be
\frac{1}{\kstar}(\Phi(\kstar) + \Phi^\ast(\kstar))\, h = \mu h~.
\label{e:Phieig}
\ee
Although not stated as such, this
is solved with the Galerkin method \cite[Sec.~13.5]{LIE}
using a set of (MPS) global basis functions
$\xi_i(k):\overline\Omega \to \mathbb{C}$, $i=1,\ldots,N$,
each satisfying $(\Delta+k^2)\xi_i(k) = 0$ in $\overline{\Omega}$.
The original basis choice was plane waves
(which seem to require $\Omega$ to be convex \cite{mythesis});
since then, fundamental solutions \cite{que} and 
Fourier-Bessel wedge solutions
\cite{mush} have also been used to handle nonconvex domains with
one singular corner.
The action of $\Phi$ on $\xi_i|_\pO$
is known analytically because each $\xi_i$ is an interior Helmholtz solution.
Then the Galerkin approximation to \eqref{e:Phieig}
is the generalized eigenvalue problem
\be
\mbf{G}h^{(N)} = \mu^{(N)} \mbf{F}h^{(N)}~,
\label{e:scaling}
\ee
where the `mass' matrix $\mbf{F}$ has elements
$\mbf{F}_{ij} = \wang{\xi_i}{\xi_j}$,
and $\mbf{G}$ has elements $\mbf{G}_{ij} = 
(\wang{\xi_i}{ x\cdot\nabla \xi_j} + \wang{x\cdot\nabla \xi_i}{\xi_j})/\kstar$.
Further assuming that $\xi_i(k)(x) = \tilde\xi_i(kx)$, $x\in\overline\Omega$
i.e.\ the basis $k$-dependence is dilational, 
one may then check that $\mbf{G}=d\mbf{F}/dk|_{k=\kstar}$,
explaining Eq.~(2) of \cite{v+s}.
In practice, it is well known that good global bases are highly ill-conditioned
\cite{uwvf,monkwang},
thus $\mbf{F}$ and $\mbf{G}$ share a numerical nullspace.
Then \eqref{e:scaling} must be regularized, e.g.\ by projection onto
the numerical range of one of the matrices, 
in a similar fashion to \cite{mps,incl}. 

Reconstruction of eigenfrequencies is via 
$\hat k = \kstar  - 2/\mu$, and empirically has accuracy $O(\ep^3)$
\cite{mythesis}, not the $O(\ep^4)$ claimed in \cite{v+s}.
Eigenmodes are reconstructed from the
corresponding eigenvector components $h_i^{(N)}$ by ``undoing'' the dilation via
$\hat\phi = \sum_{i=1}^N h^{(N)}_i \xi_i(\hat k)$; 
boundary error $\|\hat\phi\|_\lpo$
is then dominated empirically by
$O(\ep^3)$ with unknown $k$-dependence \cite[sec.~6.3]{mythesis}.

\subsection{Connecting scaling and NtD methods via dilation}

In place of \eqref{e:xdphidil} one could instead write
$$
\xn \partial_n \phi_j(\kstar) = (1 + O(\epsilon)) f ~,
$$
which, with \eqref{e:phidil}, tells us that $f$
is an approximate eigenfunction of $\Theta(\kstar)$ with eigenvalue
$-\epsilon/k$.
It is this that motivated the authors to consider the weighted NtD flow%
---arguably more closely related to
spectral theory of the Laplacian on $\Omega$---as an alternative to dilation.

To connect the eigenfrequency estimators of the methods, we note
that $\Phi(k) = \ntd(k)^{-1} + W$, where $W$ is the tangential vector field in \eqref{Alex's-identity}, and
hence that $\Phi(k) + \Phi(k)^* = 2\ntd(k)^{-1} - m$, where $m$ is defined
by \eqref{e:m}.
Thus the operator appearing in \eqref{e:Phieig}
can be written as $(2/\kstar)\ntd(\kstar)^{-1}$
plus $\kstar^{-1}$ times a multiplication operator; this shows that  the eigenvalues of \eqref{e:Phieig} and $\ntd(\kstar)$ are related by 
$\mu = (2/\kstar)(\beta^{-1} + O(1))$ as $\kstar\to k_j$.
Thus one predicts that the scaling method has eigenfrequency accuracy
no better than that of \eqref{e:khatl}; this is observed numerically.

For eigenfunction error,
the authors are not aware of an explanation
of why in the scaling method the combination of
\eqref{e:Phieig} and reconstruction by dilation has error as high-order
as $O(\ep^3)$, as opposed to the naive $O(\ep)$.
Presumably the spectral flow of \eqref{e:Phieig} is very close
to the flow with $k$ of $\phi_j(k)|_\pO$ under exact dilation.
However, we may also
connect our quadratic NtD estimator \eqref{fhatimpsecondorder}
to this exact dilational flow.
Let $u$ be a Helmholtz solution, 
and let $f := \xn\dn{u}$ and $g := u|_{\dOmega}$ be Cauchy data for its dilation
$u(k)$ to frequency $k$.
Then one can check that $f$ and $g$ satisfy a second-order
evolution equation on $\pO$ of the form 
$$
\frac{d}{dk} \, {\vt{g}{f}} = L \vt{g}{f}
$$
where 
$$ L = \mt{L_{11}}{L_{12}}{L_{21}}{L_{22}} = 
\frac{1}{k}
\mt{W}{1}{-(x \cdot n)^2(\Delta_{\dOmega} + k^2) + W'}{ W + m }. 
$$
From this we can derive the first and second derivatives of $f$ when $g = 0$:
\begin{equation}\begin{aligned}
\dk{f} &=  \frac1{k} (W + m)f ; \\
\ddk{f} &=  \frac1{k^2} \Big( (W+m)^2 f - (W+m) f  - \xn^2 (\Delta_{\dOmega} + k^2) f   \Big). 
\end{aligned}\label{12scalingderivs}\end{equation}
Comparing to \eqref{12derivs}, we can see that the first derivative for this dilation flow at $\beta = 0$ agrees with the first derivative for the $\ntd(k)$ flow, up to an irrelevant normalization term. Moreover, the second derivative terms agree to highest order (if we agree that the $\ntd^{-1}(mf)$ term is lower order as per Remark~\ref{sqrtk}). Consequently \eqref{fhatimpsecondorder} corresponds to the dilation flow just as well as it does for the $\ntd(k)$ flow.

\subsection{Advantages of NtD method over the scaling method}
\label{s:adv}

Although the NtD and scaling methods have similar eigenfunction
error, share the same $O(N)$ acceleration factor and both are restricted
to star-shaped domains,
the NtD method has several advantages:
\bi
\item Higher-order accuracy in eigenfrequencies is possible (see section
\ref{s:khigher}), giving 3 to 5 extra correct digits in practice
(Fig.~\ref{f:kerr}).
\item Modes are reconstructed via \eqref{e:phirec},
without recourse to dilation (the latter requires continuation of basis
functions to a strip lying {\em outside} of $\Omega$).
\item
A formulation in terms of the NtD operator
allows rigorous estimates such as Propositions~\ref{p:khatl},
\ref{p:fperror} and \ref{p:higherorderbetaerror}.
\item The NtD method, as implemented in Sec.~\ref{s:basic},
is robust at all choices of $\kstar$, whereas
the scaling method is known to lose accuracy as $\kstar$ approaches
each Dirichlet eigenfrequency \cite{v+s,mythesis}.
\item Regularization of the numerically-singular pencil \eqref{e:scaling}
requires a choice of small parameter that is not fully understood
\cite{v+s,mythesis,que,mush}.
\item Our method
leverages known spectrally-accurate discretizations of
boundary integral operators,
whereas the Galerkin method \eqref{e:scaling} implicit in the scaling method
is limited by the accuracy of an available global MPS Helmholtz basis.
Success of the latter basis is {\it ad hoc}
and quite particular to the shape of $\Omega$.
\ei
However, on the last point,
we note that some global bases are much more efficient than BIE
because they need only 2--3 degrees of freedom per wavelength on the
boundary \cite{v+s,que}, and can be faster to evaluate than Hankel kernels.


\section{Conclusions}
\label{s:conc}

We have presented, analyzed, and tested a fast algorithm for computing
high-frequency
Dirichlet eigenvalues and eigenmodes of smooth star-shaped domains in
$\RR^d$.
The acceleration
is achieved by linearizing, over a frequency distance $\ep$,
the flow of the spectrum of the weighted NtD map.
The choice of weight function $\xni$ is crucial
since it equalizes the gradients in this flow and
prevents ``avoided crossings''.
$\ep$ controls both the total time to compute all modes lying in a given
frequency interval, and their resulting errors.
Windows of size $\ep$ are handled independently; the scheme is
``embarrassingly parallel''.
Maintaining bounded absolute eigenfrequency errors, one may choose $\ep=O(1)$,
giving a speed-up of $O(k^{d-1}) = O(N)$ over standard methods,
and more robustness since no root-search is needed.
This factor is in practice in $d=2$ roughly the
number of wavelengths across the domain; we show an example where it is $10^3$.

We proved robustness
(neither spurious nor missing modes, see Remark~\ref{r:robust}),
a leading third-order absolute accuracy in eigenfrequencies,
and, given a spectral nonconcentration assumption, third-order
$L^2$-errors of mode boundary functions.
This required developing some new results in the analysis of
elliptic PDE of interest in their own right.
Understanding the NtD spectral flow led to improved
estimators that are empirically fifth-order for eigenfrequencies,
and third-order for modes (with constant improved by factor $k^{1/2}$).
Our scheme has
many advantages over the scaling method (see section~\ref{s:adv}),
including an integral operator formulation,
rigorous error analysis, and much smaller eigenfrequency errors.

It is important to realize that the acceleration mechanism works at the
operator level, and is therefore independent of any further acceleration
that could be applied, such as:
block iterative solvers to extract the small negative matrix
eigenvalues (we used exclusively
dense direct solvers in this work), and fast multipole or fast direct solvers
to apply or compress the discretized operators.
However, since we are in a high-frequency regime (oscillatory kernel),
it is not at all obvious that fast solvers will make much difference;
testing this is an obvious next step.


Other natural questions for future work include the following:
\bi
\item Can the method be modified to remove the star-shaped restriction?
\item Can a modified method  (possibly using ideas from \cite{neubnds})
handle other homogeneous boundary conditions such as Neumann and Robin?
\item What accuracy can be reached for domains with corners in $d=2$ or $d=3$
using appropriate BIE discretizations? (Note that the scaling method
has been used with
nonsmooth boundaries \cite{v+s,que,mush}.)
%
\item Can boundary error bounds on $\hat f$ be extended to $\hat \phi$?
(see Remark~\ref{r:L2err}).
\item Can \eqref{fhatimpsecondorder} be analyzed, or improved upon in practice,
while preserving the $O(N)$ speed-up? One idea along these lines is
high-order extrapolation from a $\ep$-grid of $\kstar$ values;
analysis would need the spectral flow for {\em complex} $k$.
\ei

The reader is encouraged to try out the algorithms presented here
by downloading \mpspack\
from
{\tt http://code.google.com/p/mpspack}










\subsection*{Acknowledgments}
This work has benefited from discussions with
Timo Betcke, Doron Cohen, Lennie Friedlander, Rick Heller, and Eduardo Vergini.
AB acknowledges the support of the National Science Foundation through
grant DMS-0811005,
and is grateful for Visiting Fellowships
to the Mathematics Department, Australian National University
in February 2007 and February 2009.
AH acknowledges the support of the Australian Research Council through a Future Fellowship FT0990895 and Discovery Grant DP1095448 and thanks the Mathematics Department, Dartmouth College for its hospitality during a visit in July 2010.


\appendix

\section{Smoothness of eigenvalues and eigenprojections in $k$}
\label{a:A}

We are interested in the flow of eigenvalues and eigenprojections of the operator $\ntd$ in the parameter $k$. The operator $\ntd$ has a pole whenever $k^2$ is a Neumann eigenvalue of $\Omega$, and we wish to show that small negative eigenvalues and eigenprojections flow smoothly across such values of $k$. To do this we consider the Cayley transform of $\ntd$, as in \eqref{e:cayley1}.
Recalling \eqref{e:cayleyu} in the case $\eta=1$,
and solving for $f$ and $g$ in terms of $u_n$ and $\ub$,
we see that $R(k) f = g$ is equivalent to the existence of $u$ such that 
\begin{equation}
\big( \Delta + k^2) u = 0, \quad f = i \ub - \xn \dnub, \quad g = i \ub   + \xn \dnub. 
\label{e:ufR}
\end{equation}

\begin{proposition}\label{Cayleytransform} There is a neighbourhood $U \subset \CC$ of the positive real axis
such that there is a unique solution to the problem
\begin{equation}
(\Delta + k^2) u = 0 \text{ in } \Omega, \quad  i \ub - \xn \dnub = f 
\label{uf}\end{equation}
for every $k \in U$ and every $f \in L^2(\dOmega)$. Moreover, the solution $u = u(k)$ depends holomorphically on $k$ for $k \in U$. 
\end{proposition}

\begin{corollary}\label{cor:R(k)-analytic}
The Cayley transform $R(k)$ of $\ntd(k)$ is analytic in a neighbourhood $U \subset \CC$ of the positive real axis. 
\end{corollary}

Before we give the proof of this proposition we need a couple of preparatory lemmas. 

\begin{lemma}\label{lem:hom}
There is a neighbourhood $U \subset \CC$ of the positive real axis
such that for $k \in U$, the equation
\begin{equation}
(\Delta + k^2) u = 0 \text{ in } \Omega, \quad i \ub - \xn \dnub = 0
\label{hom}\end{equation}
has only the trivial solution.
\end{lemma}

\begin{proof} 
Write $k = a + ib$ with $a, b$ real. If $u$ satisfies \eqref{hom}, then 
we have 
\begin{equation}\begin{gathered}
- k^2 \int_\Omega |u|^2 + \int_\Omega |\nabla u|^2 = 
\int_\Omega (\Delta u) \, \ubar + \int_\Omega \nabla u \cdot \nabla \ubar \\ 
= \int_{\dOmega} \partial_n u \, \ubar = \int_{\dOmega} i \xni |u|^2.
\end{gathered}\end{equation}
Taking the imaginary part we find that 
\begin{equation}
- 2ab \int_\Omega |u|^2 = \int_{\dOmega}  \xni |u|^2.
\label{impart}\end{equation}
On the other hand, we can express $u$ in $\Omega$
via Green's representation formula \eqref{e:grf}.
It is standard that $\Sc(k)$ and $\Dc(k)$ are bounded operators from $L^2(\dOmega)$ to $\lo$, and it is straightforward to check that their norms are uniform in $k$ on compact subsets of the $k$-axis. 
Using the boundary condition for $u$ to replace $\partial_n u$ by $i \xni u$ in \eqref{e:grf}, we see that this gives
$$
\| u(k) \|_{\lo} \leq C(k) \| u \|_{L^2(\dOmega)}
$$
where $C(k)$ is uniform on compact subsets. But if we combine this with \eqref{impart}, then we see that for $|b|$ small enough compared to $a$, then \eqref{impart} has only the trivial solution $u = 0$. 
\end{proof}

\begin{lemma}\label{L} There is a compact operator $L : L^2(\Omega) \to L^2(\Omega)$ such that $Lz$ is the unique solution $u$ to the equation 
$$
\Delta u = z \text{ in } \Omega, \quad i \ub - \xn \dnub = 0.
$$
\end{lemma}

\begin{proof} We define operator $L_1$ to be inverse operator to the Dirichlet Laplacian on $\Omega$, i.e. $L_1 z$ is the function $u_1$ such that $\Delta u_1 = z$ in $\Omega$ with $u_1 |_{\dOmega} = 0$. It is standard that $L_1$ is well-defined and compact. We then try to solve
\begin{equation}
\Delta u_2 = 0 \text{ in } \Omega, \quad i u_2 |_{\dOmega} - \xn \partial_n u_2   |_{\dOmega} = \xn \partial_n u_1  |_{\dOmega};
\label{u2}\end{equation}
then $u_1 + u_2$ is the solution $u$ that we seek. Notice that 
\eqref{u2} implies that
$$
\xn \partial_n u_1  |_{\dOmega} = (i - B(0)) u_2  |_{\dOmega},
$$
where $B(0) = \xn \Lambda(0)$ is the weighted Dirichlet to Neumann operator at zero energy. The operator $B(0)$ is self-adjoint on $L^2(\dOmega)$ with our weighted inner product, so we can invert $i - B(0)$ and find that 
$$
u_2  |_{\dOmega} = (i - B(0))^{-1} (\xn \partial_n u_1  |_{\dOmega}).
$$
Finally, if $P$ is the classical Poisson operator taking functions on $\dOmega$ to the harmonic function in $\Omega$ with the given boundary value, then we have 
$$
u_2 = P \circ (i - B(0))^{-1} \xn \partial_n L_1 z.
$$
We recall some standard mapping properties of these operators. The operator $L_1$ maps $L^2(\Omega)$ to $H^2(\Omega)$, then $\xn$ times the normal derivative of this at the boundary maps to $H^{1/2}(\dOmega)$, then $(i - B(0))^{-1}$ is a pseudodifferential operator of order $-1$, hence maps $H^{1/2}(\Omega)$ to $H^{3/2}(\Omega)$, while $P$ maps $H^{3/2}(\Omega)$ to $H^{2}(\Omega)$ \cite[Ch. 5, Prop. 1.7]{TaylorI}. Denote the composite operator $L_2$, i.e. $u_2 = L_2 z$. Then we see that $L_2$ maps $L^2(\Omega)$ continuously to $H^2(\Omega)$, and hence, using the compact embedding of $H^{2}(\Omega)$ into $L^2(\Omega)$, we see that $L_2$ is compact on $L^2(\Omega)$. Hence $L = L_1 + L_2$ is compact. Uniqueness of the solution follows from Lemma~\ref{lem:hom} with $k=0$. This completes the proof of Lemma~\ref{L}. 
\end{proof}

\begin{proof}[Proof of Proposition]
Let $U$ be as in Lemma~\ref{lem:hom}. Then this lemma guarantees the uniqueness of $u$ satisfying \eqref{uf}. It remains to establish existence. To do this, we first find $w_1$ such that 
$$
\Delta w_1 = 0 \text{ in } \Omega, \quad  i w |_{\dOmega} - \xn d_n w_1 |_{\dOmega} = f 
$$
which is done exactly as in \eqref{u2}. Then we look for $w_2$ satisfying
$$
(\Delta + k^2) w_2 = -k^2 w_1, \quad i w_2 - \xn \partial_n w_2 |_{\dOmega} = 0.
$$
If we can find such a $w_2$, then $u = w_1 + w_2$ is our required
solution of \eqref{uf}.  
Using the operator $L$ from Lemma~\ref{L}, this can be written
$$
w_2 = -L(k^2 w_1 + k^2 w_2),
$$
which is equivalent to 
$$
\big( \Id + k^2 L \big) w_2 = -k^2 L w_1.
$$
Thus we get a solution provided that $\Id + k^2 L$ is invertible. Since $L$ is compact, this will be the case provided that $\Id + k^2 L$ has trivial null space. But if $v$ is in the null space of this operator then $v$ satisfies \eqref{hom}, which means by Lemma~\ref{lem:hom} that indeed $v = 0$. Therefore the null space is trivial, so $\Id + k^2 L$ is invertible and we can find $w_2$ as above. This establishes existence of $u$. Finally, using the compactness of $L$ and analytic Fredhom theory \cite[Thm. VI.14]{RSv1}, for $k \in U$, $(\Id + k^2 L)^{-1}(-k^2 L w_1)$ is analytic in $k$, showing that $u(k)$ is analytic in $k$. 
\end{proof}

It follows from the analyticity of $R(k)$ 
that in any interval $I$ of the unit circle in which the spectrum of $R(k)$ is discrete at $k = k_0$, the eigenvalues of $R(k)$ in $I$ are analytic as a function of $k$, and one can choose an orthonormal basis of the corresponding eigenspaces that varies analytically \cite[Ch. VII, sec.~3]{Kato}. This implies that the eigenspaces of $\Theta(k)$ vary analytically   in any interval where the spectrum is discrete, with the exception of a finite number that have a pole at each Neumann eigenfrequency. Since $\Theta(k)$ is a pseudodifferential operator of order $-1$ and therefore compact, this means that the eigenspaces vary analytically except when the eigenvalue hits zero. In fact, we can say more. Before we state the next proposition, recall that the eigenvalues of $\Theta(k)$ are monotonic increasing in $k$ --- see \eqref{betadotpositive}. 

\begin{proposition}\label{prop:eig-branches}

(i) Suppose that $0$ is an eigenvalue of $\ntd(k_*)$. Then there is an analytic eigenvalue branch $\beta(k)$ with $\beta \uparrow 0$ as $k \uparrow k_*$, and the multiplicity of the $0$ eigenspace is equal to the sums of the multiplicities of all such branches. 

(ii) Conversely, suppose that $\beta(k)$ is an eigenvalue branch of $\Theta(k)$ tending to zero as $k \uparrow k_*$. Then  $k_*$ is a Dirichlet eigenfrequency,  the eigenprojection has a limit as $k \uparrow k_*$, and it is the projection onto a subspace of the space of weighted normal derivatives of Dirichlet eigenfunctions with eigenfrequency $k_*$. The eigenvalue $\beta(k)$ is $C^1$ as a function of $k$ up to and including $k = k_*$, and satisfies \eqref{e:invk}. 
Finally, if the eigenvalue $\beta(k)$ is simple up to and including $k = k_*$, then the eigenfunction $f(k)$ is $C^2$ up to an including $k = k_*$, and satisfies \eqref{12derivs}.
\end{proposition}

\begin{proof}

(i) Suppose that $0$ is an eigenvalue of $\ntd(k_*)$, with eigenspace $V$. Choose an interval $(a, b)$ containing $0$, with neither $a$ nor $b$ in the spectrum of $\ntd(k_*)$, and such that there are no eigenvalues in the interval $(a, 0)$, and let $\Pi_{a,b}(k)$ denote the projection onto the eigenspaces of $\ntd(k)$ with eigenvalues in the interval $(a,b)$. Define $\tilde \ntd(k) = \Pi_{a,b}(k) \ntd(k)$. Then for $k$ close to $k_*$, $\tilde \ntd(k)$ is an analytic family, again using \cite[Ch. VII, sec.~3]{Kato}.  By the calculation in Lemma~\ref{lem:Fried}, $d\tilde \ntd(k)/dk$ is a positive operator. Since 
$$
\ang{\tilde \ntd(k_*) f, f} = 0 \text{ for all } f \in V,
$$
we have 
$$
\ang{\tilde \ntd(k) f, f} < 0 \text{ for all } f \in V, \ k < k_*.
$$
So there are at least $(\dim V)$ negative eigenvalues of $\tilde \ntd(k)$, which tend to $0$ as $k \uparrow k_*$. These branches are analytic for $k < k_*$ since the negative spectrum of $\tilde \ntd(k)$ is discrete. The statement that there are exactly $(\dim V)$ negative eigenvalues of $\tilde \ntd(k)$ which tend to $0$ as $k \uparrow k_*$ follows from the proof of (ii) below.

(ii)
For simplicity, we first prove (ii) assuming that $\beta(k)$ is simple. In
that case, taking the eigenfunction $f(k)$ to be normalized in
$L^2(\dOmega)$, we see from the identity \eqref{qf-identity} that the
extended eigenfunction $u(k)$ is uniformly bounded in $H^1(\Omega)$ as $k
\uparrow k_*$. Therefore, there is a sequence $k_i$ tending upward to
$k_*$ such that $u(k_i)$ has a weak limit $v$ in $H^1(\Omega)$, and therefore,
a strong limit in $L^2$, along this sequence. 
It also follows from \eqref{betadotpositive} and \eqref{betadot} that the
$L^2$ norm of $u(k_i)$ does not tend to zero along this sequence, so $v$ is
nonzero. 
From the fact that the $u(k_i)$ are Helmholtz, we find that along this
sequence, we have 
$$
\lim \int_\Omega u(k_i) (\Delta + k_i^2) \psi = 0 \quad \text{for all } \psi \in C_c^\infty(\Omega),
$$
implying that $v$ is a weak solution of the equation $(\Delta + k_*^2) v = 0$. By elliptic regularity, this means that $v$ is smooth in the interior of $\Omega$ and satisfies the equation in the strong sense. Also, using the continuous map from $H^1(\Omega) \to L^2(\dOmega)$ given by restriction to the boundary, we see that $v |_{\dOmega}$ is the weak limit (in $L^2(\dOmega)$) of $u(k_i) |_{\dOmega}$. But  $u(k_i) |_{\dOmega}$ tends \emph{strongly} to zero (since $u(k_i) |_{\dOmega} = \beta(k_i) f(k_i)$ and $\beta(k_i)$ tends to zero) and \emph{a fortiori} weakly, so $v$ is zero at the boundary. It follows that $v$ is a Dirichlet eigenfunction. We see that $u(k)$ has a continuous extension to $k = k_*$, such that it is a Dirichlet eigenfunction at $k = k_*$. That is, $0$ is an eigenfunction of $\Theta(k_*)$, so the eigenvalue branch $\beta(k)$ extends continuously to $k = k_*$. Given \eqref{betadot-est}, we see that $\dk \beta$ has a limit $1/k_*$ as $k \uparrow k_*$, and therefore, $\beta(k)$ is $C^1$ up to an including $k = k_*$, and \eqref{e:invk} holds. 

If $\beta$ is a multiple eigenvalue, we proceed similarly. We take a sequence of extended eigenfunctions as before, and produce a Dirichlet eigenfunction $v_1$. Next we take another sequence of extended eigenfunctions orthogonal (at the same value of $k$) to the first sequence, and produce another Dirichlet eigenfunction $v_2$, and so on. We find a subspace of Dirichlet eigenfunctions at frequency $k_*$ of dimension equal to that of the multiplicity of $\beta(k)$. 

Again assuming that the eigenvalue $\beta(k)$ is simple up to and including $k = k_*$, let $\eta$ be a positive number such that Assumption~\ref{ASC} holds in some interval $[k_* - \delta, k_*]$ for some $\delta > 0$. Then Lemma~\ref{lem:fdotsize} and Lemma~\ref{lem:utu}, show that $\dot f(k)$ is uniformly bounded as $k \uparrow k_*$, and hence $f(k)$ has a limit as $k \to k_*$. 
Now referring to \eqref{fdot}, using the continuity of $f(k)$ just shown, Lemma~\ref{lem:utu} and \eqref{higherderivs} to bound derivatives of $f(k)$, and Proposition~\ref{prop:gen-inverse} to control the norm of $(\Theta(k) - \beta)^{-1}$ uniformly as $k \uparrow k_*$, we see that  $\dot f$ itself is continuous up to $k = k_*$. Iterating once more using \eqref{fdot}, we see that $\ddot f$ is continuous up to $k = k_*$. Hence $f(k)$  is $C^2$ up to $k = k_*$ and formula \eqref{fdot} extends by continuity to $k = k_*$ to yield \eqref{12derivs} when $\beta = 0$. 
\end{proof}

\section{Computation of reference eigenfrequencies and eigenmodes}
\label{a:ref}

Here we describe our implementation of a
standard published method for computation of eigenpairs,
which we use as a reference to assess both accuracy and speed of the NtD method.
Recalling the definition \eqref{e:D}, we have the following standard
result (e.g. see \cite[Lemma~8.4]{mitrea} which applies for
domains with Lipschitz boundary; note the opposite sign convention).

\begin{lemma} 
\label{l:mitrea}
A positive frequency $k$ is a Dirichlet
eigenfrequency if and only if the operator $(\half-D^\ast(k))$ has a
non-trivial nullspace.
Furthermore, its nullspace is precisely the
space of boundary normal derivatives of solutions of
$(\Delta + k^2)u=0$ in $\Omega$ with homogeneous Dirichlet data on the
boundary.
\end{lemma}
Its proof uses the jump relations and
the uniqueness of the exterior Helmholtz Neumann boundary value problem
\cite{CK83,mitrea}.
The numerical method is then,
following B\"acker \cite[sec.~3.3]{backerbim},
to search along the $k$ axis for (near) zeros of the lowest singular value of
a matrix discretization of the operator $(\half-D^\ast(k))$.
We use the Nystr\"om quadrature as in \eqref{e:Kst}--\eqref{e:Kmat};
the same $N$ as before may be used to achieve quadrature
errors around machine precision.
The cost of each minimum singular value evaluation is then $O(N^3)$.
(We note that finding roots of the determinant is faster but
is not able to distinguish close eigenfrequencies
or handle degeneracies reliably \cite{backerbim}).

The minimum singular value, which we call $t$, as a function of $k$,
has the form of a series of V-shapes
with the bottom of each `V' approaching zero
(e.g. see Fig.~8 of \cite{Mart05} or Fig.~5.1 of \cite{incl}).
Reliably locating all such minima in a range of $k$ is not trivial,
crudely speaking because close eigenfrequencies lead to small-scale W-shapes
that are difficult to distinguish from a `V'.
We make use of the empirical observation that the slope of $t(k)$ appears to
have an upper bound $C_t$ of size $O(1)$ which depends only on $\Omega$,
and that most of the V-shapes have this slope.
We initially evaluate $t(k)$ on a
regular grid of spacing about 0.2 times the mean
eigenfrequency spacing.
At each local minimum on this grid we use the information
about higher singular values to decide whether to i) perform
fitting of a parabola to the three neighbouring
samples of $t^2(k)$, and iterate this fit procedure
until convergence, or ii) recursively call the same routine on an
(about 3 times) finer grid covering three (or more, if there are
nearby small values of $t$) neighbouring grid points.
We omit several details of the algorithm required for robustness.

This has
been coded into \mpspack\ and may be run
(for instance for the example of section~\ref{s:khat})
via
\begin{verbatim}
o.maxslope = 1.5; o.tol = 1e-12; p.solvespectrum([90 100], 'ms', o);
\end{verbatim}
where {\tt maxslope} defines the value $C_t$,
and {\tt tol} the requested absolute tolerance on $k$.
When $C_t$ is chosen correctly, the algorithm finds all $k_j$ in
a given $k$ interval, needing around 15 evaluations per simple
eigenfrequency found, and typical errors are $10^{-13}$ or less.
When eigenfrequencies are {\em degenerate}, many more recursions are needed
to establish reliably that they are not distinct; for instance
at {\tt o.tol = 1e-6} it still requires around 50
evaluations per multiple eigenfrequency found (this scales like log {\tt tol}),
and typical errors are 
$10^{-7}$.

Once accurate eigenfrequencies have been found, modes are found as follows.
For each $k_j$, the last right singular
vector of the above matrix is computed at a cost of $O(N^3)$;
according to Lemma~\ref{l:mitrea}
this approximates $\partial_n \phi_j$ at the quadrature nodes.
Normalization is done via \eqref{e:rellich}.
Eigenfunctions $\phi$ may then be reconstructed via \eqref{e:grf}.
In the case of an $p$-fold degeneracy, the last $p$ right singular vectors
are used.
The whole method thus scales as $O(N^3)$ per mode with a rather large constant.


%
%
%


\section{Proof of Lemma~\ref{lem:utu}}
\label{a:AppB}

\begin{proof}[Proof of Lemma~\ref{lem:utu}] 
We prove the theorem under very slightly more general conditions. That is, we replace  $\xn$ --- both in the boundary condition \eqref{robin-bc} and in the weight factor in the inner product on $L^2(\dOmega)$ --- by an arbitrary smooth positive weight, which we denote $m$. 
First, we introduce a spectral cutoff. Since we are using a weighted inner product we define the operator
$$
\Deltabw =  \nabla_{\tan}^{*,w} \nabla_{\tan}
$$
on $L^2(\dOmega)$, where ${}^{*,w}$ denotes the adjoint with respect to the weighted inner product. (Below, we write ${}^*$ instead of ${}^{*,w}$ but all adjoints in this appendix should be understood to be with respect to the weighted inner product.)
We write 
$$
\Id = \Psi(\Deltabw/k^2) + (1 - \Psi)(\Deltabw/k^2), \quad \text{ on } L^2(\dOmega),
$$
where $\Psi(t)$ is $1$ for $t \geq 3/2$ and $0$ for $t \leq 5/4$. 
Let us write $\Psi$ for $\Psi(\Deltabw/k^2)$ below; note that 
$\Psi$ is a semiclassical pseudodifferential operator of order $(0,0)$,\footnote{A  operator with parameter $h$  is a semiclassical pseudodifferential operator of  order $(l,m)$ on $\dOmega$ if its Schwartz kernel can be written locally (that is, with respect to some local coordinate patch $y = (y_1, \dots, y_{d-1})$) in the form 
$$
h^{-(d-1) - l} \int_{\RR^{d-1}} e^{i(y-y') \cdot \eta/h} a(y, \eta, h) \, d\eta,
$$
where $\eta \in \RR^{d-1}$ and the symbol $a$ is smooth in $h$ and satisfies symbol estimates 
$$
\big| \partial_y^\alpha \partial_\eta^\gamma a(y, \eta, h) \big| \leq C_{\alpha, \gamma} \big( \sqrt{1 + |\eta|^2} \big)^{m - |\gamma|}.
$$
Here the parameter $h$ is $k^{-1}$. }
 supported where $|\eta| \geq 3/2$.  
We can expect that the $L^2$ norm of $(\Id - \Psi) \nabla_{\tan} u$ is bounded by $2k$ times that of $u$, since applying $(\Id - \Psi)$ removes frequencies of order $\geq 2k$. To verify this, given a vector field $W$ of unit length and tangential to the boundary, we compute 
\begin{equation}\begin{gathered}
\| (\Id  - \Psi) W u \|_{L^2(\dOmega)}^2 
= \| W ( \Id - \Psi) u + [1-\Psi, W] u \|_{L^2(\dOmega)}^2 \\
\leq \frac{4}{3} \| W ( \Id - \Psi) u  \|_{L^2(\dOmega)}^2 + 4 \| [1-\Psi, W] u \|_{L^2(\dOmega)}^2  \\
= \frac{4}{3} \ang{ (\Id - \Psi) W^* W (\Id - \Psi) u, u} + 4 \| [1-\Psi, W] u \|_{L^2(\dOmega)}^2 . 
\end{gathered}\label{1-Psi^2}\end{equation}
Notice that $[1-\Psi, W]$ is a semiclassical pseudodifferential operator of order $(0, -\infty)$, hence with uniformly bounded (in $k$) $L^2(\dOmega) \to L^2(\dOmega)$ operator norm. If we sum over an orthonormal basis $W_1, \dots, W_{n-1}$, then using $\sum_i W_i^* W_i = \Deltabw$ and the fact that $1 - \Psi(t)$ vanishes when $t \geq 3/2$, we have
\begin{equation}\begin{gathered}
\| (\Id  - \Psi) \nabla_{\tan} u \|_{L^2(\dOmega)}^2 
\leq 2 \ang{ (\Id - \Psi) \Deltabw (\Id - \Psi) u, u} + 2 \sum_i \| [1-\Psi, W_i] u \|_{L^2(\dOmega)}^2  \\
\leq (2k^2 + C) \| u \|_{L^2(\dOmega)}^2. 
\end{gathered}\label{lowenergy}\end{equation}

Next we analyze the high energy part, $\Psi \nabla_{\tan} u$. 
We use the single and double layer boundary integral operators $S(k)$
defined by \eqref{e:S}, 
and $D(k)$ defined by \eqref{e:D}.
We also write $Q(k)$ for the
(hypersingular)   
operator $ \partial_{n_x} \partial_{n_y} G_0(k;x,y)$ restricted to the boundary in both variables.

We now quote results from \cite[Section 4]{hassell}. Here it is shown that $S(k)$ and $D(k)$ are pseudodifferential operators of order $(-1, -1)$ in the `elliptic region' $ \{ |\eta| > 1 \}$ (where $|\eta|$ is the length of $\eta$ with respect to the induced boundary metric on $\dOmega$), in the sense that if $\Phi$ is a semiclassical pseudodifferential operator of order $(l,m)$, microsupported in the elliptic region, then  $\Phi S(k)$ and $\Phi D(k)$ are semiclassical pseudodifferential operators of orders $(l-1, m-1)$. Moreover, the  analysis from \cite[Section 4]{hassell} applies to $\Phi Q(k)$ which shows that $\Phi Q(k)$ is a semiclassical pseudodifferential operator of order $(l+1,m+1)$, with principal symbol $- \half k \sigma(\Phi) \sqrt{1 - |\eta|^2}$ where $\sigma(\Phi)$ is the principal symbol of $\Phi$.   (See Remark~\ref{rem:layerorders} in case this is confusing.)

For any Helmholtz solution $u$ we have the Green's representation formula
\eqref{e:grf}.
By differentiating normally at the boundary $\dOmega$, we obtain, using
\eqref{e:jump-tr},
\begin{equation}
\partial_n u(x) =  (D(k)^t + \half)  \partial_n u - Q(k) u  .
\label{dnuQ}\end{equation} 
Let us write $\tilde D(k)$ for the kernel $m^{-1} D(k) m$. We then obtain from \eqref{dnuQ} and the boundary condition \eqref{robin-bc} that 
\begin{equation}
\half \beta m \partial_n u =   \beta (\tilde D(k))^t (m \partial_n u) - \beta m Q(k) u \\
\implies u =  2(\tilde D(k))^t u - 2\beta m Q(k) u .
\label{uintermsofitself}\end{equation}

Next we differentiate tangentially, apply $\Psi$, and take the inner product with $\Psi W u$, where $W$ is a tangential vector field of unit length. We obtain 
\begin{equation}
\ang{\Psi W u,\Psi W u} =  2\ang{\Psi^2 W (\tilde D(k))^t u, W u} - 2\beta\ang{ W^*\Psi^2 W m Q(k) u,  u}.  
\label{firstsecond}\end{equation}
Using the results of \cite{hassell} mentioned above, we see that 
$\Psi^2 W (\tilde D(k))^t$ is a semiclassical operator of order $(0,0)$, hence bounded on $L^2(\dOmega)$ uniformly in $k$. (Here we use the property of $\Psi$ that it is microsupported in the elliptic region, in fact in the region $\{ |\eta| \geq 4/3 \}$.)  Hence the first term in \eqref{firstsecond} is estimated by 
\begin{equation}
C \| u \|_{L^2(\dOmega)} \| \nabla_{\tan} u \|_{L^2(\dOmega)}.
\label{est1}\end{equation}

In the second term, we have the operator $W^*\Psi^2 W m Q(k)$. Since $W^*\Psi^2 W m$ is also a pseudodifferential operator microsupported in the elliptic region, and since $W$ and $W^*$ are of pseudodifferential operator $(1,1)$, we see that $W^*\Psi^2 W m Q(k)$  is a pseudodifferential operator of order $(3,3)$, with 
principal symbol 
\begin{equation}
- k^{3} |\sigma(hW)|^2 m \psi^2(\eta) \sqrt{|\eta|^2 - 1}.
\label{e:signedterm}\end{equation}
 This is minus the square of a smooth symbol, namely $$k^{3/2} m^{1/2} \psi(\eta) \sigma(ihW) (|\eta|^2 - 1)^{1/4}.$$ 
 The sign of \eqref{e:signedterm} is crucial, as it will effectively allow us to discard this term, which would otherwise be too big to estimate. 
 This works as follows: by the pseudodifferential calculus, we have $W^*\Psi^2 W m Q(k) = -  B^* B + A_2$, with $A_2, B$ semiclassical pseudos, $B$ of order $(3/2, 3/2)$ and $A_2$ of order $(2,2)$.  
This term can therefore be expressed 
$$
2 \beta \| B  u \|_2^2  -2\beta \ang{u, A_2 u}. 
$$
Since $A_2$ is supported in $\{ |\eta| \geq 3/2 \}$ we can write
$$
A_2 = \nabla_{\tan}^* A_0 \nabla_{\tan} + A_1
$$
with $A_0$ of order $(0,0)$ and $A_1$ of order $(1,1)$. Since $A_1$ can be chosen to be microsupported in the elliptic region, we have $A_1 = A_0' \cdot \nabla_{\tan}  + A_0''$, where $A_0'$ and $A_0''$ are pseudos of order $(0,0)$. Thus we have 
$$
- \beta\ang{ W^*\Psi^2 W m Q(k) u,  u} = 2\beta \| B  u \|_2^2  -2\beta \Big( \ang{\nabla_{\tan} u, A_0 \nabla_{\tan} u} + \ang{A_0' \Deltabw^{1/2} u + A_0'' u, u} \Big) .
$$
This gives an estimate for the second term of \eqref{firstsecond} of the form  
\begin{equation}
 \beta \| B  u \|_2^2 + C |\beta| \Big(  \|  \nabla_{\tan} u \|_{L^2(\dOmega)}^2 +  \|  u \|_{L^2(\dOmega)}^2  \Big) 
 \label{est2}\end{equation}
 (where we dropped the term $C |\beta| \|  u \|_{L^2(\dOmega)} \| \nabla_{\tan} u \|_{L^2(\dOmega)}$ since it is controlled by the other terms on the right hand side). 
 Combining \eqref{est1} and \eqref{est2}, we find that 
$$\begin{gathered}
 \| \Psi W u \|_2^2 \leq \beta \| B  u \|_2^2 + C |\beta| \Big(  \|  \nabla_{\tan} u \|_{L^2(\dOmega)}^2 +  \|   u \|_{L^2(\dOmega)}^2 \Big) + C \| u \|_{L^2(\dOmega)} \| \nabla_{\tan} u \|_{L^2(\dOmega)}   \\
 \leq C |\beta| \Big(  \|  \nabla_{\tan} u \|_{L^2(\dOmega)}^2 +  \|   u \|_{L^2(\dOmega)}^2 \Big) + C \| u \|_{L^2(\dOmega)} \| \nabla_{\tan} u \|_{L^2(\dOmega)}
 \end{gathered} $$
 where we are able to discard the $B  u$ term since 
$\beta < 0$. Combining this with \eqref{lowenergy} and using the 
inequality $\|a+b\|^2 \leq 3\|a\|^2/2 + 3\|b\|^2$ we find that 
$$
\| Wu \|_2^2 \leq (3k^2 + C) \| u \|_{L^2(\dOmega)}^2 +  C \| u \|_{L^2(\dOmega)} \| \nabla_{\tan} u \|_{L^2(\dOmega)}  + C |\beta|  \|  \nabla_{\tan} u \|_{L^2(\dOmega)}^2,
$$
and summing over an orthonormal basis of $W$ we find that 
$$
\| \nabla_{\tan} u \|_2^2 \leq (3k^2 + C) \| u \|_{L^2(\dOmega)}^2 +  C \| u \|_{L^2(\dOmega)} \| \nabla_{\tan} u \|_{L^2(\dOmega)}  + C |\beta|  \|  \nabla_{\tan} u \|_{L^2(\dOmega)}^2. 
$$
Finally we write 
$$
C \| u \|_{L^2(\dOmega)} \| \nabla_{\tan} u \|_{L^2(\dOmega)} 
\leq 4C^2 \| u \|_{L^2(\dOmega)}^2 + \frac1{16} \| \nabla_{\tan} u \|_{L^2(\dOmega)} ^2,
$$
and observe that for $|\beta| \leq C/16$ we can absorb the $\| \nabla_{\tan} u \|_2^2$ terms on the left hand side to deduce 
$$
\| \nabla_{\tan} u \|_2^2 \leq \frac{8}{7} (3k^2 + 4C^2) \| u \|_{L^2(\dOmega)}^2 . 
$$
For $k \geq K$, we have $8(3k^2 + 4C^2)/7 \leq 4 k^2$ and we arrive at 
\eqref{u_t-in-terms-of-u}. 
\end{proof}

\begin{remark}\label{rem:layerorders} Since $D(k)$ involves one extra  derivative than $S(k)$, it might seem peculiar that $S(k)$ and $D(k)$ are both order $(-1,-1)$ in the elliptic region. In fact,
the \emph{distributional} limit of $\partial_{n_y} G_0(k; x, y)$ to the boundary in both variables is a pseudodifferential operator of order $(0,0)$ in the elliptic region, but the leading part of this operator is half the identity --- supported at the diagonal --- so it does not appear when the kernel function is restricted to the boundary in both variables; rather this part of the operator shows up as the $\half$ in the jump formula \eqref{e:jump}. This does not happen for $Q(k)$, hence its  order is two more than that of $S(k)$, as expected. 
\end{remark}


\section{Estimates involving $(\ntd - \beta)^{-1}$}
\label{a:estinv} 

To prepare for this operator norm estimate we first generalize an estimate from \cite{bnds} from the Dirichlet boundary condition to the `near Dirichlet' Robin boundary conditions, that is the boundary condition
\begin{equation}
\ub - \beta \xn \dnub = 0, \quad -\frac{\delta}{k} \leq \beta \leq 0,
\label{betabc}\end{equation}
where $\delta$ is small and $\beta < 0$. 
This is a self-adjoint boundary condition and there is a corresponding orthonormal basis of  eigenfunctions $\phi_j^\beta$, with eigenvalues $E_j^\beta = (k_j^\beta)^2$. 
Then we have

\begin{proposition}\label{qowbeta}
Let $\dOmega$ be  smooth. Then there exists $\delta > 0$  and a constant $C_\Omega$ depending only on $\Omega$ and $\delta$ such that the operator norm 
\begin{equation}
\Big\| \sum_{|k - k_j^\beta| \leq 1} \xn \partial_n \phi^\beta_j \wang{\xn \partial_n \phi^\beta_j}{\cdot} \Big\|_{L^2(\dOmega), \xni d\sigma}
\label{Top}\end{equation}
is bounded by $C_\Omega k^2$, uniformly for $\beta$ in the range
$[-\delta/k, 0]$. 
\end{proposition}

\begin{proof} We use the same method of proof as in \cite{bnds}, but with additional work since we no longer have our functions vanishing at the boundary. 

Our starting point is the identity 
\begin{equation}\begin{gathered}
\int_\Omega \phi [\Delta + k^2, V] \phi = \int_\Omega (\Delta + k^2) \phi V \phi  \\
- \int \phi V (\Delta + k^2) \phi + \int_{\dOmega}  \phi \partial_n V \phi -  \partial_n \phi V \phi .
\end{gathered}\label{comm-identity}\end{equation}
If we choose $V$ to be a smooth vector field equal to $\xn \partial_n$ at the boundary, then the last term in \eqref{comm-identity} is $\| \xn \partial_n \phi \|^2$. We can therefore express  
\begin{equation}\begin{gathered}
\| \xn \partial_n \phi \|^2 = \int_{\dOmega} \xn |\partial_n \phi|^2 \\ = \int_{\dOmega} \phi \partial_n \xn \partial_n \phi + \int_{\Omega} (\Delta + k^2) \phi V \phi - \int_\Omega \phi [\Delta + k^2, V] \phi - \int \phi V (\Delta + k^2) \phi \\
= I + II - III - IV . 
\end{gathered}\label{b-identity}\end{equation}

The first step in the proof of Proposition~\ref{qowbeta} is to  estimate this squared $L^2$ norm when $\phi$ is an approximate eigenfunction, that is, a function satisfying
\begin{equation}
\| \phi \|_{L^2(\Omega)} = 1, \quad \phi + \beta \xn \partial_n \phi = 0 \text{ at } \dOmega, \quad \big\| (\Delta + k^2) \phi \big\|_{L^2(\Omega)} = O(k).
\label{approxefn}\end{equation}
We claim that this implies that 
$$
\| \xn  \partial_n \phi \|_{L^2(\dOmega)} \leq C k,
$$
which we prove by estimating  terms $I$ --- $IV$ in \eqref{b-identity} by 
 by a constant times $k^2$. Before doing so,
observe that using \eqref{betabc} and Lemma~\ref{lem:utu} any boundary term
of the form 
$$
\int_{\dOmega} a_1 k^2 |\phi|^2 +a_2 |\nabla_{\tan} \phi|^2 + k a_3 \phi |\nabla
_{\tan} \phi| + k a_4 \phi \partial_n \phi + a_5 \partial_n \phi |\nabla_{\tan} \phi| 
$$
where $a_i$ are bounded functions on $\dOmega$, not depending on $\phi$ or $k$,
can be estimated by $C\delta \| \partial_n \phi\|_2^2$ and therefore (for sufficiently
small $\delta$) can be absorbed in the left hand side; we will call them
`acceptable' boundary terms. 

 Consider the identity
\be
 \int_\Omega |\nabla \phi|^2 = \int_\Omega (-\Delta \phi) \overline{\phi} + \int_{\dOmega} \partial_n \phi \overline{\phi}
~.
\label{qf-identity}
\ee
For $\beta \leq 0$,  the last term is negative using the boundary condition \eqref{betabc}, implying (using  \eqref{approxefn}) that
\begin{equation}
\| \nabla \phi \|_{L^2(\Omega)} = O(k) \implies \| V\phi \|_{L^2(\dOmega)} = O(k). 
\label{gradest}\end{equation}
Using this we see that the term $II$ on the right hand side of \eqref{b-identity} is $O(k^2)$.

Term $IV$ can be expressed after integration by parts as
$$
\int_\Omega V \phi (\Delta + k^2) \phi + (\operatorname{div} V \phi) (\Delta + k^2) \phi - \int_{\dOmega} \xn (V \cdot n) \phi (\Delta + k^2) \phi. 
$$
The first two terms are dealt with as above. In the third term, we expand $\Delta = \partial_n^2 + \Deltab + (d-1) H \partial_n$. Notice that the $\partial_n^2$ term cancels  term $I$ in \eqref{b-identity} up to an acceptable boundary term. So we have to estimate the terms 
$$
\int_{\dOmega} \phi (\Deltab + (d-1) H \partial_n + k^2) \phi.
$$
The $H \partial_n$ and $k^2$ terms are acceptable.  The $\Deltab$ term is estimated by integrating by parts to convert the integrand  to $|\nabla_{\tan} \phi|^2$ which is also acceptable. 

To estimate term $III$ we use the fact that $[\Delta, V]$ is a second order operator and therefore of the form of a finite sum $\sum V_i W_i$ where $V_i$, $W_i$ are smooth vector fields. We can integrate by parts modulo an acceptable boundary term and obtain
$$
\int_{\dOmega} \sum V_i \phi W_i \phi,
$$
and the $L^2$ norm is bounded by $C \| \nabla u \|_{L^2(\Omega)}^2 = O(k^2)$. This completes the proof that $\| \xn \partial_n \phi \|_{L^2(\dOmega)}^2 = O(k^2)$.

 The second step of the proof is the same as in \cite{bnds}. For the reader's convenience we repeat the argument here. We define an operator $T$ from the range of the spectral projector $1_{[k-1, k+1]} (\sqrt{-\Delta})$, that is, the vector space spanned by eigenfunctions of $-\Delta$ with eigenvalues in the range $[(k-1)^2, (k+1)^2]$, to $L^2(\dOmega)$, by 
 $$
T \phi = \xn \partial_n \phi |_{\dOmega}. $$ 
Then, any such $\phi$ satisfies
$$
\big\| (\Delta + k^2) \phi \big\|_{L^2(\Omega)} \leq (2k+1) \big\|  \phi \big\|_{L^2(\Omega)},
$$
meaning that $\phi$ is (after normalization) an approximate eigenfunction 
in the sense of \eqref{approxefn}.  
So in the first step of the proof above, we showed that $\| T\phi \| \leq C k \| \phi \|$, or in other words that $T$ has operator norm at most $Ck$.  But the operator norm of $T$ is equal to that of $T^*$, and the operator that appears in \eqref{Top} is precisely $T T^*$. Since the operator norm of $T T^*$ is precisely the square of the operator norm of $T$,  this completes the proof of the theorem. 
\end{proof}

We now prove an estimate on the term involving the generalized inverse $(\ntd - \beta)^{-1}$ in \eqref{fdot}. 

\begin{proposition}\label{prop:gen-inverse} Let $\beta$, for $k \in [\kstar, \kstar + \epsilon]$, be an eigenvalue of $\ntd(k)$ satisfying $-\epsilon/k \leq \beta < 0$ and satisfying Assumption~\ref{ASC} at  the scale $\eta$. Assume also that $\eta$ satisfies \eqref{etaconditions}.  Then there exists $C$ depending only on $\Omega$ such that for any $z \in H^1(\dOmega)$, 
we have the estimate 
\begin{equation}
\big\| (\ntd(k) - \beta)^{-1} z \big\|_{L^2(\dOmega)} \leq C 
\Big( \frac{k^2}{\eta} \| z \|_{L^2(\dOmega)} + \| z \|_{H^1(\dOmega)} \Big).
\label{ntdest}\end{equation}
\end{proposition}


\begin{proof}
Let $f$ be the eigenfunction of $\ntd(k)$ with eigenvalue $\beta$. 
Let $\tilde z \in L^2(\dOmega)$ be the projection of $z$ into the subspace orthogonal to $f$. Let $v$ be a Helmholtz solution at frequency $k$ such that $(v - \beta \xn \partial_n v )|_{\dOmega} = \tilde z$; then $(\Theta(k) - \beta)^{-1} z = \Pi^\perp_f \xn \partial_n v |_{\dOmega}$, where $\Pi^\perp_f$ is the orthogonal projection onto the subspace orthogonal to $f$. We need to estimate the norm of $\xn \partial_n v |_{\dOmega}$ relative to the norm of $z$; for this it suffices to assume $\tilde z = z$. 

To estimate the size of $\xn \partial_n v |_{\dOmega}$, we expand $v$ in eigenfunctions $\phi_j^\beta$ as in Proposition~\ref{qowbeta}.  We write $E^\beta =  k^2$. By assumption, there is a $j = j_*$ such that $E_{j_*}^\beta = E^\beta$; the corresponding eigenfunction $\phi_{j_*}^\beta$ satisfies $\xn \partial_n \phi_{j_*}^\beta = f$. 

 Let 
$$
v = \sum a_j \phi^\beta_j.
$$
We may assume that $a_{j_*} = 0$, as $\Pi^\perp_f (\xn \partial_n \phi_{j_*}^\beta) = 0$. Then
$$\begin{gathered}
a_j = \ang{v, \phi^\beta_j}_{L^2(\Omega)} = \frac1{E^\beta - E^\beta_j} \Big( -\ang{\Delta v, \phi^\beta_j} + \ang{v, \Delta \phi^\beta_j} \Big) \\
= \frac1{E^\beta - E^\beta_j}\int_{\dOmega} v \partial_n \phi^\beta_j - \partial_n v \phi^\beta_j \\
= \frac1{E^\beta - E^\beta_j}\int_{\dOmega} v \partial_n \phi^\beta_j - \partial_n v \beta \xn \partial_n\phi^\beta_j \\
= \frac1{E^\beta - E^\beta_j} \int_{\dOmega} z \partial_n \phi^\beta_j.
\end{gathered}$$
Therefore, 
$$
v = \sum_{j \neq j_*} \Big( \frac1{E^\beta - E^\beta_j} \int_{\dOmega} z \partial_n \phi^\beta_j \Big) \phi^\beta_j.
$$
If we try to take the normal derivative term by term in this series and sum, unfortunately we end up with a divergent series (even when the quasi-orthogonality of the boundary values $\xn \partial_n \phi_j^\beta$ is taken into account). To avoid this problem we write $v'$ as the solution to 
$$
\Delta v' = 0 \text{ in } \Omega, \quad (v - \beta \xn \partial_n v )|_{\dOmega} = z;
$$
there is a unique solution to this problem due to the negativity of $\beta$. Then we can express $v' = \sum a'_j \phi^\beta_j$, where from a similar computation to above
$$
a'_j = -\frac1{E^\beta_j} \int_{\dOmega} z \partial_n \phi^\beta_j \quad (\implies a'_{j_*} = 0.)
$$
Therefore, $v - v'$ has an expansion 
\begin{equation}\begin{gathered}
v - v' = \sum_{j \neq j_*} \Big( \Big( \frac1{E^\beta - E^\beta_j} + \frac1{E_j^\beta} \Big) \int_{\dOmega} z \partial_n \phi^\beta_j \Big) \phi^\beta_j \\
= \sum_{j \neq j_*}  \Big( \frac{E^\beta}{E^\beta_j(E^\beta - E^\beta_j)}  \int_{\dOmega} z \partial_n \phi^\beta_j \Big) \phi^\beta_j .
\end{gathered}\label{v-v'}\end{equation}
which has improved convergence properties as the denominator is now $\sim (E_j^\beta)^{-2}$ instead of $(E_j^\beta)^{-1}$, as $j \to \infty$. From this we see that 
$$
\xn \partial_n (v - v') = \sum_{j \neq j_*}  \Big( \frac{E^\beta}{E^\beta_j(E^\beta - E^\beta_j)}  \int_{\dOmega} z \partial_n \phi^\beta_j \Big) \xn \partial_n \phi^\beta_j.
$$

Now we use Proposition~\ref{qowbeta}, proceed as in Section 4 of \cite{bnds} and show that the operator 
$$
z \mapsto \sum_{j \neq j_*}  \Big( \frac{E^\beta}{E^\beta_j(E^\beta - E^\beta_j)}  \int_{\dOmega} z \partial_n \phi^\beta_j \Big) \xn \partial_n \phi^\beta_j
$$
has operator norm at most $C + C k^2 /d(E^\beta, \sigma^*)$, where $d(E^\beta, \sigma^*)$ denotes the distance from $E^\beta$ to the nearest point of the spectrum on $\Delta$ with boundary condition \eqref{betabc}. By Assumption~\ref{ASC} and Lemma~\ref{lem:evalgaps}, this distance is at least $C \eta$, so we get an estimate on the operator norm of $C k^2/\eta$. Therefore $v - v'$ has norm at most $C k^2/\eta \| z \|_{L^2(\dOmega)}$. 

To treat the term $v'$, notice that that $\xn \partial_n v'$ is $(\ntd(0) - \beta)^{-1}z$; we will estimate the operator norm of $(\ntd(0) - \beta)^{-1}$. The operator $\ntd(0)$ is a positive operator, since
$$
\ang{v, \xn \partial_n v} = \int_{\dOmega} v \partial_n v = \int_{\Omega} v \Delta v + |\nabla v|^2 \geq 0.
$$
Therefore, as $\beta$ is negative, the norm of $ (\ntd(0) - \beta)^{-1} z $ is no bigger than that of $\ntd(0)^{-1} z$.  The operator $\ntd(0)^{-1}$, which is nothing other than the multiplication opertor $\xn$ composed with the Dirichlet-to-Neumann map at energy zero, is a pseudodifferential operator of order $1$, and therefore 
$$\| v' \|_{L^2(\dOmega)} \leq C \| z \|_{H^1(\dOmega)}.
$$
This concludes the proof of Proposition~\ref{prop:gen-inverse}.
\end{proof}

\begin{remark}\label{sqrtk} In fact, although the above analysis shows that $\| (\ntd(k) - \beta)^{-1}z \| $ can indeed be as large as $Ck^2/d(E^\beta, \sigma^*)$ times $\| z \|$, this only happens in a `worst-case scenario' in which $z$ is a multiple of $\partial_n \phi_j^\beta$ where $E_j^\beta$ is the eigenvalue of $\ntd(k)$ closest to (but distinct from) $E^\beta$ (or, more precisely, a linear combination of an $O(1)$ number of the $\xn \partial_n \phi_j^\beta$ with closest eigenvalues). In a more `typical-case scenario', the coefficients $a_j$, for $|E_j^\beta - E^\beta| \leq \sqrt{E^\beta}$, would be $\sim k^{1/2}$ in magnitude --- this can be seen from Proposition~\ref{qowbeta} and the arguments of \cite{bnds}, which show that the $\xn \partial_n \phi_j^\beta$ have norm $\sim k$ and are approximately orthogonal for $|E_j^\beta - E^\beta| \leq \sqrt{E^\beta}$. On the other hand, for $|E_j^\beta - E^\beta| \geq \sqrt{E^\beta}$, we gain a power of $\sqrt{E^\beta} = k$ in the denominator of \eqref{v-v'}. 
This suggests that, typically, we would have
$\| (\ntd(k) - \beta)^{-1}z \| $ no bigger than a constant times $k^{3/2} \| z \|/d(E^\beta, \sigma^*)$. This would imply that in formula \eqref{12derivs} for the second derivative of $f$ at $\beta = 0$, and given Assumption~\ref{ASC} with $\eta \sim 1$, the $\ntd(k)^{-1}(mf)$ term is usually smaller by a factor $\sim k^{-1/2}$  than the principal terms, even though it is   of the same order in the worst-case scenario. This is a heuristic justification for dropping this term in the quadratic estimator \eqref{fhatimpsecondorder}. 
\end{remark}

\bibliographystyle{abbrv} 
\bibliography{alex}

\end{document}